\documentclass[10pt,a4paper]{amsart}
\usepackage[english]{babel}
\usepackage[T1]{fontenc}
\usepackage[latin1]{inputenc}
\usepackage{graphicx}
\usepackage{amsmath,amssymb,amsthm,mathrsfs,,txfonts,ifthen,xr}
\usepackage{soul}
\usepackage{array, multirow}
\usepackage{stmaryrd}
\usepackage{mathrsfs}
\usepackage{color}
\usepackage{hyperref}
\usepackage{enumerate}
\usepackage[all]{xy}

\theoremstyle{plain}
\newtheorem{theo}{Theorem}[section]
\newtheorem{lemme}[theo]{Lemma}
\newtheorem{prop}[theo]{Proposition}
\newtheorem{coro}[theo]{Corollary}

\theoremstyle{definition}
\newtheorem{defn}[theo]{Definition}

\newtheorem{nota}[theo]{Notation} 

\newtheorem{conj}[theo]{Conjecture}

\theoremstyle{remark}
\newtheorem{rema}[theo]{Remark}

\def\Sp{{\rm Sp}}
\def\J{{\rm J}}
\def\U{{\rm U}}
\def\M{{\rm M}}
\def\det{{\rm det}}
\def\O{{\rm O}}
\def\GL{{\rm GL}}	
\def\End{{\rm End}}

\def\tr{{\rm tr}}
\def\G{{\rm G}}
\def\C{{\rm C}}
\def\N{{\rm N}}
\def\S{{\rm S}}
\def\exp{{\rm exp}}
\def\dim{{\rm dim}}
\def\diag{{\rm diag}}

\def\Im{{\rm Im}}
\def\Op{{\rm Op}}
\def\Hom{{\rm Hom}}

\def\supp{{\rm supp}}
\def\Ad{{\rm Ad}}
\def\rk{{\rm rk}}
\def\H{{\rm H}}
\def\Re{{\rm Re}}
\def\reg{{\rm reg}}
\def\Z{{\rm Z}}
\def\ch{{\rm ch}}
\def\sign{{\rm sign}}
\def\K{{\rm K}}

\def\P{{\rm P}}
\def\Id{{\rm Id}}
\def\Mat{{\rm Mat}}
\def\E{{\rm E}}
\def\A{{\rm A}}
\def\T{{\rm T}}
\def\Chc{{\rm Chc}}
\def\Ker{{\rm Ker}}
\def\D{{\rm D}}
\def\L{{\rm L}}
\def\ch{{\rm ch}}
\def\sh{{\rm sh}}
\def\Lie{{\rm Lie}}
\def\i{{\rm i}}
\def\V{{\rm V}}
\def\W{{\rm W}}
\def\M{{\rm M}}

\def\X{{\rm X}}
\def\Y{{\rm Y}}

\def\st{{\rm st}}
\def\min{{\rm min}}

\def\d{{\rm d}}
\def\Eig{{\rm Eig}}
\def\min{{\rm min}}
\def\b{{\rm b}}
\def\sgn{{\rm sgn}}
\def\B{{\rm B}}
\def\loc{{\rm loc}}

\allowdisplaybreaks

\title{Characters of irreducible unitary representations of $\U(n, n+1)$ via double lifting from $\U(1)$}

\author{Allan Merino}
\address{ Department of Mathematics \\ National University of Singapore \\ Block S17 \\ 10, Lower Kent Ridge Road \\ Singapore 119076 \\ Republic of Singapore}
\email{matafm@nus.edu.sg}
\keywords{Howe correspondence, Characters, Cauchy--Harish-Chandra integral, Orbital Integrals}
\subjclass[2010]{Primary: 22E45; Secondary: 22E46, 22E30.}
\date{}

\begin{document}

\maketitle

\begin{abstract}

In this paper, we obtained character formulas of irreducible unitary representations of $\U(n, n+1)$ by using Howe's correspondence and the Cauchy--Harish-Chandra integral. The representations of $\U(n, n+1)$ we are dealing with are obtained from a double lifting of a representation of $\U(1)$ via the dual pairs $(\U(1), \U(1, 1))$ and $(\U(1, 1), \U(n, n+1))$.

\end{abstract}

\section{Introduction}

For a finite dimensional representation $(\Pi, \V)$ of a group $\G$, one can associate a function $\Theta_{\Pi}$ on $\G$ given by
\begin{equation*}
\Theta_{\Pi}: \G \ni g \to \tr(\Pi(g)) \in \mathbb{C}\,.
\end{equation*}
The function $\Theta_{\Pi}$ is the character of the representation $(\Pi, \V)$, and it determines entirely the representation. Obviously, if we remove the assumption that $\V$ is finite dimensional, the map $\Theta_{\Pi}$ does not necessarily makes sense in general. In \cite[Section~5]{HAR3}, Harish-Chandra extended the concept of character for a particular class of representations of a real reductive Lie group. More precisely, he proved that for a quasi-simple representation $(\Pi, \mathscr{H})$ (see \cite[Section~10]{HAR1}) of a real reductive Lie group $\G$, the operator $\Pi(\Psi), \Psi \in \mathscr{C}^{\infty}_{c}(\G)$, given by 
\begin{equation*}
\Pi(\Psi) = \displaystyle\int_{\G}\Psi(g)\Pi(g)dg\,,
\end{equation*}
is a trace class operator and the corresponding map 
\begin{equation*}
\Theta_{\Pi}: \mathscr{C}^{\infty}_{c}(\G) \ni \Psi \to \tr(\Pi(\Psi)) \in \mathbb{C}
\end{equation*}
is a distribution; $\Theta_{\Pi}$ is usually called the distribution character of $\Pi$. Moreover, Harish-Chandra proved (see \cite[Theorem~2]{HAR4}) that there exists a locally integrable function $\Theta_{\Pi}$ on $\G$, analytic on $\G^{\reg}$ (where $\G^{\reg}$ is the set of regular points of $\G$, see \cite[Section~3]{HAR4}), such that
\begin{equation*}
\Theta_{\Pi}(\Psi) = \displaystyle\int_{\G} \Theta_{\Pi}(g)\Psi(g)dg, \qquad \left(\Psi \in \mathscr{C}^{\infty}_{c}(\G)\right)\,.
\end{equation*}
The locally integrable function $\Theta_{\Pi}$ is the character of $\Pi$. In few cases, an explicit value of $\Theta_{\Pi}$ is well-known:
\begin{enumerate}
\item $\G$ compact (H. Weyl),
\item $(\Pi, \mathscr{H})$ a discrete series representation:
\begin{enumerate}
\item Harish-Chandra (see \cite{HAR5}) established a formula for $\Theta_{\Pi}$ on the compact Cartan subgroup $\T$ of $\G$,
\item Hecht (see \cite{HEC}) determined the value of $\Theta_{\Pi}$ on every Cartan subgroup of $\G$ for holomorphic discrete series representation,
\end{enumerate}
\item $(\Pi, \mathscr{H})$ irreducible unitary highest weight module (Enright, \cite[Corollary~2.3]{ENR}, see also \cite{MER3}).
\end{enumerate}

\noindent The goal of this paper is to explain how to use Howe's correspondence and the Cauchy--Harish-Chandra integral introduced by T. Przebinda to get explicit values of characters for some particular irreducible unitary non-highest weight modules of $\U(n, n+1)$ starting from a representation of $\U(1)$. Our method is as follows. Let $(\G, \G') = (\U(1), \U(p, q)), p,q \geq 1$,  be a dual pair in $\Sp(2(p+q), \mathbb{R})$, $\widetilde{\Sp}(2(p+q), \mathbb{R})$ be the corresponding metaplectic group (see Equation \eqref{MetaplecticGroup}), $\omega^{1}_{p, q}$ be the metaplectic representation of $\widetilde{\Sp}(2(p+q), \mathbb{R})$ (see Theorem \ref{MetaplecticRepresentation}), $\widetilde{\G}$ and $\widetilde{\G'}$ be the preimages of $\G$ and $\G'$ in $\widetilde{\Sp}(2(p+q), \mathbb{R})$ and $\Pi$ be a representation of $\U(1)$. We denote by $\theta^{1}_{p, q}$ the map coming from Howe's duality theorem (see Equation \eqref{Maptheta})
\begin{equation*}
\theta^{1}_{p, q}: \mathscr{R}(\widetilde{\U}(1), \omega^{1}_{p, q}) \to \mathscr{R}(\widetilde{\U}(p, q), \omega^{1}_{p, q})\,,
\end{equation*} 
where $\mathscr{R}(\widetilde{\U}(1), \omega^{1}_{p, q})$ and $\mathscr{R}(\widetilde{\U}(p, q), \omega^{1}_{p, q})$ are defined in Notation \ref{NotationsRGOmega}. By assumption on $p$ and $q$, $\Pi' :=\theta^{1}_{p, q}(\Pi)$ is a non-zero irreducible unitary highest weight module of $\widetilde{\U}(p, q)$. In Appendix \ref{ComputationsU(1)}, we computed the value of the character $\Theta_{\Pi'}$ of $\Pi'$ on every Cartan subgroup of $\widetilde{\U}(p, q)$. A similar result was obtained in \cite{MER} for $p = q = 1$.

\noindent We now consider the dual pair $(\G', \G_{n}) = (\U(p, q), \U(n, n+1))$ in $\Sp(2(p+q)(2n+1), \mathbb{R})$. As before, we denote by $\widetilde{\Sp}(2(p+q)(2n+1), \mathbb{R})$ the metaplectic group of $\Sp(2(p+q)(2n+1), \mathbb{R})$, $\omega^{p, q}_{n, n+1}$ the metaplectic representation of $\widetilde{\Sp}(2(p+q)(2n+1), \mathbb{R})$ and 
\begin{equation*}
\theta^{p, q}_{n, n+1}: \mathscr{R}(\widetilde{\U}(p, q), \omega^{p, q}_{n, n+1}) \to \mathscr{R}(\widetilde{\U}(n, n+1), \omega^{p, q}_{n, n+1})
\end{equation*} 
the map obtained from Howe's correspondence (see Equation \eqref{Maptheta}). Using a result of Kudla (see \cite{KUD}), it follows that $\Pi^{n} := \theta^{p, q}_{n, n+1}(\Pi') \neq 0$ for every $n \geq 2$, i.e. $\Pi^{n} \in \mathscr{R}(\widetilde{\U}(n, n+1), \omega^{p, q}_{n, n+1})$. Note that by using \cite{LOKE}, it follows that $\Pi^{n}_{1} = \Pi^{n}$, where $\Pi^{n}_{1}$ is usually called the "big theta" and is defined in Section \ref{SectionHoweDuality}.

\noindent As explained in \cite{TOM5} (see also Remark \ref{RemarkConjecture}), by using that $\Pi'$ is unitary, we get that if $p+q \leq n$, the distribution character $\Theta_{\Pi^{n}}$ of $\Pi^{n}$ can be obtained by using the Cauchy--Harish-Chandra integral (see Section \ref{SectionCauchyHarishChandra}). More precisely, because $\Pi^{n}_{1} = \Pi^{n}$, we get from Equation \eqref{ChcStar}, Theorem \ref{TheoremTheta} and Equation \eqref{ChcNonCompact} that:
\begin{equation*}
\Theta_{\Pi^{n}}(\Psi) = \sum\limits_{i=1}^{\min(p, q)+1} \displaystyle\int_{\widetilde{\H}'_{i}} \overline{\Theta_{\Pi'}(\tilde{h}'_{i})} \left|\det(\Id - \Ad(\tilde{h}'_{i})^{-1})_{\mathfrak{g}'/\mathfrak{h}'_{i}}\right|^{\frac{1}{2}} \Chc_{\tilde{h}'_{i}}(\Psi) d\tilde{h}'_{i}\,, \qquad \left(\Psi \in \mathscr{C}^{\infty}_{c}(\widetilde{\G_{n}})\right),
\end{equation*}
where $\H'_{1}, \ldots, \H'_{\min(p, q)+1}$ is a maximal set of non-conjugate Cartan subgroups of $\G'$ and $\Chc_{\tilde{h}'_{i}}$ is a family of distributions on $\widetilde{\G_{n}}$ parametrized by regular elements on the different Cartan subgroups of $\widetilde{\G'}$ as recalled in Section \ref{SectionCauchyHarishChandra}.

\noindent In this paper, we compute explicitly the value of $\Theta_{\Pi^{n}}$ on every Cartan subgroup of $\widetilde{\G_{n}}$ for $p = q = 1$. We keep the notations of Appendix \ref{AppendixCartanUnitary} and parametrize the $n+1$ Cartan subgroups of $\G_{n}$ by subsets $\S_{0} = \left\{\emptyset\right\}, \S_{1}, \ldots, \S_{n}$ of strongly orthogonal imaginary roots of $\mathfrak{g}_{n}$ (see also \cite[Section~2]{SCH}) and let $\H_{n}(\S_{0}), \ldots, \H_{n}(\S_{n})$ be the corresponding Cartan subgroups of $\G_{n}$ and $\H_{n, \S_{0}}, \ldots, \H_{n, \S_{n}}$ be the diagonal subgroups of $\GL(2n+1, \mathbb{C})$ given, for $0 \leq t \leq n$, by $\H_{n, \S_{t}} = c(\S_{t})^{-1}\H_{n}(\S_{t})c(\S_{t})$, where $c(\S_{t})$ is the Cayley transform defined in Equation \eqref{CayleyCS}.

\noindent In Theorem \ref{IntegralFormulaCharacter}, we explain how to go from the distribution character $\Theta_{\Pi^{n}}$ to the locally integrable function $\Theta_{\Pi^{n}}$ by using results of Bernon and Przebinda from \cite{TOM4} and \cite{TOM2} (see also Section \ref{SectionUnitary}). In Theorem \ref{PropositionFinalFormulaThetaPin}, we computed the value of $\Theta_{\Pi^{n}}$ of the different Cartan subgroups of $\G_{n}$ and get for every $t \in [|0, n|]$,
\begin{equation*}
\Theta_{\Pi^{n}}(c(\S_{t})\check{p}(\check{h})c(\S_{t})^{-1}) = \pm\C \begin{cases}
-\widetilde{\A} \sum\limits_{\underset{i \in J(h) \cup \B_{t}}{j \in K(h) \cup \A_{t}}} h^{n}_{i}h^{n+m}_{j} \Omega_{i, j}(h) + \delta_{t, 0}\B e^{-(m+1)\sgn(X_{2n+1})X_{2n+1}}\Sigma(h) & \text{ if } m \geq 1 \\ \widetilde{\A} \sum\limits_{\underset{i \neq j}{i, j \in J(h) \cup \B_{t}}} h^{n}_{i}h^{n}_{j} \Omega_{i, j}(h) + \delta_{t, 0}\B e^{-(m+1)\sgn(X_{2n+1})X_{2n+1}}\Sigma(h) & \text{ if } m = 0 \\ \widetilde{\A}\sum\limits_{\underset{j \in J(h) \cup \B_{t}}{i \in K(h) \cup \A_{t}}} h^{n+m}_{i}h^{n}_{j} \Omega_{i, j}(h) + \delta_{t, 0}\B e^{(m-1)\sgn(X_{2n+1})X_{2n+1}}\Sigma(h)& \text{ if } m \leq -1
\end{cases}
\end{equation*}
where $\C$ is a constant, $m$ is the highest weight of $\Pi$ (see Notation \ref{NotationA1}), $\check{\H}_{n, \S_{t}}$ is a double cover of $\H_{n, \S_{t}}$ defined in Equation \eqref{DoubleCoverHCRho}, $\check{p}: \check{\H}_{n, \S_{t}} \to \widetilde{\H}_{n, \S_{t}}$ is defined in Section \ref{SectionUnitary}, $h$ is the element of $\H_{n, \S_{t}}$ given, as in Equation \eqref{EquationH(S)}, by
\begin{equation*}
h = (h_{1}, \ldots, h_{2n+1}) =  \diag(e^{iX_{1} - X_{2n+1}}, \ldots, e^{iX_{t} - X_{2n+2-t}}, e^{iX_{t+1}}, \ldots, e^{iX_{2n+1-t}}, e^{iX_{t} + X_{2n+2-t}}, \ldots, e^{iX_{1} + X_{2n+1}})\,, \qquad X_{j} \in \mathbb{R}\,,
\end{equation*}
where $\Omega_{i, j}, 1 \leq i \neq j \leq 2n+1$, and $\Sigma$ are the functions on $\H^{\reg}_{n, \S_{t}}$ are given by
\begin{equation*}
\Omega_{i, j}(h) = \cfrac{\prod\limits_{\underset{k \neq i, j}{k=1}}^{2n+1} h_{k}}{\prod\limits_{\underset{k \neq i}{k=1}}^{2n+1} (h_{i} - h_{k}) \prod\limits_{\underset{l \neq i, j}{l=1}}^{2n+1} (h_{j} - h_{l})}\,, \qquad \Sigma(h) = \cfrac{\sgn(X_{2n+1}) e^{imX_{1}}\left|e^{(2n-2)X_{2n+1}}\right|\left(1 - e^{-2X_{2n+1}}\right)}{\left|\prod\limits_{k=2}^{2n} \left(1 - h_{1}h^{-1}_{k}\right) \prod\limits_{k=2}^{2n} \left(1- h_{k}h^{-1}_{2n+1}\right)\right|\left|1 - e^{-2X_{2n+1}}\right|^{2}}\,,
\end{equation*}
$K(h)$ and $J(h)$ are the subsets of $\{1, \ldots, t\}$ defined by
\begin{equation*}
J(h) = \{j \in \{1, \ldots, t\}, \sgn(X_{2n+2-j}) = 1\}\,, \qquad K(h) = \{j \in \{1, \ldots, t\}, \sgn(X_{2n+2-j}) = -1\}\,,
\end{equation*}
$\A_{t}$ and $\B_{t}$ are the subsets of $\left\{1, \ldots, 2n+1\right\}$ given by 
\begin{equation*}
\A_{t} = \left\{t+1, \ldots, n\right\}, \qquad \B_{t} = \left\{n+1, \ldots, 2n+1-t\right\}\,,
\end{equation*}
$\sgn$ is the sign-function on $\mathbb{R}^{*}$ given by
\begin{equation*}
\sgn(X) = \begin{cases} 1 & \text{ if } X > 0 \\ -1 & \text{ if } X < 0 \end{cases}, \qquad (X \in \mathbb{R}^{*})\,,
\end{equation*}
and $\widetilde{\A}$ and $\B$ are constants defined in Theorems \ref{IntegralFormulaCharacter} and \ref{PropositionFinalFormulaThetaPin}. As explained in Lemma \ref{Lemma1502}, the denominator
\begin{equation*}
\prod\limits_{k=2}^{2n} \left(1 - h_{1}h^{-1}_{k}\right) \prod\limits_{k=2}^{2n} \left(1- h_{k}h^{-1}_{2n+1}\right)
\end{equation*}
is real and its sign is constant on every Weyl chamber.

\bigskip

\noindent Note that the representation $\Pi^{n}$ is irreducible and unitary but not an highest weight module. In particular, its character $\Theta_{\Pi^{n}}$ cannot be obtained by using Enright's formula (see \cite[Corollary~2.3]{ENR}) or \cite{MER3}. 

\noindent Our formula for the character is Weyl denominator free.The method we used in this paper gives a general procedure to give characters of non-highest weight representations by starting from an highest weight representation of a compact group. In particular, proving Conjecture \ref{ConjectureHC} could make our method more general, by removing the assumption of stable range for the second lifting.

\bigskip

\noindent The paper is organised as follows. In Section \ref{SectionMetaplectic}, we recalled a construction of the metaplectic representation given by Aubert and Przebinda in \cite{TOM6}. The goal is to define the embedding $\T$ of the metaplectic group into the set of tempered distributions on the symplectic space $\W$ (see Equation \eqref{MapTEmbedding}) which is crucial in the construction of the Cauchy--Harish-Chandra integral. After recalling Harish-Chandra's character theory and Howe's correspondence in Section \ref{SectionHoweDuality}, we define in Section \ref{SectionCauchyHarishChandra} the Cauchy--Harish-Chandra integral and explain a conjecture of Przebinda on the transfer of characters in the theta correspondence (see Conjecture \ref{ConjectureHC}). In Section \ref{SectionUnitary}, we summarized the results of \cite{TOM4} and \cite{TOM2} on how to compute the Cauchy--Harish-Chandra integral on the different Cartan subgroups, and we adapt the results to unitary groups, which are the one we consider in this paper, and make the computations for $\Theta_{\Pi^{n}}$ in Section \ref{SectionUnitary2} (see Theorem \ref{IntegralFormulaCharacter} and Proposition \ref{PropositionFinalFormulaThetaPin}). The document contains three appendices: in Appendix \ref{ComputationsU(1)}, we make computations for $\Pi'$ for the dual pair $(\G, \G') = (\U(1), \U(p, q))$ on every Cartan subgroup of $\G'$, in Appendix \ref{AppendixCartanUnitary}, we recall how to parametrize the Cartan subgroups of $\U(p, q)$ by using strongly orthogonal roots (see also \cite[Section~2]{SCH}) and in Appendix \ref{AppendixEpsilon}, we define a character $\varepsilon$ appearing in the formulas for the Cauchy--Harish-Chandra integrals and proved a lemma for $\varepsilon$ useful in the proof of Lemma \ref{Lemma0702}.

\bigskip

\noindent \textbf{Acknowledgements: } I would like to thank Tomasz Przebinda for the useful discussions during the preparation of this paper. This research was supported by the MOE-NUS AcRF Tier 1 grants R-146-000-261-114 and R-146-000-302-114.

\section{Metaplectic Representation}

\label{SectionMetaplectic}

Let $\chi$ be the character of $\mathbb{R}$ given by $\chi(t) = e^{2i\pi t}$ and let $\W$ be a finite dimensional vector over $\mathbb{R}$ endowed with a non-degenerate, skew-symmetric, bilinear form $\langle\cdot, \cdot\rangle$.  We denote by $\Sp(\W)$ the corresponding group of isometries, i.e.
\begin{equation*}
\Sp(\W) = \left\{g \in \GL(\W), \langle g(w_{1}), g(w_{2})\rangle = \langle w_{1}, w_{2} \rangle, (\forall w_{1}, w_{2} \in \W)\right\}\,,
\end{equation*}
and by $\mathfrak{sp}(\W)$ its Lie algebra given by:
\begin{equation*}
\mathfrak{sp}(\W) = \left\{X \in \End(\W), \langle X(w_{1}), w_{2}\rangle + \langle w_{1}, X(w_{2})\rangle = 0, (\forall w_{1}, w_{2} \in \W)\right\}\,.
\end{equation*}
We first start by recalling the construction of the metaplectic group $\widetilde{\Sp}(\W)$: it is a connected two-fold cover of $\Sp(\W)$. We use the formalism of \cite{TOM6}. Let $\J$ be a compatible positive complex structure on $\W$, i.e. an element of the Lie algebra $\mathfrak{sp}(\W)$ satisfying $\J^{2} = -\Id_{\W}$ and such that the symmetric form $\langle \J\cdot, \cdot\rangle$ is positive definite. For every element $g \in \Sp(W)$, we denote by $\J_{g}$ the element of $\End(\W)$ given by $\J_{g} = \J^{-1}(g-1)$. One can easily check that the restriction of $\J_{g}$ to $\J_{g}(\W)$ is invertible and let $\widetilde{\Sp}(\W)$ be the subset of $\Sp(\W) \times \mathbb{C}^{\times}$ defined by
\begin{equation}
\widetilde{\Sp}(\W) = \left\{\tilde{g} = (g, \xi) \in \Sp(\W) \times \mathbb{C}^{\times}, \xi^{2} = i^{\dim_{\mathbb{R}}(\W)}\det(\J_{g})^{-1}_{\J_{g}(\W)}\right\}\,,
\label{MetaplecticGroup}
\end{equation}
where $\det(\J_{g})^{-1}_{\J_{g}(\W)}$ denotes the determinant of the endomorphism $\J_{g}$ restricted to $\J_{g}(\W)$. On $\widetilde{\Sp}(\W)$, we define a multiplication by;
\begin{equation*}
(g_{1}, \xi_{1})(g_{2}, \xi_{2}) = (g_{1}g_{2}, \xi_{1}\xi_{2}\C(g_{1}, g_{2})) \qquad \left(g_{1}, g_{2} \in \Sp(\W), \xi_{1}, \xi_{2} \in \mathbb{C}^{\times}\right)\,,
\end{equation*}
where $\C: \Sp(\W) \times \Sp(\W) \to \mathbb{C}^{\times}$ is a cocycle defined in \cite[Proposition~4.13]{TOM6}. Let $\Theta$ be the map defined by:
\begin{equation*}
\Theta: \widetilde{\Sp}(\W) \ni \tilde{g} = (g, \xi) \to \xi \in \mathbb{C}^{\times}\,.
\end{equation*}
One can check easily that $\widetilde{\Sp}(\W)$ is a connected two-fold cover of $\Sp(\W)$, where the covering map $\pi: \widetilde{\Sp}(\W) \to \Sp(\W)$ is given by $\pi((g, \xi)) = g$.

\noindent For every $g \in \End(\W)$, we denote by $c(g)$ the Cayley transform of $g$ defined by:
\begin{equation*}
c(g): (g-1)\W \ni (g-1)w \to (g+1)w + \Ker(g-1) \in \W/\Ker(g-1)\,.
\end{equation*}

\noindent We denote by $\S(\W)$ the Schwartz space of $\W$ and by $\S^{*}(\W)$ the corresponding space of tempered distributions. We define the map $t: \widetilde{\Sp}(\W) \to \S^{*}(\W)$ by $t(g) = \chi_{c(g)}\mu_{(g-1)\W}$, where $\chi_{c(g)}$ is the function on $(g-1)\W$ given by
\begin{equation*}
\chi_{c(g)}(w): (g-1)\W \to \chi\left(\frac{1}{4}\langle c(g)w, w\rangle\right)\,, \qquad \left(w \in (g-1)\W\right)\,,
\end{equation*}
and $\mu_{(g-1)\W}$ is the Lebesgue measure on $(g-1)\W$ such that the volume of the unit cube with respect to the bilinear form $\langle\J\cdot, \cdot\rangle$ is $1$. More precisely,
\begin{equation*}
t(g)\phi = \displaystyle\int_{(g-1)\W} \chi_{c(g)}(w) \phi(w) d\mu_{(g-1)\W}(w)\,, \qquad \left(\phi \in \S(\W)\right)\,.
\end{equation*}
We define the map $\T: \widetilde{\Sp}(\W) \to \S^{*}(\W)$ given by 
\begin{equation}
\T(\tilde{g}) = \Theta(\tilde{g}) t(g) \qquad \left(\tilde{g} \in \widetilde{\Sp}(\W), g = \pi(\tilde{g})\right)\,.
\label{MapTEmbedding}
\end{equation}

\begin{rema}

Let $\tilde{g}_{1}, \tilde{g}_{2} \in \widetilde{\Sp}(\W)$. The question of the relation between the distributions $\T(\tilde{g}_{1}), \T(\tilde{g}_{2})$ and $\T(\tilde{g}_{1}\tilde{g}_{2})$ arises naturally. In order to explain this link, we need to recall the notion of twisted convolution. 

\noindent For two functions $\phi_{1}, \phi_{2} \in \S(\W)$, we define $\phi_{1} \natural \phi_{2}$ the function on $\W$ given by
\begin{equation*}
\phi_{1} \natural \phi_{2}(w) = \displaystyle\int_{\W} \phi_{1}(u) \phi_{2}(w-u)\chi\left(\frac{1}{2}\langle u, w\rangle\right) d\mu_{\W}(u)\,, \qquad \left(w \in \W\right)\,.
\end{equation*}
One can easily check that $\phi_{1} \natural \phi_{2} \in \S(\W)$. We extend $\natural$ to some tempered distributions on $\W$. In fact, for every $g \in \Sp(W)$, the twisted convolution
\begin{equation*}
t(g) \natural \phi(w) = \displaystyle\int_{(g-1)\W} \chi_{c(g)}(u)\phi(w-u)\chi\left(\frac{1}{2}\langle u, w\rangle\right) d\mu_{\W}(u)\,, \qquad \left(w \in \W, \phi \in \S(\W)\right)\,,
\end{equation*}
is still a Schwartz function and the map:
\begin{equation*}
\S(\W) \ni \phi \to t(g) \natural \phi \in \S(\W)
\end{equation*}
is well-defined and continuous (see \cite[Proposition~4.11]{TOM6}). Similarly, $\T(\tilde{g}) \natural \phi \in \S(\W)$ for every $\tilde{g} \in \widetilde{\Sp}(\W)$ and $\phi \in \S(\W)$. In particular, it makes sense to consider $\T(\tilde{g}_{1}) \natural \left(\T(\tilde{g}_{2}) \natural \phi\right)$ for every $\tilde{g}_{1}, \tilde{g}_{2} \in \widetilde{\Sp}(\W)$ and $\phi \in \S(\W)$ and one can prove that $\T(\tilde{g}_{1}) \natural (\T(\tilde{g}_{2}) \natural \phi) = \T(\tilde{g}_{1}\tilde{g}_{2}) \natural \phi$.

\label{RemarkTwisted}

\end{rema}

\noindent Let $\W = \X \oplus \Y$ be a complete polarization of the space $\W$ and we denote by $dx$, $dy$ the Lebesgue measures on $\X$ ans $\Y$ respectively such that $d\mu_{\W} = dxdy$. Using the Weyl transform $\mathscr{K}$, we have a natural isomorphism between the spaces $\S(\W)$ and $\S(\X \times \X)$ given by
\begin{equation*}
\mathscr{K}: \S(\W) \ni \phi \to \mathscr{K}(\phi)(x_{1}, x_{2}) = \displaystyle\int_{\Y} \phi(x_{1} - x_{2} + y) \chi\left(\frac{1}{2}\langle y, x_{1} + x_{2}\rangle\right) dy \in \S(\X \times \X)\,,
\end{equation*}
which extends to an isomorphism on the corresponding spaces of distributions. Similarly, every tempered distribution on $\X \times \X$ can be identified to an element of $\Hom(\S(\X), \S^{*}(\X))$ using the Schwartz Kernel Theorem (see \cite[Equation~146]{TOM6}). The corresponding isomorphism will be denoted by $\Op$ and let $\omega: \widetilde{\Sp}(\W) \to \Hom(\S(\X), \S^{*}(\X))$ be the map given by
\begin{equation*}
\omega = \Op \circ \mathscr{K} \circ \T\,.
\end{equation*}
As proved in \cite[Section~4]{TOM6}, we get that for every $\tilde{g} \in \widetilde{\Sp}(\W)$ and $v \in \S(\X)$, $\omega(\tilde{g})v \in \S(\X)$ and that $\omega(\tilde{g}\tilde{h}) = \omega(\tilde{g}) \circ \omega(\tilde{h})$ for every $\tilde{g}, \tilde{h} \in \widetilde{\Sp}(\W)$. The operator $\omega(\tilde{g}) \in \Hom(\S(\X), \S(\X))$ can be extended to $\L^{2}(\X)$ by
\begin{equation*}
\omega(\tilde{g})\phi = \lim\limits_{\underset{v \in \S(\X)}{||\phi - v||_{2} \to 0}} \omega(\tilde{g})v\,, \qquad \left(\phi \in \L^{2}(\X)\right)\,.
\end{equation*}

\begin{theo}

For every $\tilde{g} \in \widetilde{\S}(\W)$ and $\phi \in \L^{2}(\X)$, the map
\begin{equation*}
\widetilde{\Sp}(\W) \ni \tilde{g} \to \omega(\tilde{g})\phi \in \L^{2}(\X)\,,
\end{equation*}
is well-defined and continuous. Moreover, $\omega(\tilde{g}) \in \U(\L^{2}(\X))$, i.e. $\omega$ is a faithful unitary representation of $\widetilde{\Sp}(\W)$, and for every $\Psi \in \mathscr{C}^{\infty}_{c}(\widetilde{\Sp}(\W))$, we get:
\begin{equation*}
\displaystyle\int_{\widetilde{\Sp}(\W)} \Theta(\tilde{g})\Psi(\tilde{g})d\tilde{g} = \tr \displaystyle\int_{\widetilde{\Sp}(\W)} \Psi(\tilde{g})\omega(\tilde{g}) d\tilde{g}\,,
\end{equation*}
where $d\tilde{g}$ is a Haar measure on $\widetilde{\Sp}(\W)$.
\label{MetaplecticRepresentation}

\end{theo}

\begin{rema}

\begin{enumerate}
\item Let $\Sp^{c}(\W)$ be the subset of $\Sp(\W)$ given by $\Sp^{c}(\W) = \left\{g \in \Sp(\W), \det(g-1) \neq 0\right\}$. This is the domain of the Cayley transform. We will denote by $\widetilde{\Sp}^{c}(\W)$ the preimage of $\Sp^{c}(\W)$ in $\widetilde{\Sp}(\W)$. 

\noindent For every $g \in \Sp^{c}(\W)$, $c(g) \in \mathfrak{sp}(\W)$. We denote by $\mathfrak{sp}^{c}(\W)$ the subspace of $\mathfrak{sp}(\W)$ defined by $c(\Sp^{c}(\W))$. Obviously, $c^{2}(g) = g$. It defines a bijective map $c: \mathfrak{sp}^{c}(\W) \to \Sp^{c}(\W)$. Fix an element $\widetilde{-1}$ in $\pi^{-1}(\{-1\})$. In particular, there exists a unique map $\tilde{c}: \mathfrak{sp}^{c}(\W) \to \widetilde{\Sp}^{c}(\W)$ such that $c = \pi \circ \tilde{c}$ and $\tilde{c}(0) = \widetilde{-1}$.

\noindent Moreover, for every $\Psi \in \mathscr{\Sp}(\W)$ whose support is included in $\widetilde{\Sp}^{c}(\W)$, we get:
\begin{equation*}
\displaystyle\int_{\widetilde{\Sp}(\W)} \Psi(\tilde{g}) d\tilde{g} = \displaystyle\int_{\mathfrak{sp}(\W)} \Psi(\tilde{c}(X)) j_{\mathfrak{sp}(\W)}(X) dX\,,
\end{equation*}
where $j_{\mathfrak{sp}(\W)}(X) = |\det(1-X)|^{r}$, where $r = \frac{2\dim_{\mathbb{R}}(\mathfrak{sp}(\W))}{\dim_{\mathbb{R}}(\W)}$ (see \cite[Section~3]{TOM7}).
\item For every $\Psi \in \mathscr{C}^{\infty}_{c}(\widetilde{\Sp}(\W))$, we can consider the following distribution on $\W$
\begin{equation*}
\displaystyle\int_{\widetilde{\Sp}(\W)} \Psi(\tilde{g}) \T(\tilde{g}) d\tilde{g}\,.
\end{equation*}
This distribution is in fact given by a Schwartz function. Indeed, let's first assume that the support of $\Psi$ in included in $\widetilde{\Sp}^{c}(\W)$. For every $\phi \in \S(\W)$, we get:
\begin{eqnarray*}
& & \left(\displaystyle\int_{\widetilde{\Sp}(W)} \Psi(\tilde{g}) \T(\tilde{g}) d\tilde{g}\right)(\phi) = \displaystyle\int_{\widetilde{\Sp}(W)} \Psi(\tilde{g}) \T(\tilde{g})\phi d\tilde{g} \\
& = & \displaystyle\int_{\widetilde{\Sp}(\W)} \Psi(\tilde{g}) \displaystyle\int_{\W} \Theta(\tilde{g}) \chi_{c(g)}(w) \phi(w) dw d\tilde{g} = \displaystyle\int_{\mathfrak{sp}(\W)} \displaystyle\int_{\W} \Psi(\tilde{c}(X)) \Theta(\tilde{c}(X)) \chi_{X}(w) j_{\mathfrak{sp}(\W)}(X)\phi(w) dw dX \\
& = & \displaystyle\int_{\W} \left(\displaystyle\int_{\mathfrak{sp}(\W)} \Phi_{\Psi}(X) \chi(\tau_{\mathfrak{sp}(\W)}(w)(X)) dX\right)\phi(w) dw = \displaystyle\int_{\W} \mathscr{F}(\Phi_{\Psi}) \circ \tau_{\mathfrak{sp}(\W)}(w) \phi(w) dw\,,
\end{eqnarray*}
where $\Phi_{\Psi}(X) = \Psi(\tilde{c}(X)) \Theta(\tilde{c}(X)) j_{\mathfrak{sp}(\W)}(X), X \in \mathfrak{sp}(\W)$, is smooth and compactly supported function on $\mathfrak{sp}(\W)$ such that $\supp(\Phi_{\Psi}) \subseteq \mathfrak{sp}^{c}(\W)$, $\mathscr{F}(\Phi_{\Psi})$ is the Fourier transform of $\Phi_{\Psi}$ and $\tau_{\mathfrak{sp}(\W)}: \W \to \mathfrak{sp}(\W)^{*}$ is the moment map defined by $\tau_{\mathfrak{sp}(\W)}(w)(X) = \langle X(w), w\rangle, w \in \W, X \in \mathfrak{sp}(\W)$. In particular, $\mathscr{F}(\Phi_{\Psi}) \circ \tau_{\mathfrak{sp}(\W)}$ is a Schwarz function on $W$.

\noindent We can remove the assumption on the support of $\Psi$ by using the previous result. Indeed, the Zariski topology on $\Sp(W)$ is noetherian. In particular, there exists $g_{1}, \ldots, g_{m} \in \Sp(W)^{c}$ such that
\begin{equation*}
\widetilde{\Sp}(W) = \bigoplus\limits_{i=1}^{m} \tilde{g_{i}} \widetilde{\Sp}^{c}(W)\,.
\end{equation*}
We can find functions $\Psi_{1}, \ldots, \Psi_{m} \in \mathscr{C}^{\infty}_{c}(\widetilde{\Sp}^{c}(W))$ such that for every $\tilde{g} \in \widetilde{\Sp}(W)$,
\begin{equation*}
1 = \sum\limits_{i=1}^{m} \Psi_{i}(\tilde{g_{i}}^{-1} \tilde{g})\,.
\end{equation*}
Then, for every $\Psi \in \mathscr{C}^{\infty}_{c}(\widetilde{\Sp}(W))$, we get:
\begin{eqnarray*}
& & \displaystyle\int_{\widetilde{\Sp}(W)} \Psi(\tilde{g}) \T(\tilde{g}) d\tilde{g} = \sum\limits_{i=1}^{m} \displaystyle\int_{\widetilde{\Sp}(W)} \Psi_{i}(\tilde{g_{i}}^{-1} \tilde{g}))\Psi(\tilde{g}) \T(\tilde{g}) d\tilde{g} \\
& = & \sum\limits_{i=1}^{m} \displaystyle\int_{\widetilde{\Sp}(W)} \Psi_{i}(\tilde{g})\Psi(\tilde{g_{i}}\tilde{g}) \T(\tilde{g_{i}}\tilde{g})d\tilde{g} = \sum\limits_{i=1}^{m} \T(\tilde{g_{i}}) \natural \left(\displaystyle\int_{\widetilde{\Sp}(W)} \Psi_{i}(\tilde{g})\Psi(\tilde{g_{i}}\tilde{g}) \T(\tilde{g}) d\tilde{g}\right)\,.
\end{eqnarray*}
The result follows from Remark \ref{RemarkTwisted}.
\end{enumerate}
\label{RemarkCayleyCovering}
\end{rema}

\section{Character Theory and Howe's correspondence}

\label{SectionHoweDuality}

Let $\G$ be a real connected reductive Lie group, $\mathfrak{g} = \Lie(\G)$ its Lie algebra and $\mathfrak{g}_{\mathbb{C}} = \mathfrak{g} \otimes_{\mathbb{R}} \mathbb{C}$ its complexification. We denote by $\mathscr{U}(\mathfrak{g}_{\mathbb{C}})$ the enveloping algebra of $\mathfrak{g}$ (see \cite[Chapter~3.1]{KNA}), $\Z(\mathscr{U}(\mathfrak{g}_{\mathbb{C}}))$ its center and $\D(\G)$ the set of differential operators on $\G$ and by $\D_{\G}(\G)$ the set of left-invariant differential operators on $\G$. As explained in \cite[Chapter~2]{HEL2}, $\D_{\G}(\G)$ is isomorphic to $\mathscr{U}(\mathfrak{g}_{\mathbb{C}})$. Let $\D^{\G}_{\G}(\G)$ be the set of bi-invariant differential operators on $\G$ (which is isomorphic to $\Z(\mathscr{U}(\mathfrak{g}_{\mathbb{C}}))$, see \cite{HEL2}), $\mathscr{D}'(\G)$ be the set of distributions on $\G$ and $\mathscr{D}'(\G)^{\G}$ the set of $\G$-invariant distributions.

\begin{defn}

We say that $T \in \mathscr{D}'(\G)$ is an eigendistribution if there exists $\chi_{T}: \D^{\G}_{\G}(\G) \to \mathbb{C}$ an homomorphism of algebras such that $\D(T) = \chi_{T}(\D)T$ for every $\D \in \D^{\G}_{\G}(\G)$.

\end{defn}

\noindent We will denote by $\Eig(\G)$ the set of eigendistributions on $\G$.

\begin{theo}[Harish-Chandra, \cite{HAR3}]

For every $\G$-invariant eigendistribution $T$ on $\G$, there exists a locally integrable function $f_{T}$ on $\G$, analytic on $\G^{\reg}$, such that $T = T_{f_{T}}$, i.e. for every function $\Psi \in \mathscr{C}^{\infty}_{c}(\G)$, 
\begin{equation*}
T(\Psi) = \displaystyle\int_{\G} f_{T}(g)\Psi(g)dg\,,
\end{equation*}
where $dg$ is a Haar measure on $\G$.

\label{TheoremHCEigen}

\end{theo}

\noindent Let $(\Pi, \mathscr{H})$ be an irreducible quasi-simple representation (see \cite[Section~10]{HAR1}). As explained in \cite{HAR3}, the map
\begin{equation*}
\Theta_{\Pi}: \mathscr{C}^{\infty}_{c}(\G) \ni \Psi \to \tr(\Pi(\Psi)) \in \mathbb{C}
\end{equation*}
is well-defined and is a distribution (in the sense of Laurent Schwartz). In particular, by assumption of $\Pi$, it follows from Theorem \ref{TheoremHCEigen} that there exists $\Theta_{\Pi} \in \mathscr{L}^{1}_{\loc}(\G)$ such that 
\begin{equation*}
\Theta_{\Pi}(\Psi) = \displaystyle\int_{\G} \Theta_{\Pi}(g)\Psi(g) dg
\end{equation*}
for every $\Psi \in \mathscr{C}^{\infty}_{c}(\G)$. The function $\Theta_{\Pi}$ is called the character of $\Pi$.

\bigskip

\noindent We now recall Howe's duality theorem and how it can be studied through characters. Let $\W$ be a finite dimensional vector space over $\mathbb{R}$ endowed with a non-degenerate, skew-symmetric, bilinear form $\langle \cdot, \cdot\rangle$. As in Section \ref{SectionMetaplectic}, we denote by $\Sp(\W)$ the group of isometries of $(\W, \langle \cdot, \cdot\rangle)$, by $(\widetilde{\Sp}(W), \pi)$ the metaplectic cover of $\Sp(\W)$ (see Equation \eqref{MetaplecticGroup}), by $(\omega, \mathscr{H})$ the corresponding Weil representation (see Theorem \ref{MetaplecticRepresentation}) and by $(\omega^{\infty}, \mathscr{H}^{\infty})$ the corresponding smooth representation (see \cite[Chapter~0]{VOG}).

\noindent A dual pair in $\Sp(\W)$ is a pair of subgroups $(\G, \G')$ of $\Sp(\W)$ which are mutually centralizer in $\Sp(W)$. The dual pair is called irreducible if we cannot find any orthogonal decomposition $\W = \W_{1} \oplus \W_{2}$ where both spaces $\W_{1}$ and $\W_{2}$ are $\G\cdot\G'$-invariant, and called reductive if the actions of $\G$ and $\G'$ on $\W$ are both reductive. The set of irreducible reductive dual pairs in $\Sp(\W)$ had been classified by Howe in \cite{HOW5}. 

\begin{rema}

In this paper, we will focus our attention on a dual pair consisting of two unitary groups. More precisely, let $\V$ and $\V'$ be two complex vector spaces endowed with an hermitian form $\left(\cdot, \cdot\right)$ and skew-hermitian form $\left(\cdot, \cdot\right)'$ respectively. We denote by $\U(\V)$ and $\U(\V')$ the corresponding group of isometries and by $\W$ the complex vector space given by $\W = \V \otimes_{\mathbb{C}} \V'$. The space $\W$ can naturally be seen as a real vector space, and to avoid any confusion, we will denote by $\W_{\mathbb{R}}$ the corresponding real vector space. The skew-hermitian form $\b = \left(\cdot, \cdot\right) \otimes \left(\cdot, \cdot\right)'$ on $\W$ defines a skew-symmetric form $\langle\cdot, \cdot\rangle$ on $\W_{\mathbb{R}}$ by $\langle\cdot, \cdot\rangle = \Im(\b)$. In particular, $(\U(\V), \U(\V'))$ is a dual pair in $\Sp(\W_{\mathbb{R}}, \langle\cdot, \cdot\rangle)$.

\noindent If we denote by $(p, q)$ and $(r, s)$ the signatures of $\left(\cdot, \cdot\right)$ and $\left(\cdot, \cdot\right)'$ respectively, we get that $(\U(p, q), \U(r, s))$ form a dual pair in $\Sp(2(p+q)(r+s), \mathbb{R})$.

\end{rema}

\begin{nota}

For a subgroup $\H$ of $\Sp(\W)$, we denote by $\widetilde{\H} = \pi^{-1}(\H)$ the preimage of $\H$ in $\widetilde{\Sp}(\W)$ and let $\mathscr{R}(\widetilde{\H}, \omega)$ be the set of equivalence classes of irreducible admissible representations of $\tilde{\G}$ which are infinitesimally equivalent to a quotient of $\mathscr{H}^{\infty}$ by a closed $\omega^{\infty}(\widetilde{\H})$-invariant subspace. 

\label{NotationsRGOmega}

\end{nota}

\begin{theo}[R. Howe, \cite{HOW1}]

For every reductive dual pair $(\G, \G')$ of $\Sp(\W)$, we get a bijection between $\mathscr{R}(\widetilde{\G}, \omega)$ and $\mathscr{R}(\widetilde{\G'}, \omega)$, whose graph is $\mathscr{R}(\widetilde{\G}\cdot\widetilde{\G'}, \omega)$.

\label{TheoremHowe}

\end{theo}

\noindent More precisely, if $\Pi \in \mathscr{R}(\widetilde{\G}, \omega)$, we denote by $\N(\Pi)$ the intersection of all the closed $\widetilde{\G}$-invariant subspaces $\mathscr{N}$ such that $\Pi \approx \mathscr{H}^{\infty} / \mathscr{N}$. Then, the space $\mathscr{H}(\Pi) = \mathscr{H}^{\infty} / \N(\Pi)$ is a $\widetilde{\G} \cdot \widetilde{\G'}$-module; more precisely, $\mathscr{H}(\Pi) = \Pi \otimes \Pi'_{1}$, where $\Pi'_{1}$ is a $\widetilde{\G'}$-module, not irreducible in general, but Howe's duality theorem says that there exists a unique irreducible quotient $\Pi'$ of $\Pi'_{1}$ with $\Pi' \in \mathscr{R}(\widetilde{\G'}, \omega)$ and $\Pi \otimes \Pi' \in \mathscr{R}(\widetilde{\G}\cdot\widetilde{\G'}, \omega)$.

\noindent We will denote by 
\begin{equation}
\theta: \mathscr{R}(\widetilde{\G}, \omega) \to \mathscr{R}(\widetilde{\G'}, \omega)
\label{Maptheta}
\end{equation}
the corresponding bijection.

\begin{rema}

If $\G$ is compact, the situation turns out to be slightly easier. The action of $\widetilde{\G}$ on $\mathscr{H}^{\infty}$ can be decomposed as 
\begin{equation*}
\mathscr{H}^{\infty} = \bigoplus\limits_{(\Pi, \mathscr{H}_{\Pi}) \in \mathscr{R}(\widetilde{\G}, \omega)} \mathscr{H}(\Pi)\,,
\end{equation*}
where $\mathscr{H}(\Pi)$ is the closure of $\left\{T(\mathscr{H}_{\Pi}), T \in \Hom_{\widetilde{\G}}(\mathscr{H}_{\Pi}, \mathscr{H}^{\infty}) \neq \{0\}\right\}$ and $\mathscr{R}(\widetilde{\G}, \omega)$ is the set of representations $(\Pi, \V_{\Pi})$ of $\widetilde{\G}$ such that $\Hom_{\widetilde{\G}}(\V_{\Pi}, \omega^{\infty}) \neq \{0\}$. Obviously, $\widetilde{\G'}$ acts on $\mathscr{H}(\Pi)$ and we get that $\mathscr{H}(\Pi) = \Pi \otimes \Pi'$ where $\Pi'$ is an irreducible unitary representation of $\widetilde{\G'}$.

\label{HoweCompact}

\end{rema} 

\section{Cauchy Harish-Chandra integral and transfer of invariant eigendistributions}

\label{SectionCauchyHarishChandra}

We start this section by recalling the construction of the Cauchy--Harish-Chandra integral introduced in Section \cite[Section~2]{TOM1}..

\noindent Let $(\G, \G')$ be an irreducible reductive dual pair in $\Sp(\W)$ and $\T: \widetilde{\Sp}(\W) \to \S^{*}(\W)$ the map defined in Equation \eqref{MapTEmbedding}. Let $\H_{1}, \ldots, \H_{n}$ be a maximal set of non-conjugate Cartan subgroups of $\G$ and let $\H_{i} = \T_{i}\A_{i}$ the decomposition of $\H_{i}$ as in \cite[Section~2.3.6]{WAL}, where $\T_{i}$ is maximal compact in $\H_{i}$. For every $1 \leq i \leq n$, we denote by $\A'_{i}$ the subgroup of $\Sp(\W)$ given by $\A'_{i} = \C_{\Sp(\W)}(\A_{i})$ and let $\A''_{i} = \C_{\Sp(\W)}(\A'_{i})$. As recalled in \cite[Section~1]{TOM1}, there exists an open and dense subset $\W_{\A^{''}_{i}}$, which is $\A^{''}_{i}$-invariant and such that $\A^{''}_{i} \setminus \W_{\A^{''}_{i}}$ is a manifold, endowed with a measure $\underline{dw}$ such that for every $\phi \in \mathscr{C}^{\infty}_{c}(\W)$ such that $\supp(\phi) \subseteq \W_{\A^{''}_{i}}$, 
\begin{equation*}
\displaystyle\int_{\W_{\A^{''}_{i}}} \phi(w)dw = \displaystyle\int_{\A^{''}_{i} \setminus \W_{\A^{''}_{i}}} \displaystyle\int_{\A^{''}_{i}} \phi(aw)da\underline{dw}\,.
\end{equation*}
For every $\Psi \in \mathscr{C}^{\infty}_{c}(\widetilde{\A^{'}_{i}})$, we denote by $\Chc(\Psi)$ the following integral:
\begin{equation*}
\Chc(\Psi) = \displaystyle\int_{\A^{''}_{i} \setminus \W_{\A^{''}_{i}}} \T(\Psi)(w) \underline{dw}\,.
\end{equation*}

\noindent According to Remark \ref{RemarkCayleyCovering}, the previous integral is well-defined and as proved in \cite[Lemma~2.9]{TOM1}, the corresponding map $\Chc: \mathscr{C}^{\infty}_{c}(\widetilde{\A^{'}_{i}}) \to \mathbb{C}$ defines a distribution on $\widetilde{\A^{'}_{i}}$.

\begin{rema}

We say few worlds about the dual pair $(\A'_{i}, \A''_{i})$ and the space $\W_{\A''_{i}}$. Let $\V_{0, i}$ be the subspace of $\V$ on which $\A_{i}$ acts trivially and $\V_{1, i} = \V^{\perp}_{0, i}$. The restriction of $\left(\cdot,\cdot\right)$ to $\V_{1, i}$ is non-degenerate and even dimensional. In particular, there exists a complete polarization of $\V_{1, i}$ of the form $\V_{1, i} = \X_{i} \oplus \Y_{i}$, where both spaces $\X_{i}$ and $\Y_{i}$ are $\H_{i}$-invariant.

\noindent By looking at the action of $\A_{i}$ on $\V_{1, i}$, we get:
\begin{equation*}
\X_{i} = \X^{1}_{i} \oplus \ldots \oplus \X^{k}_{i}\,, \qquad \Y_{i} = \Y^{1}_{i} \oplus \ldots \oplus \Y^{k}_{i}\,,
\end{equation*}
where all the spaces $\X^{j}_{i}, 1 \leq i \leq n, 1 \leq j \leq k$, are $\A_{i}$-invariant and mutually non-equivalent. In particular,
\begin{equation*}
\W = \Hom(\V, \V') = \Hom(\V_{0, i}, \V') \oplus \Hom(\V_{1, i}, \V') = \Hom(\V_{0, i}, \V') \oplus \bigoplus\limits_{j=1}^{k} \left(\Hom(\X^{j}_{i}, \V') \oplus \Hom(\Y^{j}_{i}, \V')\right)\,.
\end{equation*}
To simplify the notations, we denote by $\W^{i}_{j}$ the subspace of $\W$ given by $\Hom(\X^{j}_{i}, \V') \oplus \Hom(\Y^{j}_{i}, \V')$ and $\W_{0, i} = \Hom(\V_{0, i}, \V')$. One can easily check that:
\begin{equation*}
\A'_{i} = \Sp(\W_{0, i}) \times \GL(\Hom(\X^{1}_{i}, \V')_{\mathbb{R}}) \times \ldots \times \GL(\Hom(\X^{k}_{i}, \V')_{\mathbb{R}})
\end{equation*}
and 
\begin{equation*}
\A''_{i} = \O(1) \times \GL(1, \mathbb{R}) \times \ldots \times \GL(1, \mathbb{R})\,.
\end{equation*}
Moreover,
\begin{equation*}
\W_{\A''_{i}} = (\W_{0, i} \setminus \{0\}) \times \widetilde{\W}_{1, i} \times \ldots \times \widetilde{\W}_{n, i}\,,
\end{equation*}
where $\widetilde{\W}_{j, i} = \left\{(x, y) \in \Hom(\X^{j}_{i}, \V') \oplus \Hom(\Y^{j}_{i}, \V'), x \neq 0, y \neq 0\right\}$, $1 \leq j \leq k$.

\label{RemarkFixingNotations}

\end{rema}

\noindent For every $\tilde{h}_{i} \in \widetilde{\H_{i}}$, we denote by $\tau_{\tilde{h}_{i}}$ the map:
\begin{equation*}
\tau_{\tilde{h}_{i}}: \widetilde{\G'} \ni \tilde{g}' \to \tilde{h}\tilde{g}' \in \widetilde{\A^{'}_{i}}\,.
\end{equation*}
As proved in \cite{TOM1}, for every $\tilde{h}_{i} \in \widetilde{\H_{i}}^{\reg}$, the pull-back $\tau^{*}_{\tilde{h}_{i}}(\Chc)$ of $\Chc$ via $\tau_{\tilde{h}_{i}}$ (see \cite[Theorem~8.2.4]{HOR}) is a well-defined distribution on $\widetilde{\G'}$. 

\noindent For every $\tilde{h}_{i} \in \widetilde{\H_{i}}^{\reg}$, we denote by $\Chc_{\tilde{h}_{i}} := \tau^{*}_{\tilde{h}_{i}}(\Chc)$ the corresponding distribution on $\widetilde{\G'}$.

\begin{nota}

\noindent For every reductive group $\G$, we denote by $\mathcal{I}(\G)$ the space of orbital integrals on $\G$ as in \cite[Section~3]{BOU}, endowed with a natural topology defined in \cite[Section~3.3]{BOU}. We denote by $\J_{\G}$ the map $\J_{\G}: \mathscr{C}^{\infty}_{c}(\G) \to \mathscr{C}^{\infty}(\G^{\reg})^{\G}$ given as follows: for every $\gamma \in \G^{\reg}$, there exists a unique, up to conjugation, Cartan subgroups $\H(\gamma)$ of $\G$ such that $\gamma \in \H(\gamma)$, and for every $\Psi \in \mathscr{C}^{\infty}_{c}(\G)$, we define $\J_{\G}(\Psi)(\gamma)$ by:
\begin{equation*}
\J_{\G}(\Psi)(\gamma) = \left|\det(\Id - \Ad(\gamma^{-1}))_{\mathfrak{g}/\mathfrak{h}(\gamma)}\right|^{\frac{1}{2}} \displaystyle\int_{\G/\H(\gamma)} \Psi(g\gamma g^{-1}) \overline{dg}\,.
\end{equation*}

\end{nota}

\noindent As proved in \cite[Theorem~3.2.1]{BOU}, the map:
\begin{equation*}
\J_{\G}: \mathscr{C}^{\infty}_{c}(\G) \to \mathcal{I}(\G)
\end{equation*}
is well-defined and surjective. We denote by $\mathcal{I}(\G)^{*}$ the set of continuous linear forms on $\mathcal{I}(\G)$ and let $\J^{t}_{\G}: \mathcal{I}(\G)^{*} \to \mathscr{D}'(\G)$ be the transpose of $\J_{\G}$. In \cite[Theorem~3.2.1]{BOU}, Bouaziz proved that the map
\begin{equation*}
\J^{t}_{\G}: \mathcal{I}(\G)^{*} \to \mathscr{D}'(\G)^{\G}
\end{equation*}
is bijective.

\bigskip

\noindent We now apply these results to construct a map $\Chc^{*}$, transferring the invariant distributions for a given  dual pair $(\G, \G')$. Let $(\G, \G')$ be an irreducible dual pair in $\Sp(W)$ such that $\rk(\G) \leq \rk(\G')$. For every function $\Psi \in \mathscr{C}^{\infty}_{c}(\widetilde{\G'})$, we denote by $\widetilde{\Chc}(\Psi)$ the $\widetilde{\G}$-invariant function on $\widetilde{\G}^{\reg}$ given by:
\begin{equation*}
\widetilde{\Chc}(\Psi)(\tilde{h}_{i}) = \Chc_{\tilde{h}_{i}}(\Psi)\,, \qquad (\tilde{h}_{i} \in \widetilde{\H_{i}}^{\reg})\,.
\end{equation*}
As proved in \cite{TOM4},  the corresponding map 
\begin{equation*}
\widetilde{\Chc}: \mathscr{C}^{\infty}_{c}(\widetilde{\G'}) \to \mathcal{I}(\widetilde{\G})
\end{equation*}
is well-defined and continuous and factors through $\mathcal{I}(\widetilde{\G'})$, i.e.
\begin{equation*}
\widetilde{\Chc}: \mathcal{I}(\widetilde{\G'}) \to \mathcal{I}(\widetilde{\G})
\end{equation*}
and the corresponding map is continuous. In particular, we get a map
\begin{equation*}
\Chc^{*}: \mathscr{D}'(\widetilde{\G})^{\widetilde{\G}} \ni T \to \J^{t}_{\widetilde{\G'}} \circ \widetilde{\Chc}^{t} \circ (\J^{t}_{\widetilde{\G}})^{-1}(T) \in \mathscr{D}'(\widetilde{\G'})^{\widetilde{\G'}}\,.
\end{equation*}

\begin{theo}

The map $\Chc^{*}$ sends $\Eig(\widetilde{\G})^{\widetilde{\G}}$ into $\Eig(\widetilde{\G'})^{\widetilde{\G'}}$. Moreover, if $\Theta$ is a distribution on $\widetilde{\G}$ given by a locally integrable function $\Theta$ on $\widetilde{\G'}$, we get for every $\Psi \in \mathscr{C}^{\infty}_{c}(\widetilde{\G'})$ that:
\begin{equation}
\Chc^{*}(\Theta)(\Psi) = \sum\limits_{i=1}^{n} \cfrac{1}{|\mathscr{W}(\H_{i})|} \displaystyle\int_{\widetilde{\H_{i}}^{\reg}} \Theta(\tilde{h}_{i})|\det(1-\Ad(\tilde{h}^{-1}_{i}))_{\mathfrak{g}/\mathfrak{h}_{i}}| \Chc(\Psi)(\tilde{h}_{i}) d\tilde{h}_{i}\,,
\label{ChcStar}
\end{equation}
where $\H_{1}, \ldots, \H_{n}$ is a maximal set of non-conjugate Cartan subgroups of $\G$.

\label{TheoremTransferOfEigendistributions}

\end{theo}

\noindent In \cite{TOM1}, T. Przebinda conjectured that the correspondence of characters in the theta correspondence should be obtained via $\Chc^{*}$. More precisely,  

\begin{conj}

Let $\G_{1}$ and $\G'_{1}$ be the Zariski identity components of $\G$ and $\G'$ respectively. Let $\Pi \in \mathscr{R}(\widetilde{\G}, \omega)$ satisfying ${\Theta_{\Pi}}_{|_{\widetilde{\G}/\widetilde{\G_{1}}}} = 0$ if $\G = \O(\V)$, where $\V$ is an even dimensional vector space over $\mathbb{R}$ or $\mathbb{C}$. Then, up to a constant, $\Chc^{*}(\overline{\Theta_{\Pi}}) = \Theta_{\Pi'_{1}}$ on $\widetilde{\G'_{1}}$.

\label{ConjectureHC}

\end{conj}

\begin{rema}

The conjecture is known to be true in few cases:
\begin{enumerate}
\item $\G$ compact,
\item $(\G, \G')$ in the stable range and $\Pi$ a unitary representation of $\widetilde{\G}$ (see \cite{TOM5}),
\item $(\G, \G') = (\U(p, q), \U(r, s))$, with $p+q = r+s$ and $\Pi$ a discrete series representation of $\widetilde{\G}$ (see \cite{MER2}).
\end{enumerate}

\label{RemarkConjecture}

\end{rema}

\bigskip

\section{Explicit formulas of $\Chc$ for unitary groups}

\label{SectionUnitary}

In this section, we quickly explain how to compute explicitly the Cauchy--Harish-Chandra integral on the different Cartan subgroups. Because our papers will only concerns characters of some unitary groups, we will adapt the results of \cite{TOM4} and \cite{TOM2} in this context, but similar results can be obtained for other dual pairs.

\noindent Let $\V = \mathbb{C}^{p+q}$ and $\V' = \mathbb{C}^{r+s}$ be two complex vector spaces endowed with non-degenerate bilinear forms $\left(\cdot, \cdot\right)$ and $\left(\cdot, \cdot\right)'$ respectively, with $\left(\cdot, \cdot\right)$ hermitian and $\left(\cdot, \cdot\right)'$ skew-hermitian, and let $(p, q)$ (resp. $(r, s)$) be the signature of $\left(\cdot, \cdot\right)$ (resp. $\left(\cdot, \cdot\right)'$). We assume that $p+q \leq r+s$. Let $\mathscr{B}_{\V} = \left\{f_{1}, \ldots, f_{n}\right\}$, $n=p+q$ (resp. $\mathscr{B}_{\V'} = \left\{f'_{1}, \ldots, f'_{n'}\right\}$, $n' = r+s$) be a basis of $\V$ (resp. $\V'$) such that $\Mat\left(\left(\cdot, \cdot\right), \mathscr{B}_{\V}\right) = \Id_{p, q}$ (resp. $\Mat\left(\left(\cdot, \cdot\right)', \mathscr{B}_{\V'}\right) = i\Id_{r, s}$).
Let $\G$ and $\G'$ be the corresponding group of isometries, i.e.
\begin{equation*}
\G = \G(\V, \left(\cdot, \cdot\right)) \approx \left\{g \in \GL(n, \mathbb{C}), \overline{g}^{t}\Id_{p, q}g = \Id_{p, q}\right\}\,, \qquad \G' = \G(\V', \left(\cdot, \cdot\right)') \approx \left\{g \in \GL(n', \mathbb{C}), \overline{g}^{t}\Id_{r, s}g = \Id_{r, s}\right\}\,.
\end{equation*}
where $\approx$ is a Lie group isomorphism.

\noindent Let $\K = \U(p) \times \U(q)$ and $\K' = \U(r) \times \U(s)$ be the maximal compact subgroups of $\G$ and $\G'$ respectively and let $\H$ and $\H'$ be the diagonal Cartan subgroups of $\K$ and $\K'$ respectively. By looking at the action of $\H$ on the space $\V$, we get a decomposition of $\V$ of the form:
\begin{equation*}
\V = \V_{1} \oplus \ldots \oplus \V_{n}\,,
\end{equation*}
where the spaces $\V_{a}$ given by $\V_{a} = \mathbb{C}if_{a}$ are irreducible $\H$-modules. We denote by $\J$ the element of $\mathfrak{h}$ such that $\J = i\Id_{\V}$ and let $\J_{j} = i\E_{j, j}$. Similarly, we write
\begin{equation*}
\V' = \V'_{1} \oplus \ldots \oplus \V'_{n'}\,,
\end{equation*}
with $\V'_{b} = \mathbb{C}if'_{b}$, $\J'$ the element of $\mathfrak{h}'$ given by $\J' = i\Id_{\V'}$ and $\J'_{j} = i\E_{j, j}$. Let $\W = \Hom_{\mathbb{C}}(\V', \V)$ endowed with the symplectic form $\langle\cdot, \cdot\rangle$ given by:
\begin{equation*}
\langle w_{1}, w_{2}\rangle = \tr_{\mathbb{C}/\mathbb{R}}(w^{*}_{2}w_{1})\,, \qquad (w_{1}, w_{2} \in \W)\,,
\end{equation*}
where $w^{*}_{2}$ is the element of $\Hom(\V, \V')$ satisfying:
\begin{equation*}
\left(w^{*}_{2}(v'), v\right) = \left(v', w_{2}(v)\right)'\,, \qquad (v \in \V, v' \in \V')\,.
\end{equation*}
The space $\W$ can be seen as a complex vector space by
\begin{equation*}
\i w = \J \circ w\,, \qquad (w \in \W)\,.
\end{equation*}
We define a double cover $\widetilde{\GL}_{\mathbb{C}}(\W)$ of the complex group $\GL_{\mathbb{C}}(\W)$ by:
\begin{equation*}
\widetilde{\GL}_{\mathbb{C}}(\W) = \left\{\tilde{g} = (g, \xi) \in \GL_{\mathbb{C}}(\W) \times \mathbb{C}^{\times}, \xi^{2} = \det(g)\right\}\,.
\end{equation*}
Because $p+q \leq r+s$, we get a natural embedding of $\mathfrak{h}_{\mathbb{C}}$ into $\mathfrak{h}'_{\mathbb{C}}$ and we denote by $\Z' = \G'^{\mathfrak{h}}$ the centralizer of $\mathfrak{h}$ in $\G'$. 

\begin{nota}

We denote by $\Delta$ (resp. $\Delta(\mathfrak{k})$) the root system corresponding to $(\mathfrak{g}_{\mathbb{C}}, \mathfrak{h}_{\mathbb{C}})$ (resp. $(\mathfrak{k}_{\mathbb{C}}, \mathfrak{h}_{\mathbb{C}})$), by $\Psi$ (resp. $\Psi(\mathfrak{k})$) a system of positive roots of $\Delta$ (resp. $\Delta(\mathfrak{k})$) and let $\Phi = - \Psi$ (resp. $\Phi(\mathfrak{k}) = -\Psi(\mathfrak{k})$) the set of negative roots.

\noindent Let $e_{i}$ be the linear form on $\mathfrak{h}_{\mathbb{C}} = \mathbb{C}^{p+q}$ given by $e_{i}(\lambda_{1}, \ldots, \lambda_{p+q}) = \lambda_{i}$. As explained in \cite[Chapter~2]{KNA}, we know that:
\begin{equation*}
\Delta = \left\{\pm(e_{i} - e_{j}), 1 \leq i < j \leq p+q\right\}\,, \qquad \Psi = \left\{e_{i} - e_{j}, 1 \leq i < j \leq p+q\right\}\,,
\end{equation*}
and
\begin{equation*}
\Delta(\mathfrak{k}) = \left\{\pm(e_{i} - e_{j}), 1 \leq i < j \leq p\right\} \cup \left\{\pm(e_{i} - e_{j}), p+1 \leq i < j \leq p+q\right\}\,, \qquad \Psi(\mathfrak{k}) = \Psi \cap \Delta(\mathfrak{k})\,.
\end{equation*}
\noindent We define $\Delta', \Delta'(\mathfrak{k}), \Psi', \Psi'(\mathfrak{k}), \Phi', \Phi'(\mathfrak{k})$ similarly and denote by $e'_{i}, 1 \leq i \leq r+s$ the linear form on $\mathfrak{h}'_{\mathbb{C}} = \mathbb{C}^{r+s}$ given by $e_{i}(\lambda_{1}, \ldots, \lambda_{r+s}) = \lambda_{i}$.

\label{NotationsRootsUnitary}

\end{nota}

\noindent Let $\H'_{\mathbb{C}}$ be the complexification of $\H'$ in $\GL_{\mathbb{C}}(\W)$. In particular, $\H'_{\mathbb{C}}$ is isomorphic to 
\begin{equation*}
\mathfrak{h}'_{\mathbb{C}} / \left\{\sum\limits_{j=1}^{n'}2\pi x_{j}\J_{j}, x_{j} \in \mathbb{Z}\right\}\,.
\end{equation*}
We denote by $\check{\H}'_{\mathbb{C}}$ the connected two-fold cover of $\H'_{\mathbb{C}}$ isomorphic to
\begin{equation}
\mathfrak{h}'_{\mathbb{C}}/\left\{\sum\limits_{j=1}^{n'}2\pi x_{j}\J'_{j}, \sum\limits_{j=1}^{n'}x_{j} \in 2\mathbb{Z}, x_{j} \in \mathbb{Z}\right\}\,.
\label{DoubleCoverHCRho}
\end{equation}
One can easily check that $\rho' = \frac{1}{2} \sum\limits_{\alpha \in \Psi'} \alpha$  is analytic integral on $\check{\H}'_{\mathbb{C}}$. As explained in \cite[Section~2]{TOM1}, we construct a map $\check{p}: \check{\H}'_{\mathbb{C}} \to \widetilde{\H}'_{\mathbb{C}}$ which is bijective (but not an isomorphism of covering of $\H'_{\mathbb{C}}$ in general).
 
\noindent As explained in Appendix \ref{AppendixCartanUnitary}, every Cartan subgroup of $\G'$ can be parametrized by a subset $\S \subseteq \Psi'^{\st}_{n}$ consisting of non-compact strongly otrhogonal roots. We denote by $\H'(\S)$ the corresponding Cartan subgroup and by $\H'_{\S}$ the subgroup of $\H'_{\mathbb{C}}$ as in Appendix \ref{AppendixCartanUnitary}. Let $\S \subseteq \Psi'^{\st}_{n}$ and $\check{\H}'_{\S}$ the preimage of the Cartan subgroup $\H'_{\S}$ in $\check{\H}'_{\mathbb{C}}$ (see Appendix \ref{AppendixCartanUnitary}). For every $\varphi \in \mathscr{C}^{\infty}_{c}(\widetilde{\G'})$, we denote by $\mathscr{H}_{\S}\varphi$ the function of $\check{\H}'_{\S}$ defined by:
\begin{equation*}
\mathscr{H}_{\S}\varphi(\check{h}') = \varepsilon_{\Psi'_{\S, \mathbb{R}}}(\check{h}') \check{h}'^{\frac{1}{2} \sum\limits_{\alpha \in \Psi'} \alpha} \prod\limits_{\alpha \in \Psi'} (1 - \check{h}'^{-\alpha}) \displaystyle\int_{\G' / \H'(\S)} \varphi(g' c(\S) \check{p}(\check{h}') c(\S)^{-1} g'^{-1}) dg'\H'(\S) \qquad (\check{h}' \in \check{\H}'_{\S})\,,
\end{equation*}
where $\Psi'_{\S, \mathbb{R}}$ is the subset of $\Psi'$ consisting of real roots for $\H'_{\S}$ and $\varepsilon_{\Psi'_{\S, \mathbb{R}}}$ is the function defined on $\check{\H}'^{\reg}_{\S}$ by
\begin{equation*}
\varepsilon_{\Psi'_{\S, \mathbb{R}}}(\check{h}') = \sign\left(\prod\limits_{\alpha \in \Psi'_{\S, \mathbb{R}}} (1 - \check{h}'^{-\alpha})\right)\,,
\end{equation*}

\noindent We denote by $\Delta_{\Psi'}(\check{h}')$ the quantity
\begin{equation*}
\Delta_{\Psi'}(\check{h}') = \check{h}'^{\frac{1}{2} \sum\limits_{\alpha \in \Psi'} \alpha} \prod\limits_{\alpha \in \Psi'} (1 - \check{h}'^{-\alpha})\,, \qquad (\check{h}' \in \check{\H}'_{\S})\,, 
\end{equation*}
and by $\Delta_{\Phi'}$ the function on $\check{\H}'^{\reg}_{\S}$ given by $\Delta_{\Phi'}(\check{h}') = \check{h}'^{\frac{1}{2} \sum\limits_{\alpha \in \Phi'} \alpha} \prod\limits_{\alpha \in \Phi'} (1 - \check{h}'^{-\alpha})$.

\begin{rema}

For every $\check{h}' \in \check{\H}'^{\reg}_{\S}$, $\Delta_{\Phi'}(\check{h}')\Delta_{\Psi'}(\check{h}') = \prod\limits_{\alpha \in \Psi'^{+}}(1 - \check{h}'^{\alpha})(1 - \check{h}'^{-\alpha}) = \prod\limits_{\alpha \in \Psi'^{+}}(1 - \check{h}'^{\alpha})\overline{(1 - \check{h}'^{\alpha})}$. We denote by $|\Delta_{\G'}(\check{h}')|^{2} = \Delta_{\Phi'}(\check{h}')\Delta_{\Psi'}(\check{h}')$.

\label{DeltaSquare}

\end{rema}

\begin{prop}[Weyl's Integration Formula]

For every $\varphi \in \mathscr{C}^{\infty}_{c}(\widetilde{\G'})$, we get:
\begin{equation}
\displaystyle\int_{\widetilde{\G'}} \varphi(\tilde{g}') d\tilde{g}' = \sum\limits_{\S \in \Psi'^{\st}_{n}} m_{\S} \displaystyle\int_{\check{\H}'_{\S}} \varepsilon_{\Psi'_{\S, \mathbb{R}}}(\check{h}') \Delta_{\Phi'}(\check{h}') \mathscr{H}_{\S}\varphi(\check{h}')d\check{h}'\,.
\label{Equation3}
\end{equation}
where $m_{\S}$ are complex numbers. Here, the subsets $\S$ of $\Psi'^{\st}_{n}$ are defined up to equivalence (see Remark \ref{RemarkAppendixB}).

\end{prop}

\begin{proof}

See \cite[Section~2,~Page~3830]{TOM4}.

\end{proof}

\begin{rema}

One can easily see that for every $\S \subseteq \Psi'^{\st}_{n}$ and $\varphi \in \mathscr{C}^{\infty}_{c}(\widetilde{\G'})$ such that $\supp(\varphi) \subseteq \widetilde{\G'}\cdot\widetilde{\H'}(\S)^{\reg}$, the Equation \eqref{Equation3} can be written as follow:
\begin{equation*}
\displaystyle\int_{\widetilde{\G'}} \varphi(\tilde{g}') d\tilde{g}' = m_{\S} \displaystyle\int_{\check{\H}'_{\S}} \varepsilon_{\Psi'_{\S, \mathbb{R}}}(\check{h}') \Delta_{\Phi'}(\check{h}') \mathscr{H}_{\S}\varphi(\check{h}')d\check{h}'\,.
\end{equation*}
\label{RemarkWeylIntegration2}
\end{rema}

\noindent For every $\S \subseteq \Psi'^{\st}_{n}$, we denote by $\underline{\S}$ the subset of $\{1, \ldots, r+s\}$ given by $\underline{\S} = \left\{j, \exists \alpha \in \S \text{ such that } \alpha(\J'_{j}) \neq 0\right\}$. Let $\sigma \in \mathscr{S}_{n'}$ and $\S \subseteq \Psi'^{\st}_{n}$, we denote by $\Gamma_{\sigma, \S}$ the subset of $\mathfrak{h}'$ defined as
\begin{equation*}
\Gamma_{\sigma, \S} = \left\{Y \in \mathfrak{h}', \langle Y\cdot, \cdot\rangle_{\sigma\W^{\mathfrak{h}} \cap \sum\limits_{j \notin \underline{\S}} \Hom(\V'_{j}, \V)} > 0\right\}\,,
\end{equation*}
and let $\E_{\sigma, \S} = \widetilde{\exp}(i\Gamma_{\sigma, \S})$ the corresponding subset of $\widetilde{\H}'_{\mathbb{C}}$, where $\widetilde{\exp}$ is a choice of exponential map $\widetilde{\exp}: \mathfrak{h}'_{\mathbb{C}} \to \widetilde{\H}'_{\mathbb{C}}$ obtained by choosing an element $\widetilde{1}$ in $\pi^{-1}\left\{1\right\}$.

\begin{theo}

\label{TheoremTheta}

For every $\check{h} \in \check{\H} = \check{\H}_{\emptyset}$ and $\varphi \in \mathscr{C}(\widetilde{\G'})$, we get:
\begin{equation*}
\det^{\frac{k}{2}}(\check{h})_{\W^{\mathfrak{h}}} \Delta_{\Psi}(\check{h}) \displaystyle\int_{\widetilde{\G'}} \Theta(\check{p}(\check{h})\tilde{g}') \varphi(\tilde{g}')d\tilde{g}' =  
\end{equation*}
\begin{equation*}
\sum\limits_{\sigma \in \mathscr{W}(\H'_{\mathbb{C}})} \sum\limits_{\S \subseteq \Psi'^{\st}_{n}} \M_{\S}(\sigma)\lim\limits_{\underset{r \to 1}{r \in \E_{\sigma, \S}}} \displaystyle\int_{\check{\H}'_{\S}}\cfrac{\det^{-\frac{k}{2}}(\sigma^{-1}(\check{h}'))_{\W^{\mathfrak{h}}} \Delta_{\Phi'(\Z')}(\sigma^{-1}(\check{h}'))}{\det(1-p(h)rp(h'))_{\sigma\W^{\mathfrak{h}}}} \varepsilon_{\Phi'_{\S, \mathbb{R}}}(\check{h}') \mathscr{H}_{\S}(\varphi)(\check{h}') d\check{h}'\,,
\end{equation*}
where $\M_{\S}(\sigma) = \cfrac{(-1)^{u}\varepsilon(\sigma)m_{\S}}{|\mathscr{W}(\Z'_{\mathbb{C}}, \H'_{\mathbb{C}})|}$, $\alpha \in \{0, -1\}$ depends only on the choice of the positive roots $\Psi$ and $\Psi'$, $k \in \{0, -1\}$ is defined by 
\begin{equation*}
k = \begin{cases} -1 & \text{ if } n'-n \in 2\mathbb{Z} \\ 0 & \text{ otherwise } \end{cases}
\end{equation*}
and $\W^{\mathfrak{h}}$ is the set of elements of $\W$ commuting with $\mathfrak{h}$.

\end{theo}

\noindent The theorem \ref{TheoremTheta} tells us how to compute $\Chc_{\tilde{h}}$ for an element $\tilde{h}$ in the compact Cartan $\widetilde{\H} = \widetilde{\H}(\emptyset)$. Using \cite{TOM2}, it follows that the value of $\Chc$ on the other Cartan subgroups can be computed explicitely by knowing how to do it for the compact Cartan.

\begin{rema}

One can easily check that the space $\W^{\mathfrak{h}}$ is given by 
\begin{equation*}
\W^{\mathfrak{h}} = \sum\limits_{i=1}^{n} \Hom(\V'_{i}, \V_{i})\,.
\end{equation*}

\end{rema}

\noindent From now on, we assume that $p \leq q, r \leq s$ and $p \leq r$.

\begin{nota}

For every $t \in [|1, p|]$, we denote by $\S_{t}$ and $\S'_{t}$ the subsets of $\Psi^{\st}_{n}$ and $\Psi'^{\st}_{n}$ respectively given by
\begin{equation*}
\S_{t} = \left\{e_{1} - e_{p+1}, \ldots, e_{t} - e_{p+t}\right\}, \qquad \S'_{t} = \left\{e'_{1} - e'_{r+1}, \ldots, e'_{t} - e'_{r+t}\right\}\,,
\end{equation*}
where the linear forms $e_{k}, e'_{h}$ have been introduced in Notation \ref{NotationsRootsUnitary}.

\end{nota}

\noindent For every $t \in [|0, p|]$, we denote by $\H(\S_{t})$ and $\H'(\S_{t})$ the Cartan subgroups of $\G$ and $\G'$ respectively and let $\H(\S_{t}) = \T(\S_{t})\A(\S_{t})$ (resp. $\H'(\S_{t}) = \T'(\S_{t})\A'(\S_{t}))$ be the decompositions of $\H(\S_{t})$ (resp. $\H'(\S_{t}))$ as in \cite[Section~2.3.6]{WAL}. 

\noindent As in Remark \ref{RemarkFixingNotations}, we denote by $\V_{0, t}$ the subspace of $\V$ on which $\A(\S_{t})$ acts trivially, by $\V_{1, t}$ the orthogonal complement of  $\V_{0, t}$ in $\V$ and by  $\V_{1, t} = \X_{t} \oplus \Y_{t}$ a complete polarization of $\V_{1, t}$. Because $p \leq r$, we have a natural embedding of $\V_{1, t}$ into $\V'$ such that $\X_{t} \oplus \Y_{t}$ is a complete polarization with respect to $\left(\cdot, \cdot\right)'$. We denote by $\U_{t}$ the orthogonal complement of $\V_{1, t}$ in $\V'$; in particular, we get a natural embedding:
\begin{equation*}
\GL(\X_{t}) \times \G(\U_{t}) \subseteq \G' = \U(r, s)\,.
\end{equation*}
We denote by $\T_{1}(\S_{t})$ the maximal subgroup of $\T(\S_{t})$ which acts trivially on $\V_{0, t}$ and let $\T_{2}(\S_{t})$ the subgroup of $\T(\S_{t})$ such that $\T(\S_{t}) = \T_{1}(\S_{t}) \times \T_{2}(\S_{t})$ with $\T_{2}(\S_{t}) \subseteq \G(\V_{0, t})$. In particular,
\begin{equation}
\H(\S_{t}) = \T_{1}(\S_{t}) \times \A(\S_{t}) \times \T_{2}(\S_{t})\,.
\label{DecompositionHSi}
\end{equation}
Similarly, we get a decomposition of $\H'(\S'_{t})$ of the form:
\begin{equation*}
\H'(\S'_{t}) = \T'_{1}(\S'_{t}) \times \A'(\S'_{t}) \times \T'_{2}(\S'_{t})\,.
\end{equation*}

\begin{rema}

One can easily see that for every $0 \leq j < i \leq r$, we get:
\begin{equation*}
\H'(\S'_{i}) = \T'_{1}(\S'_{j}) \times \A'(\S'_{j}) \times \H'(\widetilde{\S}'_{i - j})\,,
\end{equation*}
where $\widetilde{\S}'_{i-j} = \S'_{i} \setminus \S'_{j}$ and $\H'(\widetilde{\S}'_{i - j})$ is the Cartan subgroup of $\U(r-j, s-j)$ whose split part has dimension $i-j$. In particular, 
\begin{equation*}
\H'_{\S'_{i}} = \T'_{1, \S'_{j}} \times \A'_{\S'_{j}} \times \H'_{\widetilde{\S}'_{i-j}}\,.
\end{equation*}

\label{DecompositionHSiHSj}

\end{rema}

\noindent Let $\eta(\S_{t})$ and $\eta'(\S'_{t})$ be the nilpotent Lie subalgebras of $\mathfrak{u}(p, q)$ and $\mathfrak{u}(r, s)$ respectively given by
\begin{equation*}
\eta(\S'_{t}) = \Hom(\X_{t}, \V_{0, t}) \oplus \Hom(\X_{t}, \Y_{t}) \cap \mathfrak{u}(p, q)\,, \qquad \eta'(\S'_{t}) = \Hom(\U_{t}, \X_{t}) \oplus \Hom(\X_{t}, \Y_{t}) \cap \mathfrak{u}(r, s)\,.
\end{equation*}
We will denote by $\W_{0, t}$ the subspace of $\W$ defined by $\Hom(\U_{t}, \V_{0, t})$ and by $\P(\S_{t})$ and $\P'(\S'_{t})$ the parabolic subgroups of $\G$ and $\G'$ respectively whose Levi factors $\L(\S_{t})$ and $\L'(\S'_{t})$ are given by 
\begin{equation*}
\L(\S_{t}) = \GL(\X_{t}) \times \G(\V_{0, t})\,, \qquad \L'(\S'_{t}) = \GL(\X_{t}) \times \G(\U_{t})\,,
\end{equation*}
and by $\N(\S_{t}) := \exp(\eta(\S_{t}))$ and $\N'(\S'_{t}) := \exp(\eta'(\S'_{t}))$ the unipotent radicals of $\P(\S_{t})$ and $\P'(\S'_{t})$ respectively.

\begin{rema}

One can easily check that the forms on $\V_{0, t}$ and $\U_{t}$ have signature $(p-t, q-t)$ and $(r-t, s-t)$ respectively.

\end{rema}

\noindent As proved in \cite[Theorem~0.9]{TOM2}, for every $\tilde{h} = \tilde{t}_{1}\tilde{a}\tilde{t}_{2} \in \widetilde{\H}(\S_{t})^{\reg}$ (using the decomposition of $\H(\S_{t})$ given in Equation \eqref{DecompositionHSi}) and $\varphi \in \mathscr{C}^{\infty}_{c}(\widetilde{\G'})$, we get:
\begin{equation}
|\det(\Ad(\tilde{h}) - \Id)_{\eta(\S_{t})}| \Chc_{\tilde{h}}(\varphi) = 
\label{ChcNonCompact}
\end{equation}
\begin{equation*}
\C_{t}\d_{\S_{t}}(\tilde{h}) \varepsilon(\tilde{t}_{1}\tilde{a}) \displaystyle\int_{\GL(\X_{t})/\T_{1}(\S_{t}) \times \A(\S_{t})} \displaystyle\int_{\widetilde{\G}(\U_{t})} \varepsilon(\tilde{t}_{1} \tilde{a} \tilde{y}) \Chc_{\W_{0, t}}(\tilde{t}_{2}\tilde{y})\d'_{\S_{t}}(g\tilde{t}_{1}\tilde{a}g^{-1}\tilde{y}) \varphi^{\widetilde{\K'}}_{\widetilde{\N'}(\S_{t})}(g\tilde{t}_{1}\tilde{a}g^{-1}\tilde{y}) d\tilde{y} \overline{dg}\,,
\end{equation*}
where $\C_{t}$ is the constant defined given by
\begin{equation}
\C_{t} = \left(\cfrac{2^{t\left(2(p+q+r+s)-4t+1\right)}}{(r+s)^{t}}\right) \left(\cfrac{r+s-2t}{r+s}\right)^{\frac{p+q-2t}{2}} \cfrac{1}{\mu(\K' \cap \L'(\S'_{t})) 2^{t(r+s-2t)}}\,,
\label{ConstantCi}
\end{equation}
$\varepsilon$ is the character defined in \cite[Lemma~6.3]{TOM2}, $\d_{\S_{t}}: \widetilde{\L}(\S_{t}) \to \mathbb{R}$ and $\d'_{\S'_{t}}: \widetilde{\L'}(\S'_{t}) \to \mathbb{R}$ are given by 
\begin{equation*}
\d_{\S_{t}}(\tilde{l}) = |\det(\Ad(\tilde{l})_{\eta(\S_{t})})|^{\frac{1}{2}}, \qquad  \d'_{\S'_{t}}(\tilde{l}') = |\det(\Ad(\tilde{l}')_{\eta'(\S_{t})})|^{\frac{1}{2}}, \qquad \left(\tilde{l} \in \widetilde{\L}(\S_{t}), \tilde{l}' \in \widetilde{\L'}(\S'_{t})\right)\,,
\label{MapGamma}
\end{equation*}
and $\varphi^{\widetilde{\K'}}_{\widetilde{\N'}(\S'_{t})}$ is the Harish-Chandra transform of $\varphi$, i.e. the function on $\widetilde{\L'}(\S'_{t})$ defined by:
\begin{equation*}
\varphi^{\widetilde{\K'}}_{\widetilde{\N'}(\S'_{t})}(\tilde{l}') = \displaystyle\int_{\widetilde{\N'}(\S'_{t})} \displaystyle\int_{\widetilde{\K'}} \varphi(\tilde{k}\tilde{l}'\tilde{n}\tilde{k}^{-1}) d\tilde{k}d\tilde{n}, \qquad \left(\tilde{l}' \in \widetilde{\L'}(\S'_{t})\right)\,.
\end{equation*}

\bigskip

\noindent Let's explain the method we will use in the next section to get character formulas of representations of $\U(n, n+1)$ by using the Cauchy--Harish-Chandra in the stable range.

\noindent Let $(\G, \G') = (\U(p), \U(r, s))$ be a dual pair in $\Sp(2p(r+s), \mathbb{R})$. To avoid any confusions, we will denote by $\omega^{p}_{r, s}$ the metaplectic representation of $\Sp(2p(r+s), \mathbb{R})$ and by $\theta^{p}_{r, s}: \mathscr{R}(\widetilde{\U}(p), \omega^{p}_{r, s}) \to \mathscr{R}(\widetilde{\U}(r, s), \omega^{p}_{r, s})$ the map defined in Equation \eqref{Maptheta}.

\noindent Let $\Pi \in \mathscr{R}(\widetilde{\U}(p), \omega^{p}_{r, s})$. In this case, as explained in Remark \ref{HoweCompact}, we get that $\Pi' = \Pi'_{1}$ and the corresponding representation $\Pi'$ is an irreducible unitary representation of $\widetilde{\U}(r, s)$ (and $\Pi' \in \mathscr{R}(\widetilde{\U}(r, s)), \omega^{p}_{r, s})$. In particular, Theorem \ref{TheoremTheta} tell us how to compute $\Theta_{\Pi'}$ on every Cartan subgroups of $\widetilde{\U}(r, s)$ (see Appendix \ref{ComputationsU(1)} for $p=1$).

\noindent For every $n \geq 0$, we denote by $\G_{n}$ the unitary group corresponding to an hermitian form of signature $(n, n+p)$, i.e. $\G_{n} = \U(n, p+n)$, by $\omega^{r, s}_{n, n+p}$ the metaplectic representation of $\widetilde{\Sp}(\W_{n})$, where $\W_{n} = (\mathbb{C}^{r+s} \otimes_{\mathbb{C}} \mathbb{C}^{2n+p})_{\mathbb{R}}$ and by $\theta^{r, s}_{n, n+p}: \mathscr{R}(\widetilde{\U}(r,s), \omega^{r, s}_{n, n+p}) \to \mathscr{R}(\widetilde{\U}(n,n+p), \omega^{r, s}_{n, n+p})$ the map as in \eqref{Maptheta}.

\noindent Using Kudla's persistence principle (see \cite{KUD}), we know that the representation $\Pi' \in \mathscr{R}(\widetilde{\U}(r, s), \omega^{p}_{r, s})$ satisfies $\theta^{n, p+n}_{r, s}(\Pi') \neq \{0\}$, i.e. $\Pi' \in \mathscr{R}(\widetilde{\U}(r, s), \omega^{r, s}_{n, n+p})$. We denote by $\Pi^{n}_{1}$ the corresponding representation of $\widetilde{\G_{n}}$ as in Section \ref{SectionHoweDuality}, by $\Pi^{n} \in \mathscr{R}(\widetilde{\U}(n, n+p), \omega^{r, s}_{n, n+p})$ it's unique irreducible quotient, and by $\Theta_{\Pi^{n}_{1}}$ and $\Theta_{\Pi^{n}}$ the characters of $\Pi^{n}_{1}$ and $\Pi^{n}$ respectively.

\noindent Using Theorems \ref{TheoremHCEigen} and \ref{TheoremTransferOfEigendistributions}, we know that the $\widetilde{\G_{n}}$-invariant eigendistribution $\Theta'_{n, \Pi'} := \Chc^{*}(\Theta_{\Pi'})$ is given by a locally integrable function $\Theta'_{n, \Pi'}$ on $\widetilde{\G_{n}}$, analytic on $\widetilde{\G_{n}}^{\reg}$. Note that an explicit value of $\Theta'_{n, \Pi'}$ on every Cartan subgroups of $\widetilde{\G_{n}}$ can be obtained using Equation \eqref{ChcStar}, Theorem \ref{TheoremTheta} and Equation \eqref{ChcNonCompact}.

\noindent According to \cite{LOKE}, if $n \geq r+s$, we get that $\Pi^{n}_{1} = \Pi^{n}$ because $\Pi'$ is unitary and by Remark \ref{RemarkConjecture}, it follows that $\Theta'_{n, \Pi'} = \Theta_{\Pi^{n}}$. It is also well-known that for $n \geq 1$, the representations $\Pi^{n}$ are unitary but are not highest weight modules, and in particular, it's character cannot be obtained via Enright's formula (\cite[Corollary~2.3]{ENR}).

\noindent In the next section, we are going to make $\Theta_{\Pi^{n}}$ explicit for $p=r=s=1$. 

\section{Character formulas for some representations of $\U(n, n+1)$}

\label{SectionUnitary2}

We first start with the dual pair $(\G, \G') = (\U(1), \U(1, 1))$ in $\Sp(4, \mathbb{R})$. Because the set of irreducible genuine representations of $\widetilde{\U}(1)$ is isomorphic to $\mathbb{Z}$, the corresponding representation of $\mathscr{R}(\widetilde{\U}(1), \omega^{1}_{1, 1})$ will be denoted by $\Pi_{m}, m \in \mathbb{Z}$ and let $\Pi'_{m}$ be the corresponding representation of $\widetilde{\G'}$. Moreover, as explained in Section \ref{SectionUnitary}, $\Pi^{n} = \theta^{1, 1}_{n, n+1}(\Pi'_{m})$ the lift of $\Pi'_{m}$ on $\widetilde{\G_{n}}$ is non-zero and its character $\Theta_{\Pi^{n}}$ is equal to $\Chc^{*}(\Theta_{\Pi'_{m}})$. In this section, we are going to give an explicit formula for $\Theta_{\Pi^{n}}$ on every Cartan subgroups of $\widetilde{\G_{n}}$.

\begin{rema}

We denote by $\mathfrak{g}, \mathfrak{g}'$ and $\mathfrak{g}_{n}$ the Lie algebras of $\G, \G'$ and $\G_{n}$ respectively. The Lie algebra $\mathfrak{g}'$ is given by
\begin{equation*}
\mathfrak{g}' = \left\{\begin{pmatrix} a & b \\ \bar{b} & d\end{pmatrix}, a, d \in i\mathbb{R}, b \in \mathbb{C}\right\} = \mathbb{R}\begin{pmatrix} i & 0 \\ 0 & i\end{pmatrix} \oplus \mathbb{R}\begin{pmatrix} i & 0 \\ 0 & -i\end{pmatrix} \oplus \mathbb{R}\begin{pmatrix} 0 & 1 \\ 1 & 0 \end{pmatrix} \oplus  \mathbb{R}\begin{pmatrix} 0 & i \\ -i & 0\end{pmatrix}\,.
\end{equation*}
and the two Cartan subgroups of $\G'$, up to conjugation, are of the form{\small
\begin{equation}
\H' = \H'(\S'_{0}) = \left\{\diag(h_{1}, h_{2}), h_{1}, h_{2} \in \U(1)\right\}\,, \qquad \H'(\S'_{1}) = \exp\left(\mathbb{R}\begin{pmatrix} i & 0 \\ 0 & i \end{pmatrix} \oplus \mathbb{R} \begin{pmatrix} 0 & 1 \\ 1 & 0 \end{pmatrix}\right) = \left\{\begin{pmatrix} e^{i\theta}\ch(X) & \sh(X) \\ \sh(X) & e^{i\theta}\ch(X)\end{pmatrix}, \theta, X \in \mathbb{R}\right\}\,,
\label{CartanU(1,1)}
\end{equation}}
where $\S'_{0} = \{\emptyset\}$ and $\S'_{1} = \{e_{1} - e_{2}\}$ (see Appendix \ref{AppendixCartanUnitary}). We denote by $(\V', \left(\cdot, \cdot\right))$ the skew-hermitian form corresponding to $\G'$ and by $(\V_{n}, \left(\cdot, \cdot\right)_{n})$ the hermitian form corresponding to $\G_{n}$. Let $\mathscr{B}_{\V'} = \{f'_{1}, f'_{2}\}$ be a basis of $\V'$ such that $\Mat_{\mathscr{B}'}\left(\cdot, \cdot\right)' = i\Id_{1, 1}$.We have the following complete polarization of $\V'$ 
\begin{equation*}
\V' = \X'_{1} \oplus \Y'_{1}\,, \qquad \X'_{1} = \mathbb{C}(f'_{1} + f'_{2})\,, \qquad \Y'_{1} = \mathbb{C}(f'_{1} - f'_{2})\,,
\end{equation*}
where both $\X'_{1}$ and $\Y'_{1}$ are $\H'_{2}$-invariant. Let $\mathscr{B}_{\V_{n}} = \{f^{n}_{1}, \ldots, f^{n}_{2n+1}\}$ be a basis of $\V_{n}$ such that $\Mat_{\mathscr{B}_{\V_{n}}}\left(\cdot, \cdot\right)_{n} = \Id_{n, n+1}$.

\noindent We consider the embedding of $\V'$ onto $\V_{n}$ sending $f'_{1}$ onto $f^{n}_{1}$ and $f'_{2}$ onto $f^{n}_{2n+1}$. Obviously, $\X'_{1} \oplus \Y'_{1}$ is a complete polarization of $\V' \subseteq \V_{n}$ with respect to $\left(\cdot, \cdot\right)_{n}$. We consider the subspace $\U_{1}$ of $\V_{n}$ given by
\begin{equation*}
\V_{n} = \V' \oplus \U_{1}\,, \qquad \U_{1} = \V'^{\perp}\,,
\end{equation*}
where $\V'^{\perp}$ is the orthogonal complement of $\V'$ in $\V_{n}$ with respect to $\left(\cdot, \cdot\right)_{n}$.

\noindent Let $\G(\U_{1})$ be the group of isometries corresponding to the hermitian space $(\U_{1}, {\left(\cdot, \cdot\right)_{n}}_{|_{\U_{1}}})$. Note that $\G(\U_{1}) \approx \U(n-1, n)$.
\end{rema}

\noindent As explained in Appendix \ref{AppendixCartanUnitary}, for every $0 \leq t \leq n$ and $\S_{t} = \{e_{1} - e_{2n+1}, \ldots, e_{t} - e_{2n+2-t}\}$, we denote by $\H_{n}(\S_{t})$ the Cartan subgroup of $\G_{n}$ whose split part is of dimension $t$ and by $\H_{n, \S_{t}}$ the subgroup of $\H_{n}(\emptyset)_{\mathbb{C}} = \{h = \diag(h_{1}, \ldots, h_{2n+1}), h_{i} \in \mathbb{C}\}$ given by $\H_{n, \S_{t}} = c(\S_{t}) \H_{n}(\S_{t}) c(\S_{t})^{-1}$, where $c(\S_{t})$ is the Cayley transform corresponding to $\S_{t}$ (see Appendix \ref{AppendixCartanUnitary}) We denote by $\P_{n}(\S_{1})$ the parabolic subgroup of $\G_{n}$ whose Levi factor $\L_{n}(\S_{1})$ is given by $\L_{n}(\S_{1}) = \GL(\X_{1}) \times \G(\U_{1})$.

\begin{lemme}

We get $\GL(\X'_{1}) = \H'(\S'_{1})$.

\label{LemmaGL} 

\end{lemme}

\begin{proof}

The Lie algebra of $\GL(\X'_{1})$ is the set of matrices $A = \begin{pmatrix} a & b \\ c & d \end{pmatrix}$ of $\mathfrak{g}'$ such that:
\begin{equation}
A\begin{pmatrix} 1 \\ 1 \end{pmatrix} = \begin{pmatrix} \alpha \\ \alpha\end{pmatrix} \qquad A\begin{pmatrix} 1 \\ -1 \end{pmatrix} = \begin{pmatrix} \beta \\ -\beta \end{pmatrix}, \qquad (\alpha, \beta \in \mathbb{C})\,.
\label{EquationProofCitation}
\end{equation}
We first assume that $A \in \End(\V')$ satisfies the conditions of Equation \eqref{EquationProofCitation}. Then, we get:
\begin{equation*}
\begin{cases} a + b & = \alpha \\ c + d & = \alpha \end{cases} \qquad \begin{cases} a - b & = \beta \\ c - d & = -\beta \end{cases}
\end{equation*}
In particular, $a+b = c + d$ and $a - b = -c + d$. Then, $a = d$ and $b = c$. In particular, if $A \in \mathfrak{g}'$, we get that $a \in i\mathbb{R}$ and $b \in \mathbb{R}$. In particular, 
\begin{equation*}
\GL(\X'_{1}) = \exp\left(\mathbb{R}\begin{pmatrix} i & 0 \\ 0 & i \end{pmatrix} \oplus \mathbb{R} \begin{pmatrix} 0 & 1 \\ 1 & 0 \end{pmatrix}\right) = \exp(\mathfrak{h}'(\S'_{1})) = \H'(\S'_{1})\,.
\end{equation*}

\end{proof}

\noindent In this section, we are going to determine the value of the character $\Theta_{\Pi^{n}}$ on the $n+1$ different Cartan subgroups of $\G_{n}$. 

\begin{nota}

We denote by $\Delta_{n}$ the set of roots corresponding to $(\mathfrak{g}_{n}, \mathfrak{h}_{n})$, where $\H_{n} = \H_{n}(\emptyset)$ is the compact Cartan of $\G_{n}$, by $\Psi_{n}$ a set of positive roots of $\Delta_{n}$, by $\Phi_{n} = -\Psi_{n}$, by $\Psi^{\st}_{n}(n) = \left\{e_{t} - e_{2n+2-b}, 1 \leq b \leq n\right\}$ the corresponding set of strongly orthogonal roots of $\Psi_{n}$ and by $\Z_{n}$ the subgroup of $\G_{n}$ defined by $\Z_{n} = \G^{\mathfrak{h}'}_{n}$, where $\mathfrak{h}' = \Lie(\H')$ is the Lie algebra of $\H'$ seen as a subspace of $\mathfrak{h}_{n}$.

\noindent We denote by $\eta'(\S'_{1})$ the subspace of $\mathfrak{g}'$ defined by $\Hom(\X'_{1}, \Y'_{1}) \cap \mathfrak{g}'$ and by $\eta_{n}(\S_{1})$ the subspace of $\mathfrak{g}_{n} = \Lie(\G_{n})$ given by $\Hom(\X'_{1}, \U_{1}) \oplus \Hom(\X'_{1}, \Y'_{1}) \cap \mathfrak{g}_{n}$.

\end{nota}

\begin{rema}

As explained in Remark \ref{DecompositionHSiHSj}, we get for every $t \geq 1$ that
\begin{equation*}
\H_{n}(\S_{t}) = \T_{1}(\S_{1}) \times \A(\S_{1}) \times \H_{n-1}(\widetilde{\S}_{t-1})\,,
\end{equation*}
where $\widetilde{\S}_{t-1} = \S_{t} \setminus \left\{e_{1} - e_{2n+1}\right\}$ and $\H_{n-1}(\widetilde{\S}_{t-1})$ is a Cartan subgroup of $\G(\U_{1})$ whose split part is of dimension $t-1$. 
\noindent In particular, every element $\check{h} \in \check{\H}_{n, \S_{t}}$ can be written as $\check{h} = \check{t}\check{a}\check{h}_{1}$ (where, by convention, $t = a = \Id$ and $\check{h} = \check{h}_{1}$ if $t = 0$).

\label{DecompositionTAH}

\end{rema}

\begin{theo}

\label{IntegralFormulaCharacter}
For every $t \in [|0, n|]$ and $\check{h} = \check{t}\check{a}\check{h}_{1} \in \check{\H}_{n, \S_{t}}$ as in Remark \ref{DecompositionTAH}, we get, up to a constant, that:{\small
\begin{equation*}
\Theta_{\Pi^{n}}(c(\S_{t})\check{p}(\check{h})c(\S_{t})^{-1}) = \A \sum\limits_{\sigma \in \mathscr{W}(\H^{\mathbb{C}}_{n})}  \varepsilon(\sigma) \cfrac{\Delta_{\Phi(\Z_{n})}(\sigma^{-1}(\check{h}))}{\Delta_{\Phi_{n}}(\check{h})} \lim\limits_{\underset{r \to 1}{r \in \E_{\sigma, \S_{t}}}} \displaystyle\int_{\check{\H}'} \cfrac{\overline{\Theta_{\Pi'_{m}}(\check{p}(\check{h}')) \Delta_{\Psi'}(\check{h}')}}{\det(1-p(\check{h}')rp(\check{h}))_{\sigma\W^{\mathfrak{h}}}} d\check{h}'
\end{equation*}
\begin{equation*}
+ \delta_{t, 0} \B \cfrac{\overline{\Theta_{\Pi'_{m}}(c(\S_{1})\check{p}(\check{t}\check{a})c(\S_{1})^{-1})}|\Delta_{\Psi'}(\check{t}\check{a})|^{2}\d_{\S'_{1}}(c(\S_{1})\check{p}(\check{t}\check{a})c(\S_{1})^{-1})\d'_{\S_{1}}(c(\S_{t})\check{p}(\check{t}\check{a}\check{h}_{1})c(\S_{t})^{-1})\D(c(\S_{t})\check{p}(\check{h})c(\S_{t})^{-1})|\Delta_{\G(\U_{1})}(\check{h})|^{2}}{\varepsilon(\widetilde{(-1)}c(\widetilde{\S}_{t-1})\check{p}(\check{h}_{1})c(\widetilde{\S}_{t-1})^{-1})^{-1}|\det(\Ad(c(\S_{1})\check{p}(\check{t}'\check{a}')c(\S_{1})^{-1})^{-1} - 1)_{|_{\eta'(\S_{1})}}|\D_{1}(c(\S_{t})\check{p}(\check{h})c(\S_{t})^{-1})|\Delta_{\G_{n}}(\check{h})|^{2}}\,,
\end{equation*}}
where $\delta_{t, 0} = \begin{cases} 0 & \text{ if } t = 0 \\ 1 & \text{ otherwise } \end{cases}$,  $\D$ and $\D_{1}$ are functions on $\widetilde{\H}_{n}(\S_{t})$, $1 \leq t \leq n$, given by
\begin{equation*}
\D_{1}(\tilde{h}) = |\det(\Id - \Ad(\tilde{h})^{-1})_{\mathfrak{l}_{n}(\S_{1})/\mathfrak{h}_{n}(\S_{1})}|^{\frac{1}{2}}\,, \qquad \D(\tilde{h}) = |\det(\Id - \Ad(\tilde{h})^{-1})_{\mathfrak{g}_{n}/\mathfrak{h}_{n}(\S_{1})}|^{\frac{1}{2}}\,, \qquad (\tilde{h} \in \widetilde{\H}_{n}(\S_{t}))\,.
\end{equation*}
and where $\A$ and $\B$ are constants given by 
\begin{equation*}
\A = \cfrac{(-1)^{u}}{2(2n-1)!}\,, \qquad \B = \cfrac{2^{4n+3}(2n-1)m_{\widetilde{\S}_{t-1}}}{2(2n+1)^{2}m_{\S_{t}}}\,.
\end{equation*}

\end{theo}

\noindent Before proving Theorem \ref{IntegralFormulaCharacter}, we recall a lemma concerning orbital integrals.

\begin{lemme}

For every $t \geq 1$, $\tilde{h} \in \widetilde{\H}_{n}(\S_{t})^{\reg}$ and $\Psi \in \mathscr{C}^{\infty}_{c}(\widetilde{\G_{n}})$, we get:
\begin{equation*}
\displaystyle\int_{\G_{n}/\H_{n}(\S_{t})} \Psi(g\tilde{h}g^{-1}) \overline{dg} = \C_{n, 1} \cfrac{\D_{1}(\tilde{h})}{\D(\tilde{h})} \displaystyle\int_{\L_{1}(\S_{1})/\H_{n}(\S_{t})} \Psi^{\widetilde{\K_{n}}}_{\widetilde{\N_{n}}(\S_{1})}(m\tilde{h}m^{-1}) \overline{dm}\,,
\end{equation*}
where $\K_{n}$ is the maximal compact subgroup of $\G_{n}$, $\N_{n}(\S_{1}) = \exp(\eta_{n}(\S_{1}))$ and $\C_{n, 1}$ is the constant given by 
\begin{equation*}
\C_{n, 1} = \cfrac{1}{\mu(\K_{n} \cap \L_{n}(\S_{1}))\sqrt{2}^{\dim_{\mathbb{R}}(\eta_{n}(\S_{1}))}} = \cfrac{1}{\mu(\K_{n} \cap \L_{n}(\S_{1}))\sqrt{2}^{4n-1}}\,.
\end{equation*}

\label{TechnicalLemma}

\end{lemme}

\begin{proof}

We see easily that for $t \geq 1$, $\H_{n}(\S_{t})$ is a Cartan subgroup of $\L_{n}(\S_{1})$, and then the result follows from \cite[Corollary~A.4]{TOM2}.

\end{proof}

\begin{proof}[Proof of Theorem \ref{IntegralFormulaCharacter}]

Fix $t \in [|0, n|]$ and $\Psi \in \mathscr{C}^{\infty}_{c}(\widetilde{\G_{n}})$ such that $\supp(\Psi) \subseteq \widetilde{\G_{n}}\cdot \widetilde{\H_{n}}(\S_{t})$. Using Remark \ref{RemarkWeylIntegration2}, it follows that:
\begin{eqnarray}
\Theta_{\Pi^{n}}(\Psi) & = & \displaystyle\int_{\widetilde{\G_{n}}} \Theta_{\Pi^{n}}(\tilde{g}) \Psi(\tilde{g}) d\tilde{g} = m_{\S_{t}} \displaystyle\int_{\check{\H}_{n, \S_{t}}} \Theta_{\Pi^{n}}(c(\S_{t})\check{p}(\check{h})c(\S_{t})^{-1}) \varepsilon_{\Psi_{n, \S_{t}, \mathbb{R}}}(\check{h}) \Delta_{\Phi(n)}(\check{h}) \mathscr{H}_{\S_{t}}\Psi(\check{h}) d\check{h} \nonumber \\ 
 & = & m_{\S_{t}} \displaystyle\int_{\check{\H}_{n, \S_{t}}} \Theta_{\Pi^{n}}(c(\S_{t})\check{p}(\check{h})c(\S_{t})^{-1}) |\Delta_{\G_{n}}(\check{h})|^{2} \displaystyle\int_{\G_{n} / \H_{n}(\S_{t})} \Psi(gc(\S_{t})\check{p}(\check{h})c(\S_{t})^{-1}g^{-1}) \overline{dg}d\check{h}\,. \label{ThetaPiOnPsi}
 \end{eqnarray}
\noindent According to Remark \ref{RemarkConjecture}, the global character $\Theta_{\Pi^{n}}$ of $\Pi^{n}$ is given by
\begin{equation*}
\Theta_{\Pi^{n}}(\Psi) = \frac{1}{2}\displaystyle\int_{\widetilde{\H'}} \Theta_{\Pi'_{m}}(\tilde{h}'_{0}) |\det(\Id - \Ad(\tilde{h}'_{0})^{-1})_{\mathfrak{g}'/\mathfrak{h}'}| \Chc_{\tilde{h}'_{0}}(\Psi) d\tilde{h}'_{0} + \frac{1}{2}\displaystyle\int_{\widetilde{\H'}(\S'_{1})} \Theta_{\Pi'_{m}}(\tilde{h}'_{1}) |\det(\Id - \Ad(\tilde{h}'_{1})^{-1})_{\mathfrak{g}'/\mathfrak{h}'(\S'_{1})}| \Chc_{\tilde{h}'_{1}}(\Psi) d\tilde{h}'_{1}\,,
\end{equation*}
where $\H', \H'(\S_{1})$ are the two Cartan subgroups of $\G'$ (up to conjugation) defined in Equation \eqref{CartanU(1,1)}. Using that $\supp(\Psi) \subseteq \widetilde{\G_{n}}\cdot \widetilde{\H_{n}}(\S_{t})$, we get from Theorem \ref{TheoremTheta} and \cite[Equation~8]{TOM4} that:{\small
\begin{eqnarray}
 & & \displaystyle\int_{\widetilde{\H'}} \overline{\Theta_{\Pi'_{m}}(\tilde{h}'_{0})} |\det(\Id - \Ad(\tilde{h}'_{0})^{-1})_{\mathfrak{g}'/\mathfrak{h}'}| \Chc_{\tilde{h}'_{0}}(\Psi) d\tilde{h}'_{0} =  \displaystyle\int_{\widetilde{\H'}} \displaystyle\int_{\widetilde{\G_{n}}} \overline{\Theta_{\Pi'_{m}}(\tilde{h}'_{0})} |\det(\Id - \Ad(\tilde{h}'_{0})^{-1})_{\mathfrak{g}'/\mathfrak{h}'}|^{2} \Theta(\widetilde{(-1)}\tilde{g}\tilde{h}'_{0}) \Psi(\tilde{g}) d\tilde{g}d\tilde{h}'_{0} \nonumber \\ 
     & = &  \displaystyle\int_{\check{\H}'}\overline{\Theta_{\Pi'_{m}}(\check{p}(\check{h}'_{0})) \Delta_{\Psi'}(\check{h}'_{0})} \left(\Delta_{\Psi'}(\check{h}'_{0}) \displaystyle\int_{\widetilde{\G_{n}}}  \Theta(\check{p}(\check{h}'_{0}) \tilde{g}) \Psi(\tilde{g}) d\tilde{g}\right) d\check{h}'_{0} \nonumber \\
     & = &  \sum\limits_{\sigma \in \mathscr{W}(\H^{\mathbb{C}}_{n})}  \M_{\S_{t}}(\sigma) \lim\limits_{\underset{r \to 1}{r \in \E_{\sigma, \S_{t}}}} \displaystyle\int_{\check{\H}'} \overline{\Theta_{\Pi'_{m}}(\check{p}(\check{h}'_{0})) \Delta_{\Psi'}(\check{h}'_{0})} \displaystyle\int_{\check{\H}^{\reg}_{n, \S_{t}}} \cfrac{\Delta_{\Psi_{n}}(\check{h}) \Delta_{\Phi(\Z_{n})}(\sigma^{-1}(\check{h}))}{\det(1-p(\check{h}'_{0})rp(\check{h}))_{\sigma\W^{\mathfrak{h}'}}} \displaystyle\int_{\G_{n} / \H_{n}(\S_{t})} \Psi(gc(\S_{t})\check{p}(\check{h})c(\S_{t})^{-1}g^{-1}) \overline{dg} d\check{h}'_{0}\,. \label{EquationCompactEqn}
\end{eqnarray}}

\noindent Similarly, by using Equation \eqref{ChcNonCompact}, we get:{\small
\begin{equation}
\displaystyle\int_{\widetilde{\H'}(\S'_{1})} \overline{\Theta_{\Pi'_{m}}(\tilde{h}'_{1})} |\det(\Id - \Ad(\tilde{h}'_{1})^{-1})_{\mathfrak{g}'/\mathfrak{h}'(\S'_{1})}|^{2} \Chc_{\tilde{h}'_{1}}(\Psi) d\tilde{h}'_{1} =
\label{ValueOfChcOnH2}
\end{equation}
\begin{equation*}
\begin{cases} 0 & \text{ if } \S = \emptyset \\ \C_{1}\displaystyle\int_{\widetilde{\H'}(\S'_{1})} \overline{\Theta_{\Pi'_{m}}(\tilde{h}'_{1})}|\det(\Id - \Ad(\tilde{h}'_{1})^{-1})_{\mathfrak{g}'/\mathfrak{h}'(\S'_{1})}|\cfrac{|\det(\Ad(\tilde{h}'_{1}))_{|_{\eta'(\S'_{1})}}|^{\frac{1}{2}}}{|\det(\Ad(\tilde{h}'_{1}) - 1)_{|_{\eta'(\S'_{1})}}|}  \displaystyle\int_{\widetilde{\G}(\U_{1})} |\det(\tilde{h}'_{1}\tilde{u})_{|_{\eta_{n}(\S_{1})}}|^{\frac{1}{2}} \varepsilon(\widetilde{(-1)}\tilde{u}) \Psi^{\widetilde{\K}_{n}}_{\widetilde{\N}_{n}(\S_{1})}(\tilde{h}'_{1}\tilde{u}) d\tilde{u} d\tilde{h}'_{1} & \text{ otherwise } \end{cases}
\end{equation*}}
where the constant $\C_{1}$ defined in Equation \eqref{ConstantCi} is given by
\begin{equation*}
\C_{1} = \cfrac{2^{4n+3}(2n-1)}{(2n+1)^{2}}\C_{n, 1}\,.
\end{equation*}
and $\C_{n, 1}$ is defined in Lemma \ref{TechnicalLemma}. In particular, the theorem follows for $t = 0$, i.e. $\S_{0} = \{\emptyset\}$. From now on, we assume that $t \geq 1$, i.e. without loss of generality that $e_{1} - e_{2n+1} \in \S_{t}$. In this case, using Remark \ref{DecompositionHSiHSj}, we get:
\begin{equation*}
\H_{n}(\S_{t}) = \T_{1}(\S_{1}) \times \A(\S_{1}) \times \H_{n-1}(\widetilde{\S}_{t-1})\,.
\end{equation*}

\noindent The Cartan subgroup $\H_{n}(\S_{t})$ is included in the Levi $\L_{n}(\S_{1}) = \GL(\X'_{1}) \times \G(\U_{1})$ of $\P_{n}(\S_{1})$. In particular, using Lemma \ref{TechnicalLemma}, Equation \ref{EquationCompactEqn} can be written as:{\small
\begin{equation*}
\sum\limits_{\sigma \in \mathscr{W}(\H^{\mathbb{C}}_{n})}  \M_{\S_{t}}(\sigma) \C_{n, 1} \lim\limits_{\underset{r \to 1}{r \in \E_{\sigma, \S_{t}}}} 
\end{equation*}
\begin{equation*}
\displaystyle\int_{\check{\H}'} \overline{\Theta_{\Pi'_{m}}(\check{p}(\check{h}'_{0})) \Delta_{\Psi'}(\check{h}'_{0})} \displaystyle\int_{\check{\H}^{\reg}_{n, \S_{t}}} \cfrac{\Delta_{\Psi_{n}}(\check{h}) \Delta_{\Phi(\Z_{n})}(\sigma^{-1}(\check{h}))}{\det(1-p(\check{h}'_{0})rp(\check{h}))_{\sigma\W^{\mathfrak{h}}}} \cfrac{\D_{1}(c(\S_{t})\check{p}(\check{h})c(\S_{t})^{-1})}{\D(c(\S_{t})\check{p}(\check{h})c(\S_{t})^{-1})} \displaystyle\int_{\L_{n}(\S_{1}) / \H_{n}(\S_{i})} \Psi^{\widetilde{\K}_{n}}_{\widetilde{\N}_{n}(\S_{1})}(gc(\S_{t})\check{p}(\check{h})c(\S_{t})^{-1}g^{-1}) \overline{dg} d\check{h} d\check{h}'_{0}\,.
\end{equation*}}
\noindent Similarly, Equation \eqref{ValueOfChcOnH2} is equal to {\small
\begin{equation}
\C_{1}m_{\widetilde{\S}_{t-1}}\displaystyle\int_{\check{\T}'_{1, \S_{1}}} \displaystyle\int_{\check{\A}'_{\S_{1}}} \overline{\Theta_{\Pi'_{m}}(c(\S_{1})\check{p}(\check{t}'\check{a}')c(\S_{1})^{-1})}|\det(\Id - \Ad(c(\S_{1})\check{p}(\check{t}'\check{a}')c(\S_{1})^{-1})^{-1})_{\mathfrak{g}'/\mathfrak{h}'(\S_{1})}|\cfrac{|\det(\Ad(c(\S_{1})\check{p}(\check{t}'\check{a}')c(\S_{1})^{-1}))_{|_{\eta'(\S_{1})}}|^{\frac{1}{2}}}{|\det(\Ad(c(\S_{1})\check{p}(\check{t}'\check{a}')c(\S_{1})^{-1}) - \Id)_{|_{\eta'(\S_{1})}}|}  
\label{FormulaOnH2}
\end{equation}
\begin{equation*}
\displaystyle\int_{\check{\H}_{n-1, \widetilde{\S}_{t-1}}} \displaystyle\int_{\G(\U_{1}) / \H_{n-1}(\widetilde{\S}_{t-1})} |\Delta_{\G(\U_{1})}(\check{h})|^{2} \det(c(\S_{1})\check{p}(\check{t}'\check{a}')c(\S_{1})^{-1}gc(\widetilde{\S}_{t-1})\check{p}(\check{h})c(\widetilde{\S}_{t-1})^{-1}g^{-1})_{|_{\eta_{n}(\S_{1})}}|^{\frac{1}{2}} \varepsilon(\widetilde{(-1)}c(\widetilde{\S}_{t-1})\check{p}(\check{h})c(\widetilde{\S}_{t-1})^{-1}) 
\end{equation*}
\begin{equation*}
\Psi^{\widetilde{\K_{n}}}_{\widetilde{\N}_{n}(\S_{1})}(c(\S_{1})\check{p}(\check{t}'\check{a}')c(\S_{1})^{-1}gc(\widetilde{\S}_{t-1})\check{p}(\check{h})c(\widetilde{\S}_{t-1})^{-1}g^{-1}) \overline{dg}d\check{h}d\check{a}'d\check{t}'\,. 
\end{equation*}}
where $\widetilde{\S}_{t - 1} = \S_{t} \setminus \{e_{1} - e_{2n+1}\}$. From \eqref{ThetaPiOnPsi}, we get:{\small
\begin{equation}
\Theta_{\Pi^{n}}(\Psi) = m_{\S_{t}} \C_{n, 1} \displaystyle\int_{\check{\H}'_{n, \S_{t}}} \Theta_{\Pi^{n}}(c(\S_{t})\check{p}(\check{h})c(\S_{t})^{-1}) |\Delta_{\G_{n}}(\check{h})|^{2} \cfrac{\D_{1}(c(\S_{t})\check{p}(\check{h})c(\S_{t})^{-1})}{\D(c(\S_{t})\check{p}(\check{h})c(\S_{t})^{-1})}\displaystyle\int_{\L_{n}(\S_{1}) / \H_{n}(\S_{t})} \Psi(gc(\S_{t})\check{p}(\check{h})c(\S_{t})^{-1}g^{-1}) \overline{dg}d\check{h}
\label{ThetaPiNU(1,1)}
\end{equation}}
and using that:
\begin{equation*}
\L_{n}(\S_{1})/\H_{n}(\S_{t}) = \GL(\X_{1})/(\T_{1}(\S_{1}) \times \A(\S_{1})) \times \G(\U_{1})/\H_{n-1}(\widetilde{\S}_{t-1}) = \G(\U_{1})/\H_{n-1}(\widetilde{\S}_{t-1})\,,
\end{equation*}
it follows from Equations \eqref{EquationCompactEqn}, \eqref{FormulaOnH2} and \eqref{ThetaPiNU(1,1)} that for every $\check{h} = \check{t}\check{a}\check{h}_{1} \in \check{\H}_{n, \S_{t}} = \T_{1, \S_{i}} \times \A_{\S_{i}} \times \H_{\widetilde{\S}_{t-1}}$, {\small
\begin{equation*}
m_{\S_{t}}\C_{n, 1} \Theta_{\Pi^{n}}(c(\S_{t})\check{p}(\check{h})c(\S_{t})^{-1}) |\Delta_{\G_{n}}(\check{h})|^{2} \cfrac{\D_{1}(c(\S_{i})\check{p}(\check{h})c(\S_{t})^{-1})}{\D(c(\S_{t})\check{p}(\check{h})c(\S_{t})^{-1})} 
\end{equation*}
\begin{equation*}
= \cfrac{\C_{n, 1} (-1)^{u}m_{\S_{t}}}{2(2n-1)!}\cfrac{\Delta_{\Psi_{n}}(\check{h}) \D_{1}(c(\S_{t})\check{p}(\check{h})c(\S_{t})^{-1})}{\D(c(\S_{t})\check{p}(\check{h})c(\S_{t})^{-1})} \sum\limits_{\sigma \in \mathscr{W}(\H^{\mathbb{C}}_{n})}  \varepsilon(\sigma) \Delta_{\Phi(\Z_{n})}(\sigma^{-1}(\check{h})) \lim\limits_{\underset{r \to 1}{r \in \E_{\sigma, \S_{t}}}} \displaystyle\int_{\check{\H}'} \cfrac{\overline{\Theta_{\Pi'_{m}}(\check{p}(\check{h}'_{0})) \Delta_{\Psi'}(\check{h}'_{0})}}{\det(1-p(\check{h}'_{0})rp(\check{h}))_{\sigma\W^{\mathfrak{h}}}} d\check{h}'_{0}
\end{equation*}
\begin{equation*}
+ \cfrac{m_{\widetilde{\S}_{t-1}}\C_{1}}{2} \cfrac{\overline{\Theta_{\Pi'_{m}}(c(\S_{1})\check{p}(\check{t}\check{a})c(\S_{1})^{-1})|\Delta_{\Psi'}(\check{t}\check{a})|^{2}}\d_{\S'_{1}}(c(\S_{1})\check{p}(\check{t}\check{a})c(\S_{1})^{-1})\d'_{\S_{1}}(c(\S_{t})\check{p}(\check{t}\check{a}\check{h}_{1})c(\S_{t})^{-1})\varepsilon(\widetilde{(-1)}c(\widetilde{\S}_{t-1})\check{p}(\check{h}_{1})c(\widetilde{\S}_{t-1})))|\Delta_{\G(\U_{1})}(\check{h})|^{2}}{|\det(\Ad(c(\S_{1})\check{p}(\check{t}'\check{a}')c(\S_{1})^{-1}) - \Id)_{|_{\eta'(\S_{1})}}|}
\end{equation*}}
and the result follows.

\end{proof}

\begin{lemme}

For every $\tilde{g} \in \widetilde{\G_{n}}$, $\varepsilon(\tilde{g}) = \pm 1$. 

\label{Lemma0702}

\end{lemme}

\begin{proof}

The space $\X'_{1} \otimes _{\mathbb{C}} \V_{n} \oplus \Y'_{1} \otimes_{\mathbb{C}} \V_{n}$ is a complete polarization of $\V' \otimes_{\mathbb{C}} \V_{n}$. In particular, $\widetilde{\X'}=  (\X'_{1} \otimes \V_{n})_{\mathbb{R}}$ is a maximal isotropic subspace of $\W = (\V' \otimes_{\mathbb{C}} \V_{n})_{\mathbb{R}}$. 

\noindent According to Equation \ref{LemmaAppendixC1}, the character $\varepsilon$ is defined on $\GL(\widetilde{\X'})^{c}$ by the following formula:
\begin{equation*}
\varepsilon(\tilde{g}) = \cfrac{\Theta(\tilde{g})}{|\Theta(\tilde{g})|} \qquad \left(g \in \GL(\widetilde{\X'})^{c}\right)\,.
\end{equation*}
In particular, using Equation \ref{LemmaAppendixC2}, for every $\tilde{g} \in \widetilde{\G}^{c}_{n}$, we get:
\begin{equation*}
\varepsilon(\tilde{g}) = \cfrac{\det_{\widetilde{\X'}}(\tilde{g})^{-\frac{1}{2}}}{|\det_{\widetilde{\X'}}(\tilde{g})^{-\frac{1}{2}}|} = \pm \cfrac{|\det_{\X'}(g)|^{-1}}{|\det_{\X'}(g)|^{-1}}\,.
\end{equation*}
Using the fact that $|\det_{\X'}(g)| = 1$, it follows that $\varepsilon(\tilde{g}) = \pm1$. 

\end{proof}

\noindent As explained in Appendix \ref{AppendixCartanUnitary} (see Equation \eqref{DiagonalCartanSubgroupHS}), for every $0 \leq t \leq n$ and $\S_{t} = \{e_{1} - e_{2n+2-t}, \ldots, e_{t} - e_{2n+2-t}\}$, \begin{equation}
\H_{n, \S_{t}} = \left\{h = \diag(e^{iX_{1} - X_{2n+1}}, \ldots, e^{iX_{t} - X_{2n+2-t}}, e^{iX_{t+1}}, \ldots, e^{iX_{2n+1-t}}, e^{iX_{t} + X_{2n+2-t}}, \ldots, e^{iX_{1} + X_{2n+1}}), X_{j} \in \mathbb{R}\right\}\,.
\label{EquationH(S)}
\end{equation}
In particular, using Remark \ref{DecompositionTAH}, we get that $h$ can be written as $h = tah_{1}$, where 
\begin{equation*}
t = \diag(e^{iX_{1}}, \underbrace{1, \ldots, 1}_{2n-1}, e^{iX_{1}}), \qquad a = \diag(e^{-X_{2n+1}}, \underbrace{1, \ldots, 1}_{2n-1}, e^{X_{2n+1}})
\end{equation*}
and 
\begin{equation*}
h_{1} = \diag(1, e^{iX_{2} - X_{2n}}, \ldots, e^{iX_{t} - X_{2n+2-t}}, e^{iX_{t+1}}, \ldots, e^{iX_{2n+1-t}}, e^{iX_{t} + X_{2n+2-t}}, \ldots, e^{iX_{2} + X_{2n}}, 1)\,.
\end{equation*}

\noindent We get $\Phi(\Z_{n}) = \left\{e_{i} - e_{j}, 2 \leq i < j \leq 2n\right\}$. In particular, for every $\sigma \in \mathscr{S}_{2n+1}$ and $\check{h} \in \check{\H}_{n}(\S_{t})$, we get:
\begin{equation*}
\cfrac{\Delta_{\Phi(\Z_{n})}(\sigma^{-1}(\check{h}))}{\Delta_{\Phi(n)}(\check{h})} = \cfrac{\prod\limits_{2 \leq i < j \leq 2n} \left(h^{\frac{1}{2}}_{\sigma(j)}h^{-\frac{1}{2}}_{\sigma(i)} - h^{-\frac{1}{2}}_{\sigma(j)}h^{\frac{1}{2}}_{\sigma(i)}\right)}{\prod\limits_{1 \leq i < j \leq 2n+1} \left(h^{\frac{1}{2}}_{i}h^{-\frac{1}{2}}_{j} - h^{-\frac{1}{2}}_{i}h^{\frac{1}{2}}_{j}\right)} = \varepsilon(\sigma)\cfrac{h^{n}_{\sigma(1)}h^{n}_{\sigma(2n+1)} \prod\limits_{\underset{i \neq \sigma(1), \sigma(2n+1)}{i=1}}^{2n+1} h_{i}}{\prod\limits_{\underset{j \neq \sigma(1)}{j=1}}^{2n+1} \left(h_{\sigma(1)} - h_{j}\right) \prod\limits_{\underset{j \neq \sigma(1), \sigma(2n+1)}{j=1}}^{2n+1} \left(h_{\sigma(2n+1)} - h_{j}\right)}\,.
\end{equation*}
Up to a sign, the quotient $\cfrac{\Delta_{\Phi(\Z_{n})}(\sigma^{-1}(\check{h}))}{\Delta_{\Phi(n)}(\check{h})}$ is "uniquely" determined by $\sigma(1)$ and $\sigma(2n+1)$. For $\sigma \in \mathscr{S}_{2n+1}$ and $i, j \in \{1, \ldots, 2n+1\}$ such that $\sigma(1) = i, \sigma(2n+1) = j$, we denote by $\Delta(i, j, \check{h})$ the following quantity:
\begin{equation*}
\Delta(i, j, \check{h}) = \varepsilon(\sigma) \cfrac{\Delta_{\Phi(\Z_{n})}(\sigma^{-1}(\check{h}))}{\Delta_{\Phi(n)}(\check{h})} = \cfrac{h^{n}_{i}h^{n}_{j} \prod\limits_{\underset{k \neq i, j}{k=1}}^{2n+1} h_{k}}{\prod\limits_{\underset{k \neq i}{k=1}}^{2n+1} \left(h_{i} - h_{k}\right) \prod\limits_{\underset{l \neq i, j}{l=1}}^{2n+1} \left(h_{j} - h_{l}\right)}\,.
\end{equation*}

\begin{lemme}

Let $k \in \mathbb{Z}$ and $a \in \mathbb{C}^{*} \setminus \S^{1}$. Then,
\begin{equation*}
\cfrac{1}{2i\pi}\displaystyle\int_{\S^{1}} \cfrac{z^{k}}{z-a}dz = \begin{cases} a^{k} & \text{ if } k \geq 0 \text{ and } |a| < 1 \\ -a^{k} & \text{ if } k < 0 \text{ and } |a| > 1 \\ 0 & \text{ otherwise }
\end{cases}
\end{equation*}

\label{LemmaComplexIntegrals}

\end{lemme}

\noindent We denote by $\C_{2}$ the constant $\frac{(-1)^{u}}{(p+q-2)!}$.

\begin{nota}

For an element $h \in \H^{\reg}_{n, \S_{t}}$ as in Equation \eqref{EquationH(S)}, we denote by $\J(h)$ and $\K(h)$ the subsets of $\{1, \ldots, t\}$ given by
\begin{equation*}
J(h) = \{j \in \{1, \ldots, t\}, \sgn(X_{2n+2-j}) = 1\}\,, \qquad K(h) = \{j \in \{1, \ldots, t\}, \sgn(X_{2n+2-j}) = -1\}\,,
\end{equation*}
where $\sgn(X)$ is defined for every $X \in \mathbb{R}^{*}$ by
\begin{equation*}
\sgn(X) = \begin{cases} 1 & \text{ if } X > 0 \\ -1 & \text{ if } X < 0 \end{cases}\,.
\end{equation*}
To simplify the notations, we will denote by $(h_{1}, \ldots, h_{2n+1})$ the components of $h$.

\noindent Finally, we denote by $\A_{t}$ and $\B_{t}$ the subsets of $\left\{1, \ldots, 2n+1\right\}$ given by
\begin{equation*}
\A_{t} = \left\{t+1, \ldots, n\right\}, \qquad \B_{t} = \left\{n+1, \ldots, 2n+1-t\right\}\,.
\end{equation*}

\label{NotationsSection6}

\end{nota}

\begin{theo}

For every $0 \leq t \leq n$, the value of $\Theta_{\Pi^{n}}$ on $\widetilde{\H}_{n}(\S_{t})$ is given by 
\begin{equation}
\Theta_{\Pi^{n}}(c(\S_{t})\check{p}(\check{h})c(\S_{t})^{-1}) = \pm\C \begin{cases}
-\widetilde{\A} \sum\limits_{\underset{i \in J(h) \cup \B_{t}}{j \in K(h) \cup \A_{t}}} h^{n}_{i}h^{n+m}_{j} \Omega_{i, j}(h) + \delta_{t, 0}\B e^{-(m+1)\sgn(X_{2n+1})X_{2n+1}}\Sigma(h) & \text{ if } m \geq 1 \\ \widetilde{\A} \sum\limits_{\underset{i \neq j}{i, j \in J(h) \cup \B_{t}}} h^{n}_{i}h^{n}_{j} \Omega_{i, j}(h) + \delta_{t, 0}\B e^{-(m+1)\sgn(X_{2n+1})X_{2n+1}}\Sigma(h) & \text{ if } m = 0 \\ \widetilde{\A}\sum\limits_{\underset{j \in J(h) \cup \B_{t}}{i \in K(h) \cup \A_{t}}} h^{n+m}_{i}h^{n}_{j} \Omega_{i, j}(h) + \delta_{t, 0}\B e^{(m-1)\sgn(X_{2n+1})X_{2n+1}}\Sigma(h)& \text{ if } m \leq -1
\end{cases}
\label{FINALEQUATION1908}
\end{equation}
where $\Omega_{i, j}, 1 \leq i \neq j \leq 2n+1$, and $\Sigma$ are the functions on $\H^{\reg}_{n, \S_{t}}$ are given by
\begin{equation*}
\Omega_{i, j}(h) = \cfrac{\prod\limits_{\underset{k \neq i, j}{k=1}}^{2n+1} h_{k}}{\prod\limits_{\underset{k \neq i}{k=1}}^{2n+1} (h_{i} - h_{k}) \prod\limits_{\underset{l \neq i, j}{l=1}}^{2n+1} (h_{j} - h_{l})}\,, \qquad \Sigma(h) = \cfrac{\sgn(X_{2n+1}) e^{imX_{1}}\left|e^{(2n-2)X_{2n+1}}\right|\left(1 - e^{-2X_{2n+1}}\right)}{\left|\prod\limits_{k=2}^{2n} \left(1 - h_{1}h^{-1}_{k}\right) \prod\limits_{k=2}^{2n} \left(1- h_{k}h^{-1}_{2n+1}\right)\right|\left|1 - e^{-2X_{2n+1}}\right|^{2}}\,,
\end{equation*}
and where $\check{h} \in \check{\H}_{n, \S_{t}}$ is as in Equation \eqref{EquationH(S)}, $\widetilde{\A} = \cfrac{\A}{(2n-1)!}$ and $\C$ is a constant.

\label{PropositionFinalFormulaThetaPin}
\end{theo}

\begin{proof}

Let $h = tah_{1} \in \H_{n, \S_{t}}$. We denote by $(h_{1}, \ldots, h_{2n+1})$ be the components of $h$. In particular, 
\begin{equation*}
h_{c} = \begin{cases} e^{iX_{c} - X_{2n+2-c}} & \text{ if } 1 \leq c \leq t \\ e^{iX_{c}} & \text{ if } t+1 \leq c \leq 2n+1-t \\ e^{iX_{2n+2-c}+ X_{c}} & \text{ if } 2n+2-t \leq c \leq 2n+1 \end{cases}\,.
\end{equation*}
We denote by $\Delta_{n}(\mathfrak{l}) := \Delta_{n}(\mathfrak{l}_{n}(\S_{1}))$ the set of roots of $\mathfrak{l}_{n}(\S_{1})$ and let $\Psi_{n}(\mathfrak{l}) := \Delta_{n}(\mathfrak{l}) \cap \Psi_{n}$. One can easily check that
\begin{equation*}
\Psi_{n}(\mathfrak{l}) = \left\{e_{i} - e_{j}, 2 \leq i < j \leq 2n\right\}\,.
\end{equation*}
Similarly, let $\Psi_{n}(\eta_{n}(\S_{1})) = \left\{e_{1} - e_{k}, 2 \leq k \leq 2n+1\right\} \cup \left\{e_{k} - e_{2n+1}, 2 \leq k \leq 2n\right\}$ the roots of $\eta_{n}(\S_{1})$. Then,
\begin{eqnarray*}
\cfrac{\D(c(\S_{t})\check{p}(\check{h})c(\S_{t})^{-1})}{\D_{1}(c(\S_{t})\check{p}(\check{h})c(\S_{t})^{-1})} & = & \cfrac{\left|\det(\Id - \Ad(c(\S_{t})\check{p}(\check{h})c(\S_{t})^{-1})^{-1})_{\mathfrak{g}_{n}/\mathfrak{h}_{n}(\S_{t})}\right|^{\frac{1}{2}}}{\left|\det(\Id - \Ad(c(\S_{t})\check{p}(\check{h})c(\S_{t})^{-1})^{-1})_{\mathfrak{l}_{n}(\S_{1})/\mathfrak{h}_{n}(\S_{t})}\right|^{\frac{1}{2}}} = \d_{\S_{1}}(c(\S_{t})\check{p}(\check{h})c(\S_{t})^{-1})^{-2}\cfrac{\prod\limits_{\alpha \in \Psi_{n}(\mathfrak{l})} \left|1 - \check{h}^{\alpha}\right|}{\prod\limits_{\alpha \in \Psi_{n}} \left|1 - \check{h}^{\alpha}\right|} \\
& = & \d_{\S_{1}}(c(\S_{t})\check{p}(\check{h})c(\S_{t})^{-1})^{-2}\prod\limits_{\alpha \in \Psi_{n}(\eta_{n}(\S_{1}))} \left|1 - \check{h}^{\alpha}\right| \\
& = & \d_{\S_{1}}(c(\S_{t})\check{p}(\check{h})c(\S_{t})^{-1})^{-2} \left|1 - h_{1}h^{-1}_{2n+1}\right| \prod\limits_{k=2}^{2n} \left|1 - h_{1}h^{-1}_{k}\right| \prod\limits_{k=2}^{2n} \left|1- h_{k}h^{-1}_{2n+1}\right|
\end{eqnarray*}

\noindent Moreover, 
\begin{equation*}
\cfrac{\left|\Delta_{\G}(\check{h})\right|^{2}}{\left|\Delta_{\G(\U_{1})}(\check{h})\right|^{2}} = \cfrac{\prod\limits_{\alpha \in \Psi_{n}}\left|1-\check{h}^{\alpha}\right|^{2}}{\prod\limits_{\alpha \in \Psi(\mathfrak{g}(\U_{1}))}\left|1-\check{h}^{\alpha}\right|^{2}} =  \left|1-h_{1}h^{-1}_{2n+1}\right|^{2} \prod\limits_{k=2}^{2n} \left|1-h_{1}h^{-1}_{k}\right|^{2} \prod\limits_{j=2}^{2n} \left|1-h_{k}h^{-1}_{2n+1}\right|^{2} 
\end{equation*}
In particular,
\begin{equation*}
\cfrac{\D(c(\S_{t})\check{p}(\check{h})c(\S_{t})^{-1})\left|\Delta_{\G(\U_{1})}(\check{h})\right|^{2}}{\D_{1}(c(\S_{t})\check{p}(\check{h})c(\S_{t})^{-1})\left|\Delta_{\G}(\check{h})\right|^{2}} =\d_{\S_{1}}(c(\S_{t})\check{p}(\check{h})c(\S_{t})^{-1})^{-2}\left|1-e^{-2X_{2n+1}}\right|^{-1}\left|\prod\limits_{k=2}^{2n} \left(1 - h_{1}h^{-1}_{k}\right) \prod\limits_{k=2}^{2n} \left(1- h_{k}h^{-1}_{2n+1}\right)\right|^{-1}.
\end{equation*}
Similarly, 
\begin{equation*}
\left|\Delta_{\Psi'}(\check{t}\check{a})\right|^{2} = \left|1 - h_{1}h^{-1}_{2n+1}\right|^{2} = \left|1-e^{-2X_{2n+1}}\right|^{2}
\end{equation*}
and it follows from Remark \ref{LastRemarkU(1,1)} that 
\begin{equation*}
\Theta_{\Pi'_{m}}(c(\S_{1})\check{p}(\check{t}\check{a})c(\S_{1})^{-1}) = \pm \begin{cases} \sgn(X_{2n+1})\cfrac{e^{-ikX_{1}}e^{k\sgn(X_{2n+1})X_{2n+1}}}{e^{X_{2n+1}} - e^{-X_{2n+1}}} & \text{ if } m \leq -1 \\ \sgn(X_{2n+1}) \cfrac{e^{-ikX_{1}}e^{-k\sgn(X_{2n+1})X_{2n+1}}}{e^{X_{2n+1}} - e^{-X_{2n+1}}} & \text{ if } m \geq 0 \end{cases}\,,
\end{equation*}
i.e.
\begin{equation*}
\left|\Delta_{\Psi'}(\check{t}\check{a})\right|^{2} \overline{\Theta_{\Pi'_{m}}(c(\S_{1})\check{p}(\check{t}\check{a})c(\S_{1})^{-1})} = \pm \begin{cases} \sgn(X_{2n+1})e^{ikX_{1}}e^{(k-1)\sgn(X_{2n+1})X_{2n+1}}(1-e^{-2X_{2n+1}}) & \text{ if } m \leq -1 \\ \sgn(X_{2n+1}) e^{ikX_{1}}e^{-(k+1)\sgn(X_{2n+1})X_{2n+1}}(1-e^{-2X_{2n+1}}) & \text{ if } m \geq 0 \end{cases},
\end{equation*}
Finally, using that 
\begin{equation*}
\left|\det(\Ad(c(\S_{1})\check{p}(\check{t}'\check{a}')c(\S_{1})^{-1})^{-1} - 1)_{|_{\eta'(\S_{1})}}\right| = |1 - e^{-2X_{2n+1}}|, \qquad \d_{\S'_{1}}(c(\S_{1})\check{p}(\check{t}\check{a})c(\S_{1})^{-1}) = \left|e^{-X_{2n+1}}\right|
\end{equation*}
and
\begin{equation*}
\d'_{\S_{1}}(c(\S_{t})\check{p}(\check{h})c(\S_{t})^{-1}) = \left|\prod\limits_{j = 2}^{2n+1} \left(e^{iX_{1}-X_{2n+1}}h^{-1}_{j}\right) \prod\limits_{j = 2}^{2n} \left(h_{j}e^{-iX_{1}-X_{2n+1}}\right)\right|^{\frac{1}{2}} = \left|e^{-(2n-1)X_{2n+1}}\right|\,,
\end{equation*}
we get the second members of Equation \eqref{FINALEQUATION1908}.

\bigskip 

\noindent We now look at the first member of Equation \eqref{FINALEQUATION1908}. One can easily check that for every $\sigma \in \mathscr{S}_{2n+1}$,  $w = w_{1, 1}\E_{1, 1} + w_{2n+1, 2}\E_{2n+1, 2} \in \W^{\mathfrak{h}'}$ and $y \in \mathfrak{h}_{n}$ such that $y = (y_{1}, \ldots, y_{2n+1}) = (iX_{1}, \ldots, iX_{2n+1})$, we get:
\begin{equation*}
\langle y\sigma(w), \sigma(w)\rangle = \begin{cases} X_{\sigma(1)} |w_{1, 1}|^{2} + X_{\sigma(2n+1)} |w_{2n+1, 2}|^{2} & \text{ if } \sigma(1), \sigma(2n+1) \in \{1, \ldots, n\} \\ - X_{\sigma(1)} |w_{1, 1}|^{2} - X_{\sigma(2n+1)} |w_{2n+1, 2}|^{2} & \text{ if } \sigma(1), \sigma(2n+1) \in \left\{n+1, \ldots, 2n+1\right\} \\ X_{\sigma(1)} |w_{1, 1}|^{2} - X_{\sigma(2n+1)} |w_{2n+1, 2}|^{2} & \text{ if } \sigma(1) \in \left\{1, \ldots, n\right\} \text{ and } \sigma(2n+1) \in \left\{n+1, \ldots, 2n+1\right\} \\ - X_{\sigma(1)} |w_{1, 1}|^{2} + X_{\sigma(2n+1)} |w_{2n+1, 2}|^{2} & \text{ if } \sigma(1) \in \left\{n+1, \ldots, 2n+1\right\} \text{ and } \sigma(2n+1) \in \left\{1, \ldots, n\right\} \end{cases}
\end{equation*}
(see the proof of Proposition \ref{PropositionA3} for en easier computation) and then

{\small
\begin{equation*}
\Gamma_{\sigma, \S_{t}} = \begin{cases} \mathfrak{h}_{n} & \text{ if } \left\{\sigma(1), \sigma(2n+1)\right\} \subseteq \underline{\S_{t}} \\ \left\{y = (y_{1}, \ldots, y_{2n+1}) \in \mathfrak{h}_{n}, X_{\sigma(2n+1)} > 0\right\} & \text{ if } \sigma(1) \in \underline{\S_{t}} \text{ and } \sigma(2n+1) \in \A_{t} \\ \left\{y = (y_{1}, \ldots, y_{2n+1}) \in \mathfrak{h}_{n}, X_{\sigma(2n+1)} < 0\right\} &\text{ if } \sigma(1) \in \underline{\S_{t}} \text{ and } \sigma(2n+1) \in \B_{t} \\ \left\{y = (y_{1}, \ldots, y_{2n+1}) \in \mathfrak{h}_{n}, X_{\sigma(1)} > 0\right\} &\text{ if } \sigma(2n+1) \in \underline{\S_{t}} \text{ and } \sigma(1) \in \A_{t} \\ \left\{y = (y_{1}, \ldots, y_{2n+1}) \in \mathfrak{h}_{n}, X_{\sigma(1)} < 0\right\} & \text{ if } \sigma(2n+1) \in \underline{\S_{t}} \text{ and } \sigma(1) \in \B_{t} \\ \left\{y = (y_{1}, \ldots, y_{2n+1}) \in \mathfrak{h}_{n}, X_{\sigma(1)} > 0, X_{\sigma(2n+1)} > 0\right\} & \text{ if } \left\{\sigma(1), \sigma(2n+1)\right\} \cap \underline{\S_{t}} = \left\{\emptyset\right\} \text{ and } \sigma(1), \sigma(2n+1) \in \A_{t} \\ \left\{y = (y_{1}, \ldots, y_{2n+1}) \in \mathfrak{h}_{n}, X_{\sigma(1)} > 0, X_{\sigma(2n+1)} < 0\right\} & \text{ if } \left\{\sigma(1), \sigma(2n+1)\right\} \cap \underline{\S_{t}} = \left\{\emptyset\right\} \text{ and } \sigma(1) \in \A_{t}, \sigma(2n+1) \in \B_{t} \\ \left\{y = (y_{1}, \ldots, y_{2n+1}) \in \mathfrak{h}_{n}, X_{\sigma(1)} < 0, X_{\sigma(2n+1)} > 0\right\} & \text{ if } \left\{\sigma(1), \sigma(2n+1)\right\} \cap \underline{\S_{t}} = \left\{\emptyset\right\} \text{ and } \sigma(2n+1) \in \A_{t}, \sigma(1) \in \B_{t} \\ \left\{y = (y_{1}, \ldots, y_{2n+1}) \in \mathfrak{h}_{n}, X_{\sigma(1)} < 0, X_{\sigma(2n+1)} < 0\right\} & \text{ if } \left\{\sigma(1), \sigma(2n+1)\right\} \cap \underline{\S_{t}} = \left\{\emptyset\right\} \text{ and } \sigma(1), \sigma(2n+1) \in \B_{t}\end{cases}
\end{equation*}}

\noindent Let $\check{h} \in \check{\H}_{n}(\S_{t})$ and $h = p(\check{h}) = (h_{1}, \ldots, h_{2n+1})$ as in Notations \ref{NotationsSection6}. Assume that $m \geq 1$. Using Corollary \ref{FinalCorollaryU(1)}, we get that 
\begin{equation*}
\Theta_{\Pi'_{k}}(\check{p}(\check{h}'_{0})) \Delta(\check{p}(\check{h}'_{0})) = h'^{\frac{1}{2}}_{1} h'^{\frac{1}{2}}_{2} \cfrac{h'^{-k}_{2}}{h'_{2} - h'_{1}} h'^{-\frac{1}{2}}_{1} h'^{-\frac{1}{2}}_{2}(h'_{1} - h'_{2}) = -h'^{-k}_{2}\,, \qquad (h'_{0} = \diag(h'_{1}, h'_{2}))\,.
\end{equation*} 
Then,
\begin{eqnarray*}
& & \sum\limits_{\sigma \in \mathscr{W}(\H^{\mathbb{C}}_{n})}  \varepsilon(\sigma) \cfrac{\Delta_{\Phi(\Z_{n})}(\sigma(\check{h}))}{\Delta_{\Phi(n)}(\check{h})} \lim\limits_{\underset{r \to 1}{r \in \E_{\sigma, \S_{t}}}} \displaystyle\int_{\check{\H}'} \cfrac{\overline{\Theta_{\Pi'_{m}}(\check{p}(\check{h}'_{0}))} \overline{\Delta(\check{p}(\check{h}'_{0}))}}{\det(1-p(\check{h}'_{0})p(\check{h}))_{\sigma\W^{\mathfrak{h}}}} d\check{h}'_{0} \\
& = & \frac{1}{(2i\pi)^{2}} \sum\limits_{\sigma \in \mathscr{W}(\H^{\mathbb{C}}_{n})} \varepsilon(\sigma) \cfrac{\Delta_{\Phi(\Z_{n})}(\sigma(\check{h}))}{\Delta_{\Phi(n)}(\check{h})} \lim\limits_{\underset{r \to 1}{r \in \E_{\sigma, \S_{t}}}} \displaystyle\int_{\H'_{1}} \cfrac{h'^{-1}_{1}h'^{m-1}_{2}}{(1 - h'_{1}(rh)^{-1}_{\sigma(1)})(1 - h'_{2}(rh)^{-1}_{\sigma(2n+1)})} dh'_{1}dh'_{2} \\
& = & \frac{1}{(2i\pi)^{2}} \sum\limits_{\sigma \in \mathscr{S}_{2n+1}} \varepsilon(\sigma) \cfrac{h_{\sigma(1)} h_{\sigma(2n+1)} \Delta_{\Phi(\Z_{n})}(\sigma(\check{h}))}{\Delta_{\Phi(n)}(\check{h})} \lim\limits_{\underset{r \to 1}{r \in \E_{\sigma, \S_{t}}}}  \displaystyle\int_{\S^{1}} \cfrac{h'^{-1}_{1}dh'_{1}}{h'_{1} - rh_{\sigma(1)}}\displaystyle\int_{\S^{1}} \cfrac{h'^{m-1}_{2}dh'_{2}}{h'_{2} - rh_{\sigma(2n+1)}} \\
& = & \frac{1}{(2i\pi)^{2}(2n-1)!} \sum\limits_{1 \leq i \neq j \leq 2n+1} h_{i} h_{j} \Delta(i, j, \check{h}) \lim\limits_{\underset{r \to 1}{r \in \E_{\sigma, \S_{t}}}}  \displaystyle\int_{\S^{1}} \cfrac{h'^{-1}_{1}dh'_{1}}{h'_{1} - rh_{i}}\displaystyle\int_{\S^{1}} \cfrac{h'^{m-1}_{2}dh'_{2}}{h'_{2} - rh_{j}} \\
& = & - \frac{1}{(2n-1)!} \sum\limits_{i \in J(h) \cup \B_{t}} \sum\limits_{j \in K(h) \cup \A_{t}} h_{i}h_{j}\Delta(i, j, \check{h}) h^{-1}_{i}h^{m-1}_{j} \\
& = & - \frac{1}{(2n-1)!} \sum\limits_{i \in J(h) \cup \B_{t}} \sum\limits_{j \in K(h) \cup \A_{t}} \cfrac{h^{n}_{i}h^{n+m}_{j} \prod\limits_{\underset{k \neq i, j}{k=1}}^{2n+1} h_{k}}{\prod\limits_{\underset{k \neq i}{k=1}}^{2n+1} (h_{i} - h_{k}) \prod\limits_{\underset{l \neq i, j}{l=1}}^{2n+1} (h_{j} - h_{l})}
\end{eqnarray*}
Similarly, if $m = 0$, we get:
\begin{eqnarray*}
& & \sum\limits_{\sigma \in \mathscr{W}(\H^{\mathbb{C}}_{n})}  \varepsilon(\sigma) \cfrac{\Delta_{\Phi(\Z_{n})}(\sigma(\check{h}))}{\Delta_{\Phi(n)}(\check{h})} \lim\limits_{\underset{r \to 1}{r \in \E_{\sigma, \S_{t}}}} \displaystyle\int_{\check{\H}'} \cfrac{\overline{\Theta_{\Pi'_{m}}(\check{p}(\check{h}'_{0}))} \overline{\Delta(\check{p}(\check{h}'_{0}))}}{\det(1-p(\check{h}'_{0})p(\check{h}))_{\sigma\W^{\mathfrak{h}}}} d\check{h}'_{1} \\
& = & \frac{1}{(2i\pi)^{2}(2n-1)!} \sum\limits_{1 \leq i \neq j \leq 2n+1} h_{i} h_{j} \Delta(i, j, \check{h}) \lim\limits_{\underset{r \to 1}{r \in \E_{\sigma, \S_{t}}}}  \displaystyle\int_{\S^{1}} \cfrac{h'^{-1}_{1}dh'_{1}}{h'_{1} - rh_{i}}\displaystyle\int_{\S^{1}} \cfrac{h'^{-1}_{2}dh'_{2}}{h'_{2} - rh_{j}} \\
& = & \frac{1}{(2n-1)!} \sum\limits_{\underset{i \neq j}{i, j \in J(h) \cup \B_{t}}} h_{i}h_{j}\Delta(i, j, \check{h}) h^{-1}_{i}h^{-1}_{j} = \frac{1}{(2n-1)!}\sum\limits_{\underset{i \neq j}{i, j \in J(h) \cup \B_{t}}} \cfrac{h^{n}_{i}h^{n}_{j} \prod\limits_{\underset{k \neq i, j}{k=1}}^{2n+1} h_{k}}{\prod\limits_{\underset{k \neq i}{k=1}}^{2n+1} (h_{i} - h_{k}) \prod\limits_{\underset{l \neq i, j}{l=1}}^{2n+1} (h_{j} - h_{l})}
\end{eqnarray*}
Finally, if $m \leq -1$, it follows from Corollary \ref{FinalCorollaryU(1)} that:
\begin{equation*}
\Theta_{\Pi'_{m}}(\check{p}(\check{h}'_{0})) \Delta(\check{p}(\check{h}'_{0})) = h'^{\frac{1}{2}}_{1} h'^{\frac{1}{2}}_{2} \cfrac{h'^{-m}_{1}}{h'_{2} - h'_{1}} h'^{-\frac{1}{2}}_{1} h'^{-\frac{1}{2}}_{2}(h'_{1} - h'_{2}) = -h'^{-m}_{1}\,, \qquad \left(h'_{0} = (h'_{1}, h'_{2})\right)\,.
\end{equation*} 
Then,
\begin{eqnarray*}
& & \sum\limits_{\sigma \in \mathscr{W}(\H^{\mathbb{C}}_{n})}  \varepsilon(\sigma) \cfrac{\Delta_{\Phi(\Z_{n})}(\sigma(\check{h}))}{\Delta_{\Phi(n)}(\check{h})} \lim\limits_{\underset{r \to 1}{r \in \E_{\sigma, \S_{t}}}} \displaystyle\int_{\check{\H}'} \cfrac{\overline{\Theta_{\Pi'_{m}}(\check{p}(\check{h}'_{0}))} \overline{\Delta(\check{p}(\check{h}'_{0}))}}{\det(1-p(\check{h}'_{0})p(\check{h}))_{\sigma\W^{\mathfrak{h}}}} d\check{h}'_{0} \\
& = & \cfrac{1}{(2i\pi)^{2}} \sum\limits_{\sigma \in \mathscr{W}(\H^{\mathbb{C}}_{n})} \varepsilon(\sigma) \cfrac{\Delta_{\Phi(\Z(n))}(\sigma(\check{h}))}{\Delta_{\Phi(n)}(\check{h})} \lim\limits_{\underset{r \to 1}{r \in \E_{\sigma, \S_{t}}}} \displaystyle\int_{\H'_{1}} \cfrac{h'^{m-1}_{1}h'^{1}_{2}}{(1 - h'_{1}(rh)^{-1}_{\sigma(1)})(1 - h'_{2}(rh)^{-1}_{\sigma(2n+1)})} dh'_{1}dh'_{2} \\
& = & \cfrac{1}{(2i\pi)^{2}} \sum\limits_{\sigma \in \mathscr{S}_{2n+1}} \varepsilon(\sigma) \cfrac{h_{\sigma(1)} h_{\sigma(2n+1)} \Delta_{\Phi(\Z(n))}(\sigma(\check{h}))}{\Delta_{\Phi(n)}(\check{h})} \lim\limits_{\underset{r \to 1}{r \in \E_{\sigma, \S_{t}}}}  \displaystyle\int_{\S^{1}} \cfrac{h'^{m-1}_{1}dh'_{1}}{h'_{1} - rh_{\sigma(1)}}\displaystyle\int_{\S^{1}} \cfrac{h'^{-1}_{2}dh'_{2}}{h'_{2} - rh_{\sigma(2n+1)}} \\
& = & \cfrac{1}{(2n-1)!}\sum\limits_{i \in K(h) \cup \A_{t}} \sum\limits_{j \in J(h) \cup \B_{t}} h_{i}h_{j}\Delta(i, j, \check{h}) h^{k-1}_{i}h^{-1}_{j} \\
& = & \cfrac{1}{(2n-1)!} \sum\limits_{i \in K(h) \cup \A_{t}} \sum\limits_{j \in J(h) \cup \B_{t}} \cfrac{h^{n+m}_{i}h^{n}_{j} \prod\limits_{\underset{k \neq i, j}{k=1}}^{2n+1} h_{k}}{\prod\limits_{\underset{k \neq i}{k=1}}^{2n+1} (h_{i} - h_{k}) \prod\limits_{\underset{l \neq i, j}{l=1}}^{2n+1} (h_{j} - h_{l})}\,,
\end{eqnarray*}
and the theorem follows.
\end{proof}

\noindent We finish this section with a Lemma concerning the formulas we got in Theorem \ref{PropositionFinalFormulaThetaPin}.

\begin{lemme}

For every $h \in \H_{n, \S_{t}}$, $\prod\limits_{k=2}^{2n} \left(1 - h_{1}h^{-1}_{k}\right) \prod\limits_{k=2}^{2n} \left(1- h_{k}h^{-1}_{2n+1}\right) \in \mathbb{R}$, and its sign is constant on every Weyl chamber.

\label{Lemma1502}

\end{lemme}

\begin{proof}

This result was obtained in \cite[Lemma~6.9]{MER} for $t = 1$. We prove this lemma for $t = 2$ (the proof of the general statement is similar). Assume that $t = 2$. We get:
\begin{equation*}
\prod\limits_{k=2}^{2n} \left(1 - h_{1}h^{-1}_{k}\right) \prod\limits_{k=2}^{2n} \left(1- h_{k}h^{-1}_{2n+1}\right) = (1 - h_{1}h^{-1}_{2})(1 - h_{1}h^{-1}_{2n})(1 - h_{2}h^{-1}_{2n+1})(1 - h_{2n}h^{-1}_{2n+1})\left(\prod\limits_{k=3}^{2n-1} \left(1 - h_{1}h^{-1}_{k}\right)\left(1- h_{k}h^{-1}_{2n+1}\right)\right)\,.
\end{equation*}
Firstly, 
\begin{equation*}
\left(\prod\limits_{j=3}^{2n-1} \left(1 - h_{1}h^{-1}_{j}\right)\left(1- h_{j}h^{-1}_{2n+1}\right)\right) = \prod\limits_{j=3}^{2n-1} \left(1 - e^{iX_{1} - X_{2n+1}}e^{-iX_{j}}\right)\left(1- e^{iX_{j}}e^{-iX_{1} - X_{2n+1}}\right) = \prod\limits_{j=3}^{2n-1} \left|1 - e^{iX_{1} - X_{2n+1}}e^{-iX_{j}}\right|^{2}\,.
\end{equation*}
Moreover,
\begin{equation*}
(1 - h_{1}h^{-1}_{2n})(1 - h_{2n}h^{-1}_{2n+1}) = (1 - e^{iX_{1} - X_{2n+1}}e^{-iX_{2} + X_{2n}})(1 - e^{iX_{2} + X_{2n}}e^{-iX_{1} - X_{2n+1}}) = 1 - 2\cos(X_{1} - X_{2})e^{X_{2n} - X_{2n+1}} + e^{2(X_{2n} - X_{2n+1})}
\end{equation*}
and
\begin{equation*}
(1 - h_{2}h^{-1}_{2n+1})(1 - h_{1}h^{-1}_{2n}) = (1 - e^{iX_{2} - X_{2n}}e^{-iX_{1} + X_{2n+1}})(1 - e^{iX_{1} + X_{2n+1}}e^{-iX_{2} - X_{2n}}) = 1 - 2\cos(X_{2} - X_{1})e^{X_{2n+1} - X_{2n+1}} + e^{2(X_{2n+1} - X_{2n})}\,.
\end{equation*}
so the lemma follows.
\end{proof}

\appendix

\section{The dual pair $(\G = \U(1), \G' = \U(p, q))$}

\label{ComputationsU(1)}

In \cite[Proposition~6.4]{MER}, the first author gave explicit formulas for the value of the character $\Theta_{\Pi'}$ on the compact Cartan $\widetilde{\H'} = \widetilde{\H'}(\emptyset)$ of $\widetilde{\G}'$. Moreover, for $p = q = 1$, he computed the character $\Theta_{\Pi'}$ on $\widetilde{\H}'(\S_{1})$, where $\H'(\S_{1})$ is the non-compact Cartan subgroup of $\widetilde{\U}(1, 1)$ as in Equation \ref{CartanU(1,1)}. In this section, we recover the results proved in \cite{MER} using the results of Section \ref{SectionUnitary} and get formulas for $\Theta_{\Pi'}$ on every Cartan subgroup of $\widetilde{\G'}$

\noindent By keeping the notations of Section \ref{SectionUnitary}, we get $\V = \mathbb{C}$ with the hermitian form $\left(\cdot, \cdot\right)$ given by
\begin{equation*}
\left(u,v\right)=u\overline{v}, \qquad (u, v \in \V)\,,
\end{equation*}
$\V' = \M_{n',1}(\mathbb{C})$, where $n'=p+q$, with the skew-hermitian form $\left(\cdot, \cdot\right)'$ given by:
\begin{equation*}
\left(u,v\right)=\overline{v}^{t}i\Id_{p,q}u\,,
\end{equation*}
and $\W = \V \otimes_{\mathbb{C}} \V'$ the symplectic space defined by
\begin{equation*}
\langle w, w'\rangle =\Re(\tr(w'^*w)) = \Im(\overline{w'}^t\Id_{p,q}w)\,.
\end{equation*}
Similarly, 
\begin{equation*}
\H= \G= \U(1)=\left\{h \in \mathbb{C}, |h|=1\right\}\,, \qquad \J_{1} = i, \qquad \mathfrak{h} =\mathbb{R}\J_{1}\,,
\end{equation*}
and the group $\GL_{\mathbb{C}}(\W)$ is given by:
\begin{equation*}
\GL_{\mathbb{C}}(\W)=\left\{g \in \GL(\W),\ \J'_{1}g=g\J'_{1}\right\}=\G'_{\mathbb{C}}\,, \qquad 
\widetilde{\GL}_{\mathbb{C}}(\W) =\left\{\tilde{g} = (g,\xi), g \in \GL_{\mathbb{C}}(\W), \det(g)=\xi^2\right\}\,.
\end{equation*}
Using that $\V' = \V'_{1} \oplus \ldots \oplus \V'_{n'}$, with $\V'_{j} = \mathbb{C}\E_{j, 1}$, and the embedding
\begin{equation*}
\mathfrak{h}_{\mathbb{C}} \ni \lambda \to (\lambda, 0, \ldots, 0) \in \mathfrak{h}'_{\mathbb{C}}\,,
\end{equation*}
we get that
\begin{equation*}
\W^{\mathfrak{h}}=\left\{w=(w_{1,1}, 0, 0, ..., 0)\,,\ w_{1,1} \in \mathbb{C}\right\}, \qquad \Z' = \G'^{\mathfrak{h}} = \left\{g' \in \G', g' = \begin{pmatrix} \lambda & 0 \\ 0 & X\end{pmatrix}, \lambda \in \U(1), X \in \GL(n'-1, \mathbb{C})\right\}\,.
\end{equation*}
In particular, $\Phi'(\Z') = \left\{\pm(e_{i} - e_{j}), 2 \leq i < j\leq n'\right\}$, where $e_{j}$ is the form defined in Notation \ref{NotationsRootsUnitary}. For every $\check{h}' \in \check{\H}'_{\mathbb{C}}$, with $h' = (h'_{1}, \ldots, h'_{n'})$, we get:
\begin{equation*}
\Delta_{\Phi'}(\check{h'}) = \prod\limits_{\alpha > 0} (\check{h}'^{\frac{\alpha}{2}} - \check{h}'^{-\frac{\alpha}{2}}) = \prod\limits_{1 \leq i < j \leq n'} (h'^{\frac{1}{2}}_{i}h'^{-\frac{1}{2}}_{j} - h'^{-\frac{1}{2}}_{i}h'^{\frac{1}{2}}_{j}) = \prod\limits_{i=1}^{n'} h'^{-\frac{n'-1}{2}}_{i} \prod\limits_{1 \leq i < j \leq n'} (h'_{i} - h'_{j})\,,
\end{equation*}
and for every $\sigma \in \mathscr{S}_{n'}$,
\begin{equation*}
\Delta_{\Phi'(\Z')}(\sigma(\check{h'})) = \prod\limits_{i=2}^{n'} h'^{-\frac{n'-2}{2}}_{\sigma(i)} \prod\limits_{2 \leq i < j \leq n'} (h'_{\sigma(i)} - h'_{\sigma(j)}) = \varepsilon(\sigma) \prod\limits_{\underset{i \neq \sigma(1)}{i=1}}^{n'} h'^{-\frac{n'-2}{2}}_{i} \prod\limits_{\underset{i, j \neq \sigma(1)}{1 \leq i < j \leq n'}} (h'_{i} - h'_{j})\,.
\end{equation*}
In particular, 
\begin{equation*}
\cfrac{\Delta_{\Phi'(\Z')}(\sigma(\check{h'}))}{\Delta_{\Phi'}(\check{h'})} = \varepsilon(\sigma) \cfrac{h'^{\frac{n'-1}{2}}_{\sigma(1)}\prod\limits_{\underset{i \neq \sigma(1)}{i=1}}^{n'} h'^{\frac{1}{2}}_{i}}{\prod\limits_{\underset{i \neq \sigma(1)}{i=1}}^{n'} (h'_{\sigma(1)} - h'_{i})} = \varepsilon(\sigma) \cfrac{h'^{\frac{n'-2}{2}}_{\sigma(1)}\prod\limits_{i=1}^{n'} h'^{\frac{1}{2}}_{i}}{\prod\limits_{\underset{i \neq \sigma(1)}{i=1}}^{n'} (h'_{\sigma(1)} - h'_{i})} \qquad (\check{h'} \in \check{\H}'_{\mathbb{C}})\,.
\end{equation*}

\begin{nota}

As in Section \ref{SectionUnitary2}, because the set of genuine representations of $\widetilde{\U}(1)$ is isomorphic to $\mathbb{Z}$, we will denote by $\Pi_{m}, m \in \mathbb{Z}$, the representations of $\mathscr{R}(\widetilde{\U}(1), \omega)$. Using \cite{VER}, we get that $\Pi_{m}(\tilde{h}) = \pm h^{m + \frac{q-p}{2}}$. We will denote by $\Pi'_{m}$ the corresponding representation of $\widetilde{\U}(p, q)$ and by $\Theta_{\Pi'_{m}}$ its character.

\label{NotationA1}

\end{nota}

\begin{prop}

For every $\S \subseteq \Psi'^{\st}_{n}$ (see Appendix \ref{AppendixCartanUnitary}), the value of the character $\Theta_{\Pi'_{m}}$ on $\widetilde{\H'}(\S)^{\reg}$ is given by the following formula:
\begin{equation*}
\Delta_{\Phi'}(\check{h}')\Theta_{\Pi'_{m}}(c(\S)\check{p}(\check{h}')c(\S)^{-1}) = \sum\limits_{\sigma \in \mathscr{W}(\H'_{\mathbb{C}})} \cfrac{(-1)^{u}\varepsilon(\sigma)}{|\mathscr{W}(\Z'_{\mathbb{C}}, \H'_{\mathbb{C}})|} \Delta_{\Phi'(\Z')}(\sigma^{-1}(\check{h}'))\det^{-\frac{k}{2}}(\sigma^{-1}(\check{h}'))_{\W^{\mathfrak{h}}} \lim\limits_{\underset{r \to 1}{r \in \E_{\sigma, \S}}} \displaystyle\int_{\check{\U}(1)} \cfrac{\overline{\Theta_{\Pi_{m}}(\check{p}(\check{h}))} \det^{-\frac{k}{2}}(\check{h})}{\det(1-p(\check{h})rp(\check{h}'))_{\sigma\W^{\mathfrak{h}}}} d\check{h}
\end{equation*}
for every $\check{h}' \in \check{\H}'^{\reg}_{\S}$.
\label{Proposition15Juillet}

\end{prop}

\begin{proof}

Let $\Psi \in \mathscr{C}^{\infty}_{c}(\widetilde{\G'})$ such that $\supp(\Psi) \subseteq \widetilde{\G'}\cdot\widetilde{\H}'(\S)$, we get:
\begin{equation}
\Theta_{\Pi'_{m}}(\Psi) = \tr(\mathscr{P}_{\Pi_{m}} \circ \omega(\Psi)) = \displaystyle\int_{\widetilde{\G'}} \left(\displaystyle\int_{\widetilde{\U}(1)} \overline{\Theta_{\Pi_{m}}(\tilde{g})} \Theta(\tilde{g}\tilde{g}') d\tilde{g}\right) \Psi(\tilde{g}')d\tilde{g}' = \displaystyle\int_{\widetilde{\G'}} \left(\displaystyle\int_{\check{\U}(1)} \overline{\Theta_{\Pi_{m}}(\check{p}(\check{h}))} \Theta(\check{p}(\check{h})\tilde{g}') d\check{h}\right) \Psi(\tilde{g}')d\tilde{g}',
\label{EquationAppendixU(1)}
\end{equation}
where $\mathscr{P}_{\Pi_{m}}: \mathscr{H} \to \mathscr{H}(\Pi_{m})$ is the projection onto the $\Pi_{m}$-isotypic component given by $\mathscr{P}_{\Pi_{m}} = \omega(\overline{\Theta_{\Pi_{m}}})$ (see \cite[Section~1.4.6]{WAL}), i.e. as a generalized function on $\widetilde{\G'}$\,, 
\begin{equation*}
\Theta_{\Pi'_{m}}(\tilde{g}') = \displaystyle\int_{\check{\U}(1)} \overline{\Theta_{\Pi_{m}}(\check{p}(\check{h})))} \Theta(\check{p}(\check{h})\tilde{g}') d\tilde{g}\,, \qquad (\tilde{g}' \in \widetilde{\G'})\,.
\end{equation*}
Using Remark \ref{RemarkWeylIntegration2}, we get:
\begin{eqnarray}
\Theta_{\Pi'_{m}}(\Psi) & = & \displaystyle\int_{\widetilde{\G'}} \Theta_{\Pi'_{m}}(\tilde{g}') \Psi(\tilde{g}') d\tilde{g}' = m_{\S} \displaystyle\int_{\check{\H}'_{\S}} \varepsilon_{\Psi'_{\S, \mathbb{R}}}(\check{h}') \Delta_{\Phi'}(\check{h}') \mathscr{H}_{\S}(\Theta_{\Pi'_{m}}\Psi)(\check{h}')d\check{h}' \nonumber \\
& = & m_{\S} \displaystyle\int_{\check{\H}'_{\S}} \varepsilon_{\Psi'_{\S, \mathbb{R}}}(\check{h}') \Delta_{\Phi'}(\check{h}') \Theta_{\Pi'_{m}}(c(\S)\check{p}(\check{h}')c(\S)^{-1})\mathscr{H}_{\S}\Psi(\check{h}')d\check{h}'\,. \label{Eqnarray1}
\end{eqnarray}
Using Theorem \ref{TheoremTheta}, Equation \eqref{EquationAppendixU(1)} can be written as:{\small
\begin{eqnarray}
& & \Theta_{\Pi'_{m}}(\Psi) \nonumber \\
& = & \displaystyle\int_{\check{\U}(1)} \overline{\Theta_{\Pi_{m}}(\check{p}(\check{h}))} \displaystyle\int_{\widetilde{\G'}} \Theta(\check{p}(\check{h})\tilde{g}')\Psi(\tilde{g}')d\tilde{g}'d\check{h} = \displaystyle\int_{\check{\U}(1)} \overline{\Theta_{\Pi_{m}}(\check{p}(\check{h}))} \det^{-\frac{k}{2}}(\check{h})\left(\det^{\frac{k}{2}}(\check{h})\displaystyle\int_{\widetilde{\G'}} \Theta(\check{p}(\check{h})\tilde{g}')\Psi(\tilde{g'})d\tilde{g}'\right)d\check{h} \nonumber \\
        & = & \displaystyle\int_{\check{\U}(1)} \overline{\Theta_{\Pi_{m}}(\check{p}(\check{h}))} \det^{-\frac{k}{2}}(\check{h}) \sum\limits_{\sigma \in \mathscr{W}(\H'_{\mathbb{C}})} \M_{\S}(\sigma)\lim\limits_{\underset{r \to 1}{r \in \E_{\sigma, \S}}} \displaystyle\int_{\check{\H}'_{\S}}\cfrac{\det^{-\frac{k}{2}}(\sigma^{-1}(\check{h}'))_{\W^{\mathfrak{h}}} \Delta_{\Phi'(\Z')}(\sigma^{-1}(\check{h}'))}{\det(1-p(\check{h})rp(\check{h}'))_{\sigma\W^{\mathfrak{h}}}} \varepsilon_{\Psi'_{\S, \mathbb{R}}}(\check{h}') \mathscr{H}_{\S}(\Psi)(\check{h}') d\check{h}' d\check{h} \nonumber\\
        & = & \displaystyle\int_{\check{\H}'_{\S}} \varepsilon_{\Psi'_{\S, \mathbb{R}}}(\check{h}') \left(\sum\limits_{\sigma \in \mathscr{W}(\H'_{\mathbb{C}})} \M_{\S}(\sigma) \Delta_{\Phi'(\Z')}(\sigma^{-1}(\check{h}'))\det^{-\frac{k}{2}}(\sigma^{-1}(\check{h}'))_{\W^{\mathfrak{h}}} \lim\limits_{\underset{r \to 1}{r \in \E_{\sigma, \S}}} \displaystyle\int_{\check{\U}(1)} \cfrac{\overline{\Theta_{\Pi_{m}}(\check{p}(\check{h}))} \det^{-\frac{k}{2}}(\check{h})}{\det(1-p(\check{h})rp(\check{h}'))_{\sigma \W^{\mathfrak{h}}}} d\check{h}\right)\mathscr{H}_{\S}(\Psi)(\check{h}') d\check{h}'\,. \label{ThetaPiM}
\end{eqnarray}}

\noindent The result follows by comparing Equations \ref{Eqnarray1} and \ref{ThetaPiM}.

\end{proof}

\noindent Without loss of generality, we assume that $p \leq q$ and keep the notations of Appendix \ref{AppendixCartanUnitary} (see Equation \eqref{DiagonalCartanSubgroupHS}). In particular, for every $h' \in \H'_{\S_{t}}$, $0 \leq t \leq p$, $h'$ is of the form
\begin{equation*}
h' = (h'_{1}, \ldots, h'_{n'}) = \diag(e^{iX_{1} - X_{p+1}}, \ldots, e^{iX_{t} - X_{p+t}}, e^{iX_{t+1}}, \ldots, e^{iX_{p}}, e^{iX_{1} + X_{p+1}}, \ldots, e^{iX_{t} + X_{p+t}}, e^{iX_{p+t+1}}, \ldots, e^{iX_{p+q}})\,, \qquad X_{j} \in \mathbb{R}\,,
\end{equation*}
where $\S_{t} = \left\{e_{1} - e_{p+1}, \ldots, e_{t} - e_{p+t}\right\}$.
\begin{prop}

The value of the character $\Theta_{\Pi'_{m}}$ is given, for every $\check{h}' \in \check{\H}'_{\S_{t}}$, $0 \leq t \leq p$, by the formula:
\begin{equation*}
\Theta_{\Pi'_{m}}(c(\S_{t})\check{p}(\check{h}')c(\S_{t})^{-1}) = 
\end{equation*}
\begin{equation*}
\pm\C^{2}\begin{cases} - \sum\limits_{\underset{j \in J(h')}{j=1}}^{t} \cfrac{h'^{-m+p-1}_{j}\left(\prod\limits_{i=1}^{n'} h'^{\frac{1}{2}}_{i}\right)}{\prod\limits_{\underset{i \neq j}{i=1}}^{n'} (h'_{j} - h'_{i})}  -  \sum\limits_{\underset{j \in K(h')}{j=1}}^{t} \cfrac{h'^{-m+p-1}_{p+j}\left(\prod\limits_{i=1}^{n'} h'^{\frac{1}{2}}_{i}\right)}{\prod\limits_{\underset{i \neq p+j}{i=1}}^{n'} (h'_{p+j} - h'_{i})} - \sum\limits_{i=t+1}^{p} \cfrac{h'^{-m+p-1}_{i}\left(\prod\limits_{j=1}^{n'} h'^{\frac{1}{2}}_{j}\right)}{\prod\limits_{\underset{j \neq i}{j=1}}^{n'} (h'_{i} - h'_{j})} & \text{ if } m \leq - 1 -\frac{q-p}{2} \\
\sum\limits_{\underset{j \in K(h')}{j=1}}^{t} \cfrac{h'^{-m+p-1}_{j}\left(\prod\limits_{i=1}^{n'} h'^{\frac{1}{2}}_{i}\right)}{\prod\limits_{\underset{i \neq j}{i=1}}^{n'} (h'_{j} - h'_{i})} +  \sum\limits_{\underset{j \in K(h')}{j=1}}^{t} \cfrac{h'^{-m+p-1}_{p+j}\left(\prod\limits_{i=1}^{n'} h'^{\frac{1}{2}}_{i}\right)}{\prod\limits_{\underset{i \neq p+j}{i=1}}^{n'} (h'_{p+j} - h'_{i})} + \sum\limits_{j=p+t+1}^{p+q} \cfrac{h'^{-m+p-1}_{j}\left(\prod\limits_{i=1}^{n'} h'^{\frac{1}{2}}_{i}\right)}{\prod\limits_{\underset{j \neq i}{i=1}}^{n'} (h'_{j} - h'_{i})} & \text{ otherwise } \end{cases}
\end{equation*}
where $\C = \frac{1}{(p+q-1)!}$, $h' = (h'_{1}, \ldots, h'_{n'})$ and $K(h')$, $J(h')$ are given by:
\begin{equation*}
J(h') = \{j \in \{1, \ldots, t\}, \sgn(X_{p+j}) = 1\}\,, \qquad K(h') = \{j \in \{1, \ldots, t\}, \sgn(X_{p+j}) = -1\}\,,
\end{equation*}

\label{PropositionA3}

\end{prop}

\noindent To make the equation shorter, we will denote by $\C$ the constant $\C = \frac{(-1)^{u}}{(p+q-1)!}$.

\begin{proof}

We start by determining the space $\E_{\sigma, \emptyset}$ for $\sigma \in \mathscr{S}_{n'}$. For every $w = w_{1, 1}\E_{1, 1}$ and $y' = (y'_{1}, \ldots, y'_{n'}) \in \mathfrak{h}'$ with $y'_{j} = iX'_{j}, X'_{j} \in \mathbb{R}$, we get:
\begin{eqnarray*}
& & \langle y'\sigma(w), \sigma(w)\rangle=\langle y'(w_{1, 1}\E_{\sigma(1), 1}), w_{1, 1}\E_{\sigma(1), 1}\rangle =
\langle (w_{1, 1}y'_{\sigma(1)}\E_{\sigma(1), 1}), w_{1, 1}\E_{\sigma(1), 1}\rangle = \Im(\overline{w_{1, 1}} \E_{1, \sigma(1)}\Id_{p, q} w_{1, 1}y'_{\sigma(1)}\E_{\sigma(1), 1}) \\
& = & \begin{cases} |w_{1, 1}|^{2} \Im(y'_{\sigma(1)} \E_{1, \sigma(1)} \E_{\sigma(1), 1}) & \text{ if } 1 \leq \sigma(1) \leq p \\ -|w_{1, 1}|^{2} \Im(y'_{\sigma(1)} \E_{1, \sigma(1)} \E_{\sigma(1), 1}) & \text{ if } p+1 \leq \sigma(1) \leq p+q \end{cases}= \begin{cases} X'_{\sigma(1)}|w_{1, 1}|^{2}  & \text{ if } 1 \leq \sigma(1) \leq p \\ -X'_{\sigma(1)}|w_{1, 1}|^{2} & \text{ if } p+1 \leq \sigma(1) \leq p+q \end{cases}
\end{eqnarray*}
In particular,
\begin{equation*}
\Gamma_{\sigma,\emptyset} = \begin{cases}
\left\{y' \in \mathfrak{h}', X'_{\sigma(1)} > 0\right\} & \text{ if } 1 \leq \sigma(1) \leq p \\
\left\{y' \in \mathfrak{h}', X'_{\sigma(1)} < 0\right\} & \text{ if } p+1 \leq \sigma(1) \leq n'
\end{cases}
\end{equation*}
and then
\begin{equation*}
\E_{\sigma,\emptyset} = \exp(i\Gamma_{\sigma,\emptyset}) = \begin{cases} 
\left\{h' \in \H'_{\mathbb{C}}, h' = (e^{-X'_{1}}, \ldots, e^{-X'_{n'}}), X'_{\sigma(1)} > 0 \right\} & \text{ if } 1 \leq \sigma(1) \leq p \\
\left\{h' \in \H'_{\mathbb{C}}, h' = (e^{-X'_{1}}, \ldots, e^{-X'_{n'}}), X'_{\sigma(1)} < 0 \right\} & \text{ if } p+1 < \sigma(1) \leq n'
\end{cases}
\end{equation*}
More generally, for every $\sigma \in \mathscr{S}_{n}$, we get:
\begin{equation*}
\Gamma_{\sigma, \S_{t}} = \begin{cases} \mathfrak{h}' & \text{ if } \sigma(1) \in \underline{\S_{t}} \\ \left\{y \in \mathfrak{h}', X'_{\sigma(1)} > 0\right\} & \text{ if } \sigma(1) \in \left\{t+1, \ldots, p\right\} \\ \left\{y \in \mathfrak{h}', X'_{\sigma(1)} < 0\right\} & \text{ if } \sigma(1) \in \left\{p+t+1, \ldots, n\right\} \end{cases}
\end{equation*}
In particular, 
\begin{equation*}
\E_{\sigma, \S} = \exp  \begin{cases} \left\{h' \in \H'_{\mathbb{C}}, h = \diag(e^{-X'_{1}}, \ldots, e^{-X'_{n}}), X'_{i} \in \mathbb{R} \right\} & \text{ if }  \sigma(1) \in \underline{\S_{t}} \\ \left\{h' \in \H'_{\mathbb{C}}, h' = \diag(e^{-X'_{1}}, \ldots, e^{-X'_{n}}), X'_{i} \in \mathbb{R}, X'_{\sigma(1)} > 0 \right\}  & \text{ if } \sigma(1) \in \left\{t+1, \ldots, p\right\} \\ \left\{h' \in \H'_{\mathbb{C}}, h = \diag(e^{-X'_{1}}, \ldots, e^{-X'_{n}}), X'_{i} \in \mathbb{R}, X'_{\sigma(1)} < 0 \right\} & \text{ if } \sigma(1) \in \left\{p+t+1, \ldots, n\right\} \end{cases}
\end{equation*}
Because the space $\E_{\sigma, \S_{t}}$ only depends on $\sigma(1)$, we will denote this space by $\E_{i, \S_{t}}$ for a $\sigma \in \mathscr{S}_{n'}$ such that $\sigma(1) = i$. 

\noindent We first assume that $n'$ is even, i.e. $k = 0$. Then, according to Proposition \ref{Proposition15Juillet} and that $\det(1-p(\check{h})rp(\check{h}'))_{\sigma \W^{\mathfrak{h}}} = 1 - h(rh')^{-1}_{\sigma(1)}$, we get (up to a constant):
\begin{eqnarray*}
& & \Theta_{\Pi'}(c(\S_{t})p(\check{h}')c(\S_{t})^{-1}) = \C\sum\limits_{\sigma \in \mathscr{S}_{n'}}  \varepsilon(\sigma) \cfrac{\Delta_{\Phi'(\Z')}(\sigma^{-1}(\check{h}'))}{\Delta_{\Phi'}(\check{h}')} \lim\limits_{\underset{r \in \E_{\sigma, \S}}{r \to 1}} \displaystyle\int_{\check{\U}(1)} \cfrac{\overline{\Theta_{\Pi_{m}}(\check{p}(\check{h}))}}{1- h(rh')^{-1}_{\sigma(1)}} d\check{h} \\
 & = & \C^{2} \sum\limits_{j=1}^{t} \cfrac{h'^{\frac{n'-2}{2}}_{j}\left(\prod\limits_{i=1}^{n'} h'^{\frac{1}{2}}_{i}\right)}{\prod\limits_{\underset{i \neq j}{i=1}}^{n'} (h'_{j} - h'_{i})} \displaystyle\int_{\check{\U}(1)} \cfrac{\overline{\Theta_{\Pi_{m}}(\check{p}(\check{h}))}}{1 - he^{-iX_{1} + X_{p+1}}} d\check{h} + \C^{2} \sum\limits_{j=1}^{t} \cfrac{h'^{\frac{n'-2}{2}}_{p+j}\left(\prod\limits_{i=1}^{n'} h'^{\frac{1}{2}}_{i}\right)}{\prod\limits_{\underset{i \neq p+j}{i=1}}^{n'} (h'_{p+j} - h'_{i})} \displaystyle\int_{\check{\U}(1)} \cfrac{\overline{\Theta_{\Pi_{m}}(\check{p}(\check{h}))}}{1 - he^{-iX_{j} - X_{p+j}}} d\check{h} \\
 & + & \C^{2} \sum\limits_{j=t+1}^{p} \cfrac{h'^{\frac{n'-2}{2}}_{j}\left(\prod\limits_{i=1}^{n'} h'^{\frac{1}{2}}_{i}\right)}{\prod\limits_{\underset{j \neq i}{i=1}}^{n'} (h'_{j} - h'_{i})}  \lim\limits_{\underset{0 < r < 1}{r \to 1}}  \displaystyle\int_{\check{\U}(1)} \cfrac{\overline{\Theta_{\Pi_{m}}(\check{p}(\check{h}))}}{1 - h(rh'_{j})^{-1}} d\check{h} + \C^{2} \sum\limits_{j=p+t+1}^{n'}  \cfrac{h'^{\frac{n'-2}{2}}_{j}\left(\prod\limits_{i=1}^{n'} h'^{\frac{1}{2}}_{i}\right)}{\prod\limits_{\underset{j \neq i}{i=1}}^{n'} (h'_{j} - h'_{i})}  \lim\limits_{\underset{r > 1}{r \to 1}} \displaystyle\int_{\check{\U}(1)} \cfrac{\overline{\Theta_{\Pi_{m}}(\check{p}(\check{h}))}}{1 - h(rh'_{i})^{-1}} d\check{h} \\
 & = & - \C^{2} \sum\limits_{j=1}^{t} \cfrac{h'^{\frac{n'-2}{2}}_{j}\left(\prod\limits_{i=1}^{n'} h'^{\frac{1}{2}}_{i}\right)}{\prod\limits_{\underset{i \neq j}{i=1}}^{n'} (h'_{j} - h'_{i})} \cfrac{1}{e^{-iX_{j} + X_{p+j}}}\displaystyle\int_{\check{\U}(1)} \cfrac{\overline{\Theta_{\Pi_{m}}(\check{p}(\check{h}))}}{h - e^{iX_{j} - X_{p+j}}} d\check{h} - \C^{2} \sum\limits_{j=1}^{t} \cfrac{h'^{\frac{n'-2}{2}}_{p+j}\left(\prod\limits_{i=1}^{n'} h'^{\frac{1}{2}}_{i}\right)}{\prod\limits_{\underset{i \neq p+j}{i=1}}^{n'} (h'_{p+j} - h'_{i})} \cfrac{1}{e^{-iX_{j} - X_{p+j}}} \displaystyle\int_{\check{\U}(1)} \cfrac{\overline{\Theta_{\Pi_{m}}(\check{p}(\check{h}))}}{h - e^{iX_{j} + X_{p+j}}} d\check{h} \\
 & - & \C^{2} \sum\limits_{j=t+1}^{p} \cfrac{h'^{\frac{n'-2}{2}}_{j}\left(\prod\limits_{i=1}^{n'} h'^{\frac{1}{2}}_{i}\right)}{\prod\limits_{\underset{j \neq i}{i=1}}^{n'} (h'_{j} - h'_{i})}  \lim\limits_{\underset{0 < r < 1}{r \to 1}} h'_{i} \displaystyle\int_{\check{\U}(1)} \cfrac{\overline{\Theta_{\Pi_{m}}(\check{p}(\check{h}))}}{h - rh'_{j}} d\check{h} - \C^{2} \sum\limits_{j=p+t+1}^{n'}  \cfrac{h'^{\frac{n'-2}{2}}_{j}\left(\prod\limits_{i=1}^{n'} h'^{\frac{1}{2}}_{i}\right)}{\prod\limits_{\underset{j \neq i}{i=1}}^{n'} (h'_{j} - h'_{i})}  \lim\limits_{\underset{r > 1}{r \to 1}} h'_{i} \displaystyle\int_{\check{\U}(1)} \cfrac{\overline{\Theta_{\Pi_{m}}(\check{p}(\check{h}))}}{h - rh'_{j}} d\check{h} \\
& = & - \frac{\C^{2}}{2i\pi}\sum\limits_{j=1}^{t} \cfrac{h'^{\frac{n'-2}{2}}_{j}\prod\limits_{i=1}^{n'} h'^{\frac{1}{2}}_{i}}{\prod\limits_{\underset{i \neq j}{i=1}}^{n'} (h'_{j} - h'_{i})} \cfrac{1}{e^{-iX_{j} + X_{p+j}}}\displaystyle\int_{\U(1)} \cfrac{h^{-m-1-\frac{q-p}{2}}}{h - e^{iX_{j} - X_{p+j}}} dh -  \frac{\C^{2}}{2i\pi}\sum\limits_{j=1}^{t} \cfrac{h'^{\frac{n'-2}{2}}_{p+j}\prod\limits_{i=1}^{n'} h'^{\frac{1}{2}}_{i}}{\prod\limits_{\underset{i \neq p+j}{i=1}}^{n'} (h'_{p+j} - h'_{i})} \cfrac{1}{e^{-iX_{j} - X_{p+j}}} \displaystyle\int_{\U(1)} \cfrac{h^{-m-1-\frac{q-p}{2}}}{h - e^{iX_{j} + X_{p+j}}} dh \\
 & - & \frac{\C^{2}}{2i\pi}\sum\limits_{j=t+1}^{p} \cfrac{h'^{\frac{n'-2}{2}}_{j}\prod\limits_{i=1}^{n'} h'^{\frac{1}{2}}_{i}}{\prod\limits_{\underset{j \neq i}{i=1}}^{n'} (h'_{j} - h'_{j})}  \lim\limits_{\underset{0 < r < 1}{r \to 1}} h'_{i} \displaystyle\int_{\U(1)} \cfrac{h^{-m-1-\frac{q-p}{2}}}{h - rh'_{j}} dh - \frac{\C^{2}}{2i\pi} \sum\limits_{j=p+t+1}^{n'}  \cfrac{h'^{\frac{n'-2}{2}}_{j}\prod\limits_{i=1}^{n'} h'^{\frac{1}{2}}_{i}}{\prod\limits_{\underset{j \neq i}{i=1}}^{n'} (h'_{j} - h'_{i})}  \lim\limits_{\underset{r > 1}{r \to 1}} h'_{i} \displaystyle\int_{\U(1)} \cfrac{h^{-m-1-\frac{q-p}{2}}}{h - rh'_{j}} dh 
\end{eqnarray*}
If $-m-1-\frac{q-p}{2} \geq 0$, i.e. $m \leq -1-\frac{q-p}{2}$. Then, according to Lemma \ref{LemmaComplexIntegrals}, we get:
\begin{eqnarray*}
& & \Theta_{\Pi'_{m}}(c(\S_{t})p(\check{h}')c(\S_{t})^{-1}) \\ 
& = & - \frac{\C^{2}}{2i\pi} \sum\limits_{\underset{j \in J(h')}{j=1}}^{t} \cfrac{h'^{\frac{n'-2}{2}}_{j}\prod\limits_{i=1}^{n'} h'^{\frac{1}{2}}_{i}}{\prod\limits_{\underset{i \neq j}{i=1}}^{n'} (h'_{j} - h'_{i})} \cfrac{1}{e^{-iX_{j} + X_{p+j}}}\displaystyle\int_{\U(1)} \cfrac{h^{-m-1-\frac{q-p}{2}}}{h - e^{iX_{j} - X_{p+j}}} dh - \frac{\C^{2}}{2i\pi} \sum\limits_{\underset{j \in K(h')}{j=1}}^{t} \cfrac{h'^{\frac{n'-2}{2}}_{p+j}\prod\limits_{i=1}^{n'} h'^{\frac{1}{2}}_{i}}{\prod\limits_{\underset{i \neq p+j}{i=1}}^{n'} (h'_{p+j} - h'_{i})} \cfrac{1}{e^{-iX_{j} - X_{p+j}}} \displaystyle\int_{\U(1)} \cfrac{h^{-m-1-\frac{q-p}{2}}}{h - e^{iX_{j} + X_{p+j}}} dh \\
 & - & \frac{\C^{2}}{2i\pi} \sum\limits_{j=t+1}^{p} \cfrac{h'^{\frac{n'-2}{2}}_{j}\prod\limits_{i=1}^{n'} h'^{\frac{1}{2}}_{i}}{\prod\limits_{\underset{j \neq i}{i=1}}^{n'} (h'_{j} - h'_{i})}  \lim\limits_{\underset{0 < r < 1}{r \to 1}} h'_{i} \displaystyle\int_{\U(1)} \cfrac{h^{-m-1-\frac{q-p}{2}}}{h - rh'_{i}} dh \\
  & = & - \C^{2} \sum\limits_{\underset{j \in J(h')}{j=1}}^{t} \cfrac{h'^{-m+p-1}_{j}\prod\limits_{i=1}^{n'} h'^{\frac{1}{2}}_{i}}{\prod\limits_{\underset{i \neq j}{i=1}}^{n'} (h'_{j} - h'_{i})} - \C^{2} \sum\limits_{\underset{j \in K(h')}{j=1}}^{t} \cfrac{h'^{-m+p-1}_{p+j}\prod\limits_{i=1}^{n'} h'^{\frac{1}{2}}_{i}}{\prod\limits_{\underset{i \neq p+j}{i=1}}^{n'} (h'_{p+j} - h'_{i})} - \C^{2} \sum\limits_{j=t+1}^{p} \cfrac{h'^{-m+p-1}_{j}\prod\limits_{i=1}^{n'} h'^{\frac{1}{2}}_{i}}{\prod\limits_{\underset{j \neq i}{i=1}}^{n'} (h'_{j} - h'_{i})} 
\end{eqnarray*}
Similarly, if $m > -1-\frac{q-p}{2}$, we get:
\begin{eqnarray*}
& & \Theta_{\Pi'}(c(\S_{t})p(\check{h}')c(\S_{t})^{-1}) \\ 
& = & - \frac{\C^{2}}{2i\pi} \sum\limits_{\underset{j \in K(h')}{j=1}}^{t} \cfrac{h'^{\frac{n'-2}{2}}_{j}\prod\limits_{i=1}^{n'} h'^{\frac{1}{2}}_{i}}{\prod\limits_{\underset{i \neq j}{i=1}}^{n'} (h'_{j} - h'_{i})} \cfrac{1}{e^{-iX_{j} + X_{p+j}}}\displaystyle\int_{\U(1)} \cfrac{h^{-m-1-\frac{q-p}{2}}}{h - e^{iX_{j} - X_{p+j}}} dh -  \frac{\C^{2}}{2i\pi} \sum\limits_{\underset{j \in J(h')}{j=1}}^{t} \cfrac{h'^{\frac{n'-2}{2}}_{p+j}\prod\limits_{i=1}^{n'} h'^{\frac{1}{2}}_{i}}{\prod\limits_{\underset{i \neq p+j}{i=1}}^{n'} (h'_{p+j} - h'_{i})} \cfrac{1}{e^{-iX_{j} - X_{p+j}}} \displaystyle\int_{\U(1)} \cfrac{h^{-m-1-\frac{q-p}{2}}}{h - e^{iX_{j} + X_{p+j}}} dh \\
 & - & \frac{\C^{2}}{2i\pi} \sum\limits_{j=p+t+1}^{n'} \cfrac{h'^{\frac{n'-2}{2}}_{j}\prod\limits_{i=1}^{n'} h'^{\frac{1}{2}}_{i}}{\prod\limits_{\underset{j \neq i}{i=1}}^{n'} (h'_{j} - h'_{i})}  \lim\limits_{\underset{r > 1}{r \to 1}} h'_{i} \displaystyle\int_{\U(1)} \cfrac{h^{-m-1-\frac{q-p}{2}}}{h - rh'_{j}} dh \\
 & = & \C^{2} \sum\limits_{\underset{j \in J(h')}{j=1}}^{t} \cfrac{h'^{-m+p-1}_{j}\prod\limits_{i=1}^{n'} h'^{\frac{1}{2}}_{i}}{\prod\limits_{\underset{i \neq j}{i=1}}^{n'} (h'_{j} - h'_{i})} + \C^{2} \sum\limits_{\underset{j \in K(h')}{j=1}}^{t} \cfrac{h'^{-m+p-1}_{p+j}\prod\limits_{i=1}^{n'} h'^{\frac{1}{2}}_{i}}{\prod\limits_{\underset{i \neq p+j}{i=1}}^{n'} (h'_{p+j} - h'_{i})} + \C^{2} \sum\limits_{j=p+t+1}^{n'} \cfrac{h'^{-m+p-1}_{j}\prod\limits_{i=1}^{n'} h'^{\frac{1}{2}}_{i}}{\prod\limits_{\underset{j \neq i}{i=1}}^{n'} (h'_{j} - h'_{i})} 
\end{eqnarray*}
The computations are similar if $n'$ is odd.

\end{proof}

\begin{coro}

The value of $\Theta_{\Pi'_{m}}$ on $\widetilde{\H}' = \check{p}(\check{\H}'_{\emptyset})$ is given, up to a constant, by:
\begin{equation*}
\Theta_{\Pi'_{m}}(\check{p}(\check{h}')) = \pm\C^{2} \begin{cases} \prod\limits_{i=1}^{n'} {h'_{i}}^{\frac{1}{2}} \sum\limits_{i=1}^{p}\cfrac{{h'_{i}}^{-m+p-1}}{\prod\limits_{j \neq i} (h'_{i} - h'_{j})} & \text{ if } m \leq -1-\frac{q-p}{2} \\ -\prod\limits_{i=1}^{n'} {h}'^{\frac{1}{2}}_{i} \sum\limits_{i=p+1}^{n'} \cfrac{{h'_{i}}^{-m+p-1}}{\prod\limits_{j \neq i} (h'_{i} - h'_{j})} & \text{ otherwise } \end{cases}
\end{equation*}
where $\C \in \mathbb{R}$.
\label{FinalCorollaryU(1)}

\end{coro}

\noindent This result was obtained in \cite[Section~6]{MER}.

\begin{rema}

Assume that $p = 1$, $q = 1$. Then,
\begin{equation*}
\Theta_{\Pi'_{m}}(\tilde{h}') = \pm \begin{cases}  \cfrac{(e^{i\theta-X})^{-m}}{e^{X} - e^{-X}} & \text{ if } m \leq -1 \text{ and } X > 0 \\  -\cfrac{(e^{i\theta+X})^{-m}}{e^{X} - e^{-X}} & \text{ if } m \leq -1 \text{ and } X < 0 \\ \cfrac{(e^{i\theta+X})^{-m}}{e^{X} - e^{-X}} & \text{ if } m \geq 0 \text{ and } X > 0 \\ -\cfrac{(e^{i\theta-X})^{-m}}{e^{X} - e^{-X}} & \text{ if } m \geq 0 \text{ and } X < 0 \end{cases}\,,
\end{equation*}
where $h' = \begin{pmatrix} e^{i\theta}\ch(X) & \sh(X) \\ \sh(X) & e^{i\theta}\ch(X)\end{pmatrix}$.
We recover the results of \cite[Section~7]{MER}.

\label{LastRemarkU(1,1)}

\end{rema}

\section{Cartan subgroups for unitary groups}

\label{AppendixCartanUnitary}

It is well-known that the number of non-conjugated Cartan subgroups of $\G = \U(p, q)$, up to equivalence, is $\min(p, q) +1$ (see \cite{HIRAI}). We recall in this appendix how Cartan subgroups can be parametrised using strongly orthogonal roots (see \cite[Section~2]{SCH}).

\noindent Let $\K = \U(p) \times \U(q)$ be the maximal compact subgroup of $\G$ and $\H$ be the (diagonal) compact Cartan subgroup of $\K$. We denote by $\mathfrak{h}$, $\mathfrak{k}$ and $\mathfrak{g}$ the Lie algebras of $\H$, $\K$ and $\G$ respectively and $\mathfrak{h}_{\mathbb{C}}$, $\mathfrak{k}_{\mathbb{C}}$ and $\mathfrak{g}_{\mathbb{C}}$ their complexifications.

\noindent We denote by $\Delta = \Delta(\mathfrak{g}_{\mathbb{C}}, \mathfrak{h}_{\mathbb{C}})$ be the set of roots, by $\Delta_{c} := \Delta_{c}(\mathfrak{k}_{\mathbb{C}}, \mathfrak{h}_{\mathbb{C}})$ the set of compact roots and by $\Delta_{n} = \Delta \setminus \Delta_{c}$ the set of non-compact roots. Similarly, we denote by $\Psi$ a set of positive roots of $\Delta$ and let $\Psi_{c}$ and $\Psi_{n}$ the subsets of $\Psi$ given by $\Psi_{c} = \Delta_{c} \cap \Psi$ and $\Psi_{n} = \Delta_{n} \cap \Psi$. In particular, 
\begin{equation*}
\mathfrak{g}_{\mathbb{C}} = \bigoplus\limits_{\alpha \in \Delta} \mathfrak{g}_{\mathbb{C}, \alpha}\,,
\end{equation*}
where $\mathfrak{g}_{\mathbb{C}, \alpha} = \left\{X \in \mathfrak{g}_{\mathbb{C}}, [H, X] = \alpha(H)X, H \in \mathfrak{h}_{\mathbb{C}}\right\}$.

\begin{nota}

\begin{enumerate}
\item For every $\alpha \in \Delta$, we fix $X_{\alpha} \in \mathfrak{g}_{\mathbb{C}, \alpha}, Y_{\alpha} \in \mathfrak{g}_{\mathbb{C}, -\alpha}$ and $H_{\alpha} \in i\mathfrak{h}$ such that:
\begin{equation*}
[H_{\alpha}, X_{\alpha}] = 2X_{\alpha}\,, \qquad [H_{\alpha}, Y_{\alpha}] = -2Y_{\alpha}\,, \qquad [X_{\alpha}, Y_{\alpha}] = H_{\alpha}\,, \qquad \overline{H_{\alpha}} = -H_{\alpha} = H_{-\alpha}\,,
\end{equation*}
and such that $\overline{X_{\alpha}} = -Y_{\alpha}$ if $\alpha \in \Delta_{c}$ and $\overline{X_{\alpha}} = Y_{\alpha}$ if $\alpha \in \Delta_{n}$.
\item We say that $\alpha, \beta \in \Delta$ are strongly orthogonal if $\alpha \neq \pm \beta$ and $\alpha \pm \beta \notin \Delta$. We denote by $\Psi^{\st}_{n}$ a maximal family of strongly orthogonal roots of $\Psi_{n}$ (i.e. a subset of $\Psi_{n}$ such that  every pairs $\alpha, \beta \in \Psi^{\st}_{n}$ are strongly orthogonal).
\end{enumerate}

\end{nota}

\noindent For every $\alpha \in \Psi^{\st}_{n}$, we denote by $c(\alpha)$ the element of $\GL(p+q, \mathbb{C})$ given by:
\begin{equation*}
c(\alpha) = \exp\left(\frac{\pi}{4}(Y_{\alpha} - X_{\alpha})\right)\,.
\end{equation*}
For every subset $\S$ of $\Psi^{\st}_{n}$, we denote by $c(\S)$ the element of $\GL(p+q, \mathbb{C})$ defined by
\begin{equation}
c(\S) = \prod\limits_{\alpha \in \S} c(\alpha)\,,
\label{CayleyCS}
\end{equation}
and let 
\begin{equation*}
\mathfrak{h}(\S) = \mathfrak{g} \cap \Ad(c(\S))(\mathfrak{h}_{\mathbb{C}})\,.
\end{equation*}
We denote by $\H(\S)$ the analytic subgroup of $\G$ whose Lie algebra is $\mathfrak{h}(\S)$. Then, $\H(\S)$ is a Cartan subgroup of $\G$ and one can prove that all the Cartan subgroups are of this form (up to conjugation).

\noindent For every $\S \subseteq \Psi^{\st}_{n}$, we will denote by $\H_{\S}$ the subgroup of $\H_{\mathbb{C}}$ given by:
\begin{equation*}
\H_{\S} = c(\S)^{-1}\H(\S)c(\S)\,.
\end{equation*}
where $\H_{\mathbb{C}} = \left\{\diag(\lambda_{1}, \ldots, \lambda_{p+q}), \lambda_{i} \in \mathbb{C}\right\}$.

\noindent Without loss of generality, we assume that $p \leq q$. The set of roots $\Delta$ is given by $\Delta = \left\{\pm(e_{i} - e_{j}), 1 \leq i < j \leq p+q\right\}$, where $e_{i}$ is the linear form on $\mathfrak{h}_{\mathbb{C}} = \mathbb{C}^{p+q}$ given by $e_{i}(\lambda_{1}, \ldots, \lambda_{p+q}) = \lambda_{i}$. In this case, 
\begin{equation*}
\Delta_{c} = \left\{\pm(e_{i} - e_{j}), 1 \leq i < j \leq p\right\} \cup \left\{\pm(e_{i} - e_{j}), p+1 \leq i < j \leq p+q\right\}\,, \qquad \Delta_{n} = \left\{\pm(e_{i} - e_{j}), 1 \leq i \leq p, p+1 \leq j \leq p+q\right\}\,,
\end{equation*}
and the set $\Psi^{\st}_{n}$ can be chosen as $\{e_{t} - e_{p+t}, 1 \leq t \leq p\}$. In particular, $\H(\emptyset) = \H$ and if $\S_{t} = \{e_{1} - e_{p+1}, \ldots, e_{t} - e_{p+t}\}, 1 \leq t \leq p$, we get:
\begin{equation*}
\H(\S_{t}) = \exp\left(\bigoplus\limits_{j=t+1}^{p} i\mathbb{R}\E_{j, j} \oplus \bigoplus\limits_{j=p+t+1}^{p+q} i\mathbb{R}\E_{j, j} \oplus \bigoplus\limits_{j=1}^{t} i\mathbb{R}(\E_{j, j} + \E_{p+j, p+j}) \oplus \bigoplus\limits_{j=1}^{t}\mathbb{R}(\E_{j, p+j} + \E_{p+j, j})\right)\,,
\end{equation*}
and
\begin{equation}
\H_{\S_{t}} = \left\{h = \diag(e^{iX_{1} - X_{p+1}}, \ldots, e^{iX_{t} - X_{p+t}}, e^{iX_{t+1}}, \ldots, e^{iX_{p}}, e^{iX_{1} + X_{p+1}}, \ldots, e^{iX_{t} + X_{p+t}}, e^{iX_{p+t+1}}, \ldots, e^{iX_{2p+1}}), X_{j} \in \mathbb{R}\right\}\,.
\label{DiagonalCartanSubgroupHS}
\end{equation}

\begin{rema}

As explained in \cite[Proposition~2.16]{SCH}, two Cartan subalgebras $\mathfrak{h}(\S_{1})$ and $\mathfrak{h}(\S_{2})$, with $\S_{1}, \S_{2} \subseteq \Psi^{\st}_{n}$, are conjugate if and only if there exists an element of $\sigma \in \mathscr{W}$ sending $\S_{1} \cup (-\S_{1})$ onto $\S_{2} \cup (-\S_{2})$.

\label{RemarkAppendixB} 

\end{rema}

\section{The character $\varepsilon$}

\label{AppendixEpsilon}

Let $(\W, \langle \cdot, \cdot\rangle)$ be a real symplectic space, $\Sp(\W)$ the corresponding group of isometries and $\widetilde{\Sp}(\W)$ its metaplectic cover as in Equation \eqref{MetaplecticGroup}.

\noindent Let $\W = \X \oplus \Y$ be a complete polarization of $\W$. We denote by $\Z$ the subgroup of $\Sp(\W)$ preserving both $\X$ and $\Y$. In particular, we get that 
\begin{equation*}
\Z = \left\{\begin{pmatrix} g & 0 \\ 0 & (g^{-1})^{t} \end{pmatrix}, g \in \GL(\X)\right\} \approx \GL(\X)\,.
\end{equation*}
We define a double cover $\widetilde{\GL}(\X)$ of $\GL(\X)$ by
\begin{equation*}
\widetilde{\GL}(\X) = \left\{(g, \eta) \in \GL(\X) \times \mathbb{C}^{\times}\,, \eta^{2} = \det(g)\right\}, \qquad \widetilde{\GL}(\X) \ni (g, \eta) \to g \in \GL(\X)\,.
\end{equation*}
As recalled in \cite[Section~6]{TOM2}, the restriction map
\begin{equation*}
\Z \ni g \to g_{|_{\X}} \in \GL(\X)
\end{equation*}
lifts to a group isomorphism
\begin{equation*}
\widetilde{\Z} \ni \tilde{g} \to (g_{|_{\X}}, \eta) \in \widetilde{\GL}(\X)\,,
\end{equation*}
where $\eta = \eta(\tilde{g})$ is defined on $\Z^{c} = \left\{\tilde{g} \in \widetilde{\Z}, \det(g-1)_{\W} \neq 0\right\}$ by the following formula:
\begin{equation*}
\eta(\tilde{g}) = \cfrac{\Theta(\tilde{g})}{|\Theta(\tilde{g})|} |\det(g_{|_{\X}})|^{\frac{1}{2}}\,.
\end{equation*}
We denote by $\varepsilon$ the function on $\widetilde{\Z}$ given by:
\begin{equation*}
\varepsilon: \widetilde{\Z} \ni \tilde{g} \to \varepsilon(\tilde{g}) = \cfrac{\eta(\tilde{g})}{|\eta(\tilde{g})|} \in \mathbb{C}\,.
\end{equation*}
One can easily prove that $\varepsilon$ is a character of $\widetilde{\Z}$ with values in the set $\left\{\pm 1, \pm i\right\}$ such that
\begin{equation}
\varepsilon(\tilde{g}) = \cfrac{\Theta(\tilde{g})}{|\Theta(\tilde{g})|}, \qquad \left(\tilde{g} \in \widetilde{\Z^{c}}\right).
\label{LemmaAppendixC1}
\end{equation}

\begin{lemme}

For every $\tilde{g} \in \widetilde{\Z^{c}}$, we get:
\begin{equation*}
\Theta(\tilde{g})^{2} = \det(g_{|_{X}})^{-1} \det\left(\frac{1}{2}(c(g_{|_{X}}) + 1)\right)^{2}.
\end{equation*}

\end{lemme}

\begin{proof}

Then, for every $\tilde{g} \in \widetilde{\Z^{c}}$, we get:
\begin{eqnarray*}
\Theta(\tilde{g})^{2} & = & \det(i(g-1))^{-1} = (-1)^{\frac{\dim_{\mathbb{R}}(W)}{2}} \det(g_{|_{X}} - 1)^{-1} \det(g_{|_{Y}} - 1)^{-1} \\
                                & = & \det(g_{|_{X}} - 1)^{-1} \det(1 - g_{|_{Y}})^{-1} = \det(g_{|_{X}} - 1)^{-1} \det(1 - g^{-1}_{|_{X}})^{-1} \\
                                & = & \det(g_{|_{X}} - 1)^{-1} \det(1 - g^{-1}_{|_{X}})^{-1} = \det(g_{|_{X}}) \det(g_{|_{X}} - 1)^{-2}
\end{eqnarray*}
We have:
\begin{equation*}
\frac{1}{2}\left(c(g_{|_{X}}) + 1\right) = \frac{1}{2}\left((g_{|_{X}} + 1)(g_{|_{X}} - 1)^{-1} + 1\right) = \frac{1}{2}\left((g_{|_{X}} + 1)(g_{|_{X}} - 1)^{-1} + (g_{|_{X}} - 1)(g_{|_{X}} - 1)^{-1}\right) = g_{|_{X}}(g_{|_{X}} - 1)^{-1}\,.
\end{equation*}
Then,
\begin{equation*}
\Theta(\tilde{g})^{2} = \det(g_{|_{X}}) \det(g_{|_{X}} - 1)^{-2} = \left(\det(g_{|_{X}}) \det(g_{|_{X}} - 1)^{-1}\right)^{2} \det(g_{|_{X}})^{-1} = \det(g_{|_{X}})^{-1} \det\left(\frac{1}{2}(c(g_{|_{X}}) + 1)\right)^{2}\,.
\end{equation*}

\end{proof}

\noindent We define by $\det^{-\frac{1}{2}}_{\X}(\tilde{g})$ the following quantity: 
\begin{equation*}
\det^{-\frac{1}{2}}_{\X}(\tilde{g}) = \Theta(\tilde{g}) \left| \det\left(\frac{1}{2}(c(g_{|_{X}}) + 1)\right) \right|^{-1}\,.
\end{equation*}

\noindent In particular, for every $\tilde{g} \in \widetilde{\Z^{c}}$, we get:
\begin{equation}
\varepsilon(\tilde{g}) = \cfrac{\Theta(\tilde{g})}{|\Theta(\tilde{g})|} = \cfrac{\det^{-\frac{1}{2}}_{\X}(\tilde{g})}{|\det^{-\frac{1}{2}}_{\X}(\tilde{g})|} \qquad \left(\tilde{g} \in \widetilde{\Z^{c}}\right).
\label{LemmaAppendixC2}
\end{equation}

%\bibliographystyle{plain}
%\bibliography{BiblioDraft}

\bibliographystyle{plain}

\end{document}